\documentclass[12pt]{article}
\usepackage[margin=3cm]{geometry}
\usepackage[pdftex]{graphicx}
\graphicspath{ {./images/} }
\usepackage{float}
\usepackage{booktabs,caption}
\usepackage[utf8]{inputenc}
\usepackage{parskip}
\usepackage[breaklinks]{hyperref}
\hypersetup{pdfauthor={Name}}
\usepackage{amssymb}
\usepackage{amsmath}
\usepackage[extra]{tipa}
\usepackage[mathscr]{euscript}
\usepackage[shortlabels]{enumitem}
\usepackage{xcolor}
\usepackage{csquotes}
\pagecolor{white}
\usepackage[all]{xy}
\usepackage{authblk}
\usepackage{theoremref}
\usepackage{tabularx}
\numberwithin{figure}{section}
\usepackage{physics}

\usepackage{tikz}
\tikzset{ext/.style={circle, draw,inner sep=1pt},int/.style={circle,draw,fill,inner sep=1pt},nil/.style={inner sep=1pt}}
\usepackage[toc,page]{appendix}
\usepackage{dsfont}
\usetikzlibrary{arrows}
\usetikzlibrary{knots}
\usetikzlibrary{decorations.markings, arrows.meta,knots}
\usepackage{comment}
\usepackage{array}
\usepackage{url}
\usepackage{lipsum}
\usepackage{amscd}
\usepackage{amsthm}
\usepackage{tikz-cd}
\usepackage{mathtools}
\numberwithin{equation}{section}
\newtheorem{thm}{Theorem}[section]
\newtheorem{lemma}[thm]{Lemma}
\newtheorem{proposition}[thm]{Proposition}
\newtheorem{cor}[thm]{Corollary}

\theoremstyle{definition}

\newtheorem{remark}[thm]{Remark}
\newtheorem{example}[thm]{Example}
\newcommand{\bbR}{\mathbb{R}}
\newcommand{\bbQ}{\mathbb{Q}}
\newcommand{\bbZ}{\mathbb{Z}}

\usepackage[backend=bibtex,style=numeric]{biblatex}
\title{\normalsize \textbf{GOUSSAROV-POLYAK-VIRO TYPE FORMULAS FOR $(4k-1)$-DIMENSIONAL KNOTS AND LINKS IN $\bbR^{6k}$}}

\date{}
 \author{\small NEETI GAUNIYAL AND VICTOR TURCHIN}
\begin{document}

\maketitle 

\begin{abstract}
\noindent
We produce combinatorial formulas for invariants of smooth embeddings of $(2\ell-1)$-spheres into $\bbR^{3\ell}$ for $\ell\geq 2$.  Furthermore, we obtain such a formula for the Haefliger invariant, which classifies smooth knots $S^{4k-1}\hookrightarrow \bbR^{6k}$ up to isotopy. Our approach is similar in spirit to  the work of Goussarov, Polyak, and Viro 
expressing  finite-type invariants of classical knots in terms of Gauss diagrams. We similarly  project   higher dimensional knots and links onto a hyperplane and study the preimages of the sets of double and singular points in the embedded spheres. As an auxiliary result, 
we show 
that the space of $n$-dimensional braids with $k$ strands in $\bbR^{n+q}$ is a homotopy retract of the space of long links $\underset{k}{\sqcup}\bbR^n\hookrightarrow\bbR^{n+q}$ for $q\geq 3$, thus proving  a conjecture of Komendarczyk, Koytcheff and Voli\'c.


\end{abstract}

\begin{section}{Introduction}\label{s:intro}
A \textit{high-dimensional knot} is an embedding $S^n\hookrightarrow \bbR^{n+q}$. A \textit{spherical knot} refers to an embedding of $S^n\hookrightarrow S^{n+q}$, while a \textit{long knot} is an embedding $\bbR^n\hookrightarrow \bbR^{n+q}$ that agrees with the standard inclusion outside a compact set. The spaces of such embeddings are denoted by $Emb(S^n, \bbR^{n+q}),Emb(S^n, S^{n+q})$,  and $Emb_{\partial}(\bbR^n, \bbR^{n+q})$, respectively. A \textit{high-dimensional link} with $m$ components is an embedding $\underset{i=1}{\overset{m}{\sqcup}} S^n\hookrightarrow \bbR^{n+q}$, where all components have the same dimension. 
It is known that for $q\geq 3$ and $m\geq 1$, 
\begin{equation}
\pi_0Emb(\underset{i=1}{\overset{m}{\sqcup}}S^n,\bbR^{n+q})=\pi_0Emb(\underset{i=1}{\overset{m}{\sqcup}}S^n,S^{n+q})=\pi_0Emb_{\partial}(\underset{i=1}{\overset{m}{\sqcup}}\bbR^n,\bbR^{n+q}),
\end{equation}
where each term forms a finitely generated abelian group under component-wise connected sum, see \cite{HAE4}, \cite[Lemma~3.7]{SONG}, \cite[Lemma~4.8]{RK}. This equivalence allows us to use these embeddings interchangeably throughout the paper.

Haefliger \cite{HAE} showed that smooth knots $S^n\hookrightarrow S^{n+q}$ are trivial when $q> \frac{n+3}{2}$, but can be nontrivial when $3\leq q\leq \frac{n+3}{2}$. 
In contrast, Zeeman \cite{ZM} and Stallings \cite{STA} independently proved that in piecewise linear and topological locally flat settings, respectively, all such knots are trivial when $q\geq 3$. Haefliger further classified the isotopy classes of smooth spherical embeddings in the border codimension $q=\frac{n+3}{2}$:
\begin{align*}
\pi_{0} Emb(S^{2\ell-1}, S^{3\ell}) = \left\{ \begin{array}{cc} 
              \bbZ & \hspace{3mm}  \text{ if $\ell$ is even, } \ell\geq 2 \\
                \bbZ_2 & \hspace{2.5mm} \text{ if $\ell$ is odd, } \ell\geq 3. \\
                              \end{array} \right.
                                           \end{align*}
In the even case, Haefliger \cite{HAE2} constructed an explicit embedding known as the \textit{Haefliger trefoil knot}, which generates $\pi_{0} Emb(S^{2\ell-1}, S^{3\ell})\approx \bbZ$. More recently, Koytcheff~\cite{RK} extended this result to the odd case. 

 One of the intriguing problems in this area is finding a combinatorial formula for the Haefliger invariant $\mathcal{H}:\pi_{0} Emb(S^{4k-1}, S^{6k})\xrightarrow{\approx} \bbZ$ in terms of the self-intersections of the projection of the knot to a hyperplane. This is analogous to Goussarov, Polyak, and Viro's work \cite{GPV,PV,PVC}, see also~\cite[Chapter~13]{VI}, which expressed  finite-type knot invariants 
  using Gauss diagrams in classical knot theory. A \textit{Gauss diagram} for a knot $S^1\hookrightarrow \bbR^3$ consists of an oriented circle with chords connecting double points that project to the same point in $\bbR^2$.
 Each chord is oriented from the overcrossing to the undercrossing and is assigned a sign based on the local writhe number. The following figure illustrates the Gauss diagram of the $3_1$ (trefoil) knot: 
\begin{figure}[ht]
\centering
 \includegraphics[width=0.6\textwidth]{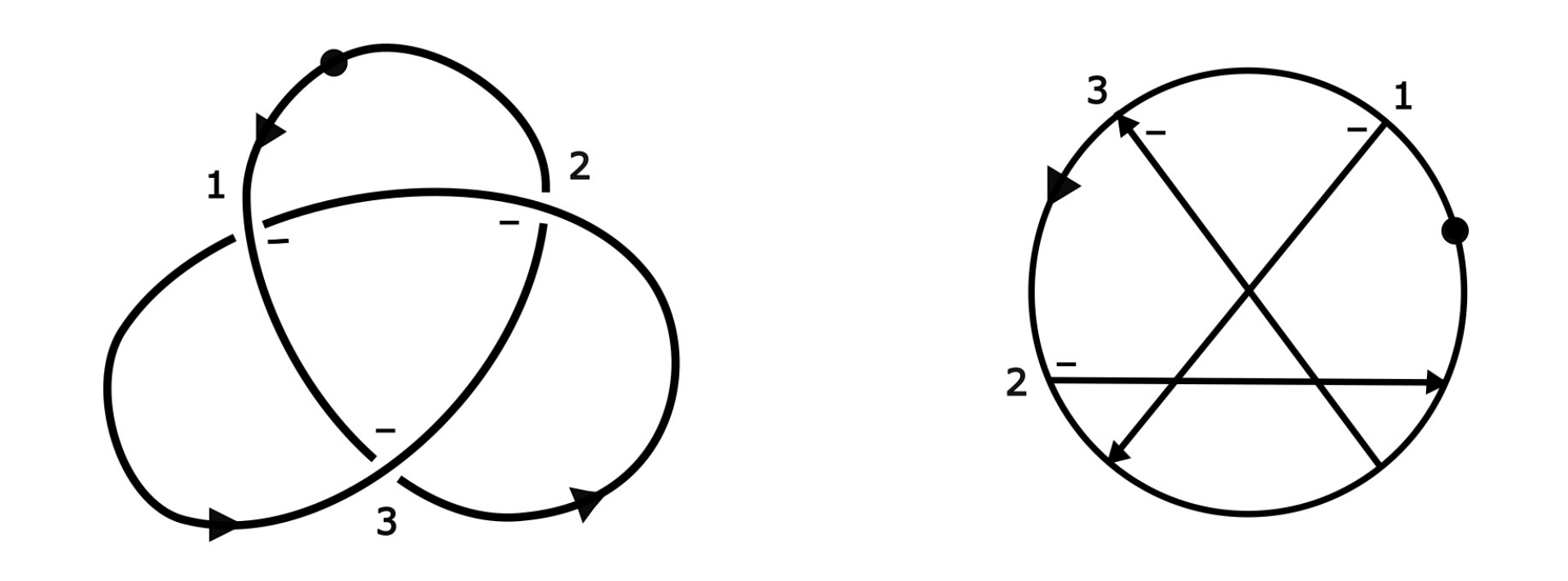}
 \caption{Gauss diagram of $3_1$ knot}
 \end{figure}

An alternative perspective on Gauss diagrams is to view them as the set of double points of a knot projection, forming a $0$-submanifold with an involution that encodes crossing data. The
signs of the intersection points describe  the orientation of this $0$-submanifold. For example, the diagram above can also be represented as follows, with red denoting overcrossings and blue denoting undercrossings.

\begin{figure}[ht]
\centering
 \includegraphics[width=0.35\textwidth]{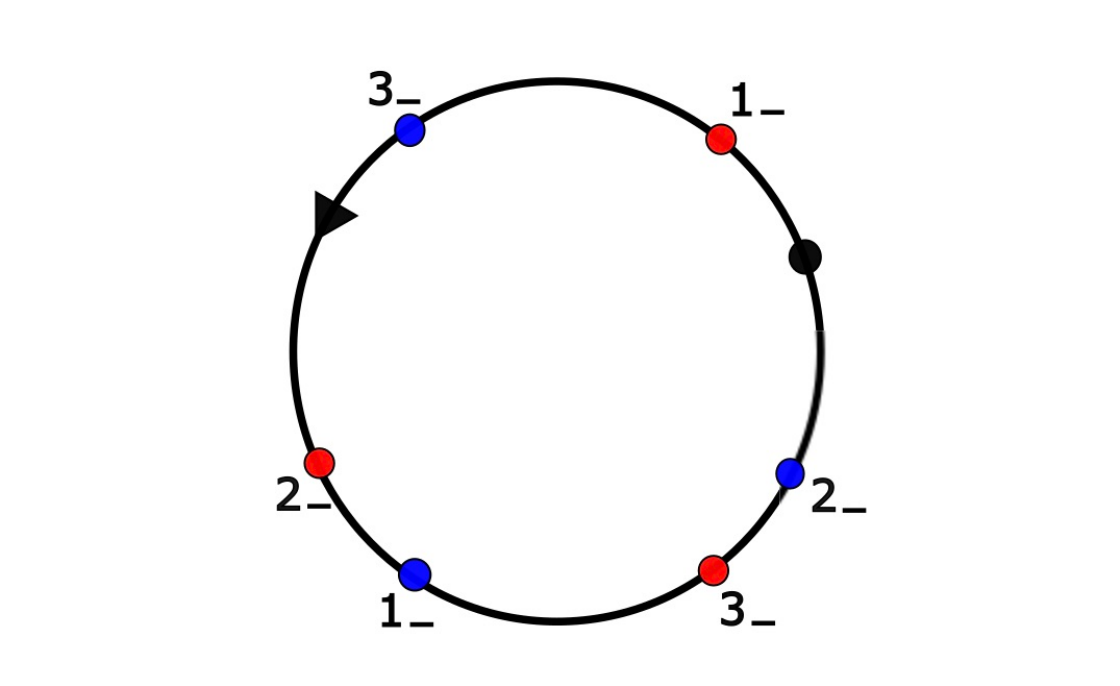}
 \caption{Another description of the Gauss diagram of $3_1$ knot}
 \end{figure}

This description allows us to generalize the notion of a Gauss diagram to higher dimensions. One of our primary motivations is to develop a Goussarov-Polyak-Viro type formula for the Haefliger knots $S^{4k-1}\hookrightarrow \bbR^{6k}$ by recovering it from their projections. The difficulty in this approach is that a nonzero multiple of the Haefliger trefoil is isotopic to an embedding in one lower dimension~\cite{HAE}, indicating that any combinatorial formula for $\mathcal{H}$ must incorporate tangential data beyond double points. However, for the related cases of $2$- and $3$-component links in these dimensions, we succeeded in deriving explicit expressions purely in terms of linking numbers of the preimages of double point sets. Previously, combinatorial formulas for the Haefliger invariant of $(4k-1)$-knots in $\bbR^{6k}$ were produced by M.~Takase~\cite{Takase1,Takase2}. However, his approach is different from ours -- instead of looking at the projection, he studies Seifert surfaces of such knots.

The main motivation for this work comes from a recent discovery~\cite[Theorem~4.2]{SONG} that rationally in codimension at least three the (abelian) group of isotopy classes of spherical and string links is isomorphic to the space of uni-trivalent trees of degree zero modulo IHX relations. This suggests that some techniques of the theory of Vassiliev invariants for classical knots and links in $\bbR^3$ can be generalized to higher dimensions.  The simplest invariant -- the linking number between two components in this description corresponds to the coefficient in front of the bar graph:
$$
\begin{tikzpicture}[baseline=-.65ex]
\node (v) at (0,0) {$1$};
\node (w) at (1.3,0) {$2$};
\draw (v) edge (w);
\end{tikzpicture}
$$
This paper is concerned with the second most simple situation of tripod graphs:
$$
\begin{tikzpicture}[baseline=0]
\node (v1) at (-.6,0) {$1$};
\node (v2) at (0,0) {$2$};
\node (v3) at (.6,0) {$3$};
\node [int] (w) at (0,.6) {};
\draw (w) edge (v1) edge (v2) edge (v3);
\end{tikzpicture}
\hspace{.5in}
\begin{tikzpicture}[baseline=0]
\node (v1) at (-.6,0) {$1$};
\node (v2) at (0,0) {$1$};
\node (v3) at (.6,0) {$2$};
\node [int] (w) at (0,.6) {};
\draw (w) edge (v1) edge (v2) edge (v3);
\end{tikzpicture}
\hspace{.5in}
\begin{tikzpicture}[baseline=0]
\node (v1) at (-.6,0) {$1$};
\node (v2) at (0,0) {$2$};
\node (v3) at (.6,0) {$2$};
\node [int] (w) at (0,.6) {};
\draw (w) edge (v1) edge (v2) edge (v3);
\end{tikzpicture}
\hspace{.5in}
\begin{tikzpicture}[baseline=0]
\node (v1) at (-.6,0) {$1$};
\node (v2) at (0,0) {$1$};
\node (v3) at (.6,0) {$1$};
\node [int] (w) at (0,.6) {};
\draw (w) edge (v1) edge (v2) edge (v3);
\end{tikzpicture}
$$
The invariants $\mathcal V$, $\mathcal W_1$, $\mathcal W_2$, $\mathcal H$ that we produce in Theorems~\ref{inv1}, \ref{inv2},
\ref{knotinv} are respective coefficients in front of these tripods.
More generally, the rational homotopy groups of spaces of string links are described as the homology of  certain {\it Hairy Graph-Complexes}~\cite[Corollary~3.6 and Section~5.5, respectively]{SONG,FTW}.
In particular the space of unitrivalent graphs modulo IHX relations
appears as part of the rational homotopy of such spaces, which we believe can be detected by a generalized Goussarov-Polyak-Viro approach.
 
\medskip

\textbf{Outline of the paper:} In Section \ref{Section2}, we introduce a higher-dimensional analogue of Gauss diagrams by analyzing the set of double points arising from projections of embeddings $M^{2\ell-1}\hookrightarrow \bbR^{3\ell}$, where $M$ is a closed, oriented, and possibly disconnected manifold. We also discuss how isotopy affects this set of double points.

In Section \ref{Section3}, we first give the \textit{Brunnian splitting} of $k$-component links $\underset{k}{\sqcup}S^n\hookrightarrow \bbR^{n+q}$ for $q\geq 3$. In particular, any Brunnian link $\underset{k}{\sqcup}S^{2\ell-1}\hookrightarrow \bbR^{3\ell}$ is trivial for $k>3$. For this reason, we consider only links with at most three components. Our first main result~Theorem~\ref{inv1}, gives a Goussarov-Polyak-Viro type invariant for $3$-component links in the given dimensions. We also discuss a few examples; for instance, this invariant evaluates to (minus) one on the high-dimensional Borromean link. We then establish a second main result, Theorem~\ref{inv2}, which provides a similar formula for $2$-component links. Corollary~\ref{inv2_imm} treats its special case, when the projection is an immersion.

In Section \ref{conjecture}, Theorem~\ref{knotinv}  presents a Goussarov-Polyak-Viro type formula for the Haefliger invariant \sloppy $\mathcal{H}:\pi_0Emb(S^{4k-1},\bbR^{6k})\xrightarrow{\approx} \bbZ$, in the case where the projection of a knot $S^{4k-1}\hookrightarrow\bbR^{6k}$ to $\bbR^{6k-1}$ is a generic immersion. In the last Subsection~\ref{ss:Sakai} we
show that our approach gives a different proof of Sakai's crossing change formula~\cite[Theorem~2.5]{SAK2} for the invariant in question.

All three combinatorial formulas from Theorem~\ref{inv1}, Corollary~\ref{inv2_imm}, and Theorem~\ref{knotinv} share a similar structure: Each one contains one summand that depends solely on the projection and an additional term expressed as a sum of linking numbers between overcrossing and undercrossing double-point set components.

In Appendix \ref{appen}, we relate our invariants to configuration spaces. 
Theorem ~\ref{retract}  states that the graphing map $G:\Omega^nC(k,\bbR^q)\rightarrow Emb_{\partial}(\underset{k}{\sqcup}\bbR^n,\bbR^{n+q})$ admits a homotopy retraction. In particular, it induces injections on both rational homotopy $\pi_*(-)\otimes\bbQ$ and rational homology $H_*(-;\bbQ)$, see \cite[Conjecture 5.7]{RAF}.

 We note that many results of this paper are part of the first author's Ph.D. thesis \cite[Chapters~3 and~4]{NG1} supervised by the second author. Namely,
 Theorems~\ref{inv2} and~\ref{retract} are \cite[Theorem~3.6.1 and~3.7.1]{NG1}, respectively. Theorem~\ref{inv1} is an improved version of \cite[Theorem~3.5.1]{NG1}. Theorem~\ref{knotinv} was \cite[Conjecture~4.4.1]{NG1}. A reader interested in generalizing our Haefliger invariant formula to the case when the projection is not an immersion is invited to check different equivalent versions of it, see \cite[Equations (4.4.1), (4.4.2), (4.4.3)]{NG1}. 

 \medskip

\textbf{Notation:}  We work in the smooth category throughout the paper. We use the symbol \enquote{$\approx$} (or \enquote{$=$}) for a group isomorphism; the symbol \enquote{$\simeq$} stands for a homotopy equivalence between two topological spaces; the symbol \enquote{$\cong$} means a diffeomorphism between manifolds. For convenience, we often use the same notation to refer to both an embedding and its isotopy class (or to immersion and its regular homotopy class). 

\medskip

\textbf{Orientation, linking and intersection  conventions:}
For an oriented manifold $M$ with boundary, we orient its boundary $\partial M$ using the \textit{outward normal first convention}: for $p\in\partial M$ an ordered basis of $T_p(\partial M)$ is positively oriented if the outward normal vector at $p$ followed by this basis is a positively oriented basis of $T_pM$. 

We use the following definition of the linking number. 
Let $M^m$ and $N^n$ be oriented, triangulated manifolds in $\bbR^{m+n+1}$, viewed as singular cycles. Let $\mathcal{N}^{n+1}$ be a singular chain such that $\partial \mathcal{N}=N$. When possible, $\mathcal{N}$ is chosen as a manifold with
boundary~$N$, oriented by the outward normal first convention. Generically, $M$ and $\mathcal{N}$ intersect at finitely many points, and the linking number $lk(M,N)$ is defined as the algebraic number of intersections of $M$ and $\mathcal{N}$, denoted by $M\cdot \mathcal{N}$. At each intersection, we assign $+1$ (respectively $-1$) if the orientation $Or(M)$ of $M$, followed by the orientation $Or(\mathcal{N})$ of $\mathcal{N}$ matches (respectively opposes) the positive orientation of $\bbR^{m+n+1}$. Note that the linking number is (anti-)commutative:
\sloppy
 \begin{equation}\label{lksym}
    lk(N,M)=(-1)^{(m+1)(n+1)}lk(M,N)  
 \end{equation}

Let $M^m$ and $N^n$ be smooth oriented manifolds that intersect transversely in $\bbR^k$. Then their intersection $M\cap N$, which has dimension $m+n-k$, inherits an orientation via the \textit{transverse intersection rule}: Given a point $p\in M\cap N$, let $\vec u=(u_1,...,u_{m+n-k})$ be a basis of $T_p(M\cap N)$. Choose complementary frames $\vec v=(v_1,...,v_{k-n})$ and $\vec w=(w_1,...,w_{k-m})$ so that $(\vec u,\vec v)$ and $(\vec u,\vec w)$ form positive bases of $T_p M$ and $T_p N$, respectively. Then $\vec u$ defines a positive orientation of $T_p(M\cap N)$ if and only if the combined frame $(\vec u,\vec v,\vec w)$ is positively oriented in $\bbR^k$. Note that this rule is (anti-)symmetric:
\begin{equation}\label{inters_sym}
     Or(N\cap M)= (-1)^{(k-m)(k-n)} Or(M\cap N).
 \end{equation}
\end{section}

\medskip

\textbf{Acknowledgement:} The authors are grateful to O.~Saeki for discussions and in particular for suggesting to use \cite[Compression Theorem]{COM} for the proof of Theorem~\ref{knotinv}. The first author was supported by the Max Planck Institute for Mathematics, Bonn. 
 The second author was partially supported by the Simons Foundation award \#933026.

\begin{section}{The set of double points and the cobordism under isotopy}\label{Section2}

Throughout this section, we assume $M$ to be a closed oriented possibly disconnected manifold of dimension $2\ell-1$. By \cite[Theorem~4.7.6]{CTC}, a generic smooth map $f:M^{2\ell-1}\rightarrow \bbR^{3\ell-1}$ is an embedding except at double points, forming an $(\ell-1)$-submanifold $L(f)$, and singular points, forming an $(\ell-2)$-submanifold $\Sigma^1(f)$. Locally near $\Sigma^1(f)$, $f$ is given by 
$$y_1=\frac{1}{2}x_1^2,\hspace{0.3cm} y_j=x_j \hspace{0.2cm}(2\leq j\leq 2\ell-1), \hspace{0.3cm} y_{i+2\ell-2}=x_1x_i\hspace{0.2cm}(2\leq i\leq \ell+1),$$
where the rank of the differential drops by one. Such singular points are known as \textit{Whitney umbrella points}. The closure $\overline{L(f)}=L(f)\cup \Sigma^1(f)$ forms a smooth closed submanifold with involution in $M$, where $\Sigma^1(f)$ represents the fixed points. Its image $f(\overline{L(f)})$ is an $(\ell-1)$-submanifold of $\bbR^{3\ell-1}$ with boundary $f(\Sigma^1(f))$.

If $f$ arises from a generic projection of an embedding $M^{2\ell-1}\hookrightarrow \bbR^{3\ell}$, then $$ L(f)=L_{+}(f)\cup L_{-}(f),$$
where $L_{+}(f)$ and $L_{-}(f)$ correspond to overcrossing and undercrossing, respectively. Their closures $\overline{L_{\pm}(f)}=L_{\pm}(f)\cup \Sigma^1(f)$
are manifolds with boundary $\partial \overline{L_{\pm}(f)}=\Sigma^1(f)$.
Consequently, $\overline{L(f)}$ is the double of $\overline{L_{+}(f)}$, forming an $(\ell-1)$-dimensional link in $M$, and serving as an analogue of the Gauss diagram.

\begin{subsection}{Orientation on $\overline{L(f)}$}\label{orient}
Recall that the signs on the chords in the Gauss diagrams correspond to the orientation of the $0$-dimensional intersection set. We now discuss how the orientation is chosen in the higher-dimensional case.

When the codimension of $f:M^{2\ell-1}\rightarrow \bbR^{3\ell-1}$ is even (i.e. $\ell=2k$), both $L(f)\subset M^{4k-1}$ and its image $f(L(f))\subset\bbR^{6k-1}$ admit natural orientations. For a generic map $f:M^{4k-1}\rightarrow \bbR^{6k-1}$, the \textit{Ekholm orientation} (see \cite{EK,EK1}) is defined on $f(L(f))$ using the transverse intersection rule, see Orientation conventions in Section~\ref{s:intro}. Since the codimension of $f$ is even, this orientation does not depend on the order of intersecting sheets. The Ekholm orientation on $L(f)$ is chosen so that $f:L(f)\rightarrow f(L(f))$ preserves~it. 

However, if Whitney umbrella points are present, this orientation does not globally extend over $\overline{L(f)}$ due to the involution symmetry near those singularities. This issue is resolved when $f$ is a projection of an embedding, that is, $f:M^{4k-1}\hookrightarrow \bbR^{6k}\rightarrow \bbR^{6k-1}$. In this case, we define an \textit{adjusted Ekholm orientation}: assign
the Ekholm orientation to $\overline{L_{+}(f)}$ and the opposite orientation to $\overline{L_{-}(f)}$, yielding $$\overline{L(f)}=\overline{L_{+}(f)}\cup-\overline{L_{-}(f)},$$
ensuring a consistent global orientation, even across Whitney umbrella points. 

For maps of odd codimension (i.e., $\ell=2k+1$), let $f:M^{4k+1}\rightarrow \bbR^{6k+2}$ be a projection of an embedding. Then $L(f)$ can still be oriented by fixing an ordering of the intersecting sheets at each point of $f(L(f))$. Locally, if $X$ and $Y$ are two such sheets intersecting transversely, we orient the preimage of $f(X)\cap f(Y)$ in $X$ according to the orientation of $f(X)\cap f(Y)$; while the corresponding preimage in $Y$ is oriented using the orientation of $f(Y)\cap f(X)$. As a result, these two preimages are mapped to $f(X)\cap f(Y)$ with opposite orientations. 

 For concreteness consider the case $\ell=2$. If $f: M^{3}\rightarrow \bbR^{5} $ is the projection of an embedding $M^{3}\hookrightarrow \bbR^{6}$, then, $\overline{L(f)}$ is a $1$-manifold with a $0$-dimensional submanifold $\Sigma^1(f)$. Since the codimension is even, $\overline{L(f)}$ (and its image) inherit a natural orientation, forming an oriented link in $M$ with  components of three types:
\begin{center}
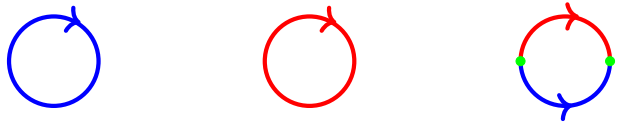

\begin{tikzpicture}[scale=0.85]
\draw[ultra thick, red,  decoration={markings, mark=at position -0.8 with {\arrow{<}}},
        postaction={decorate}]  (0,0) circle (0.7cm); 
\begin{scope}[shift={(-4,0)}]
\draw[ultra thick, blue,  decoration={markings, mark=at position -0.8 with {\arrow{<}}},
        postaction={decorate}]  (0,0) circle (0.7cm); 
\end{scope}
\begin{scope}[shift={(4,0)}]
\draw [ultra thick, red,  decoration={markings, mark=at position -0.5 with {\arrow{<}}},
        postaction={decorate}] (0:0.7) arc [radius=0.7, start angle=0, end angle=180];
  \draw [ultra thick, blue,  decoration={markings, mark=at position -0.45 with {\arrow{>}}},
        postaction={decorate}] (180:0.7) arc [radius=0.7, start angle=180, end angle=360];
  \filldraw [green] (-0.7,0) circle (2pt);
 \filldraw [green] (0.7,0) circle (2pt);
  \end{scope}
 \end{tikzpicture} 
 
\captionof{figure}{Possible components of $\overline{L(f)}$ when $\ell=2$}
\end{center}
We denote the undercrossing in blue, overcrossing in red and the points of Whitney umbrellas in green.
The images of these link components under $f$  in $\bbR^5$ are described in Figure~\ref{fig:images} in which we follow the Ekholm orientation convention.\\

\begin{center}
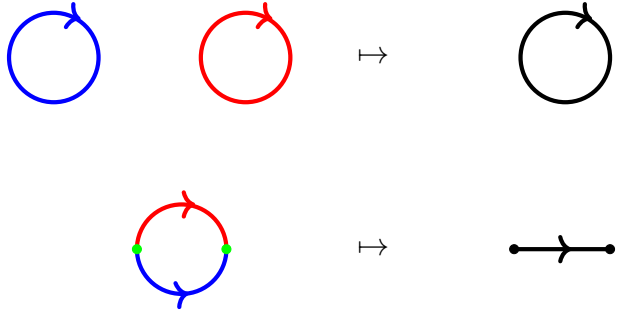

\begin{tikzpicture}[scale=0.85]
\draw[ultra thick, red,  decoration={markings, mark=at position -0.8 with {\arrow{<}}},
        postaction={decorate}]  (-1,0) circle (0.7cm); 
\begin{scope}[shift={(-3,0)}]
\draw[ultra thick, blue,  decoration={markings, mark=at position -0.8 with {\arrow{<}}},
        postaction={decorate}]  (-1,0) circle (0.7cm); 
\end{scope}
\node[ultra thick] at (1,0) {$\mapsto$};
\begin{scope}[shift={(4,0)}]
\draw[ultra thick, decoration={markings, mark=at position -0.8 with {\arrow{<}}},
        postaction={decorate}]  (0,0) circle (0.7cm); 
\end{scope}

\begin{scope}[shift={(-2,-3)}]
\draw [ultra thick, red,  decoration={markings, mark=at position -0.5 with {\arrow{<}}},
        postaction={decorate}] (0:0.7) arc [radius=0.7, start angle=0, end angle=180];
  \draw [ultra thick, blue,  decoration={markings, mark=at position -0.45 with {\arrow{>}}},
        postaction={decorate}] (180:0.7) arc [radius=0.7, start angle=180, end angle=360];
  \filldraw [green] (-0.7,0) circle (2pt);
 \filldraw [green] (0.7,0) circle (2pt);
 \node[ultra thick] at (3,0) {$\mapsto$};
  \draw [ultra thick, decoration={markings, mark=at position -0.4 with {\arrow{>}}},
        postaction={decorate}] (5.2,0)--(6.7,0);
          \filldraw (5.2,0) circle (2pt);
 \filldraw (6.7,0) circle (2pt);
  \end{scope}
 \end{tikzpicture}
  \captionof{figure}{Corresponding $f(\overline{L(f)})$ for $\ell=2$}\label{fig:images}
 \end{center} 
 \end{subsection}
 
\begin{subsection}{Cobordisms induced by isotopy}\label{cobord}

To construct an isotopy invariant, we examine what happens when we isotope $M^{2\ell-1}\hookrightarrow \bbR^{3\ell}$ and consider its projection to $\bbR^{3\ell-1}\times [0,1]$. By \cite[Theorem~4.7.6]{CTC} again, a generic smooth map $H:M^{2\ell-1}\times [0,1] \rightarrow \bbR^{3\ell-1}\times [0,1]$ is an embedding except as follows:
there are double points forming an immersed $\ell$-manifold $L(H)$ with transverse self-intersection, and Whitney umbrella points forming an $(\ell-1)$-submanifold $\Sigma^{1}(H)$. Moreover, now $H(L(H))$ can have a discrete set of triple points with transverse self-intersection. 
Thus, $\overline{L(H)}=L(H)\cup \Sigma^{1}(H)$ is an immersed compact $\ell$-manifold with involution in $M\times [0,1]$ (with the set of fixed points $\Sigma^{1}(H)$), and $H(\overline{L(H)})$ is an immersed $\ell$-manifold in $\bbR^{3\ell-1}\times [0,1].$

Projecting an isotopy yields a cobordism between $\overline{L(f_0)}$ and $\overline{L(f_1)}$ given by $$\overline{L(H)}=\overline{L_{+}(H)}\cup \overline{L_{-}(H)}.$$ 
 Here, $\overline{L_{+}(H)}$ (similarly $\overline{L_{-}(H)}$) is an $\ell$-manifold with corners, with internal boundary $\partial\overline{L_{+}(H)}=\Sigma^{1}(H)$ and other boundary components $\overline{L_{+}(f_0)}$ and $\overline{L_{+}(f_1)}$. Moreover, $\overline{L(H)}$ is a cobordism with involution (appearing as a double of $\overline{L_{+}(H)}$), where Whitney umbrella points define a codimension one submanifold of fixed points.

For instance, when $\ell=2$, we get a $2$-dimensional cobordism $\overline{L(H)}\subset M^3\times[0,1]$ as a double of a manifold $\overline{L_{+}(H)}$ with boundary $\Sigma^1(H)$, whose image under $H$ may have triple points.
In particular, we get the following possible cobordisms as the double~$\overline{L(H)}$ that we compare with the classical case~$\ell=1$. 

1) The analogue of the first Reidemeister move:

\begin{center}
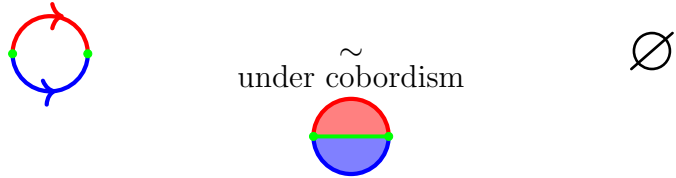

\begin{tikzpicture}
\draw [ultra thick, red,  decoration={markings, mark=at position -0.5 with {\arrow{<}}},
        postaction={decorate}] (0:0.5) arc [radius=0.5, start angle=0, end angle=180];
  \draw [ultra thick, blue,decoration={markings, mark=at position -0.45 with {\arrow{>}}},
        postaction={decorate}] (180:0.5) arc [radius=0.5, start angle=180, end angle=360];
          \filldraw [green] (-0.5,0) circle (1.5pt);
 \filldraw [green] (0.5,0) circle (1.5pt);
 
\begin{scope}[shift={(4,0)}]
 \node (a) {$\sim$};

 \node [below] at (0,0) {under cobordism};
 \end{scope}
 \begin{scope}[shift={(4,-1.1)}]
  \draw [ultra thick, red, fill=red!50](0:0.5) arc [radius=0.5, start angle=0, end angle=180];
  \draw [ultra thick, blue, fill=blue!50] (180:0.5) arc [radius=0.5, start angle=180, end angle=360];
          \draw [green, fill=green!420] (-0.5,0) circle (1.5pt);
 \filldraw [green] (0.5,0) circle (1.5pt);
 \filldraw [green, line width= 1.5pt] (-0.5,0)--(0.5,0);
 \end{scope}
 \node at (8,0) {\huge$\varnothing$}; 

\end{tikzpicture}
\captionof{figure}{First Reidemeister move for $M^3\hookrightarrow \bbR^6$}
\end{center}

\begin{center}
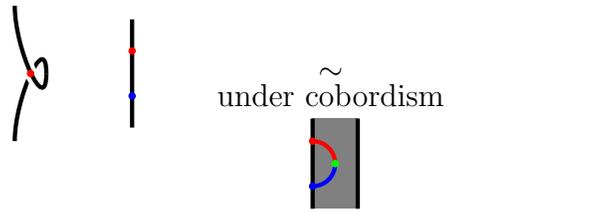

\begin{tikzpicture}[domain=-1:1, scale=0.6, knot gap=3]
      \draw [knot=black, ultra thick] (0.7,0) to [out=up, in=up, looseness=1.5] (0,-1.5);
      \draw [knot=black, ultra thick] (0,1.5) to [out=down, in=down, looseness=1.5] (0.7,0) ;
              \filldraw [red] (0.35,0) circle (2pt);

       \begin{scope}[shift={(-3,2)}]
 \draw [ultra thick] (5.6,-0.8)--(5.6,-3.2);
  \filldraw [red] (5.6,-1.5) circle (2pt);
  \filldraw [blue] (5.6,-2.5) circle (2pt);
    \end{scope}
 \node at (7,0) {$\sim$};

 \node [below] at (7,0) {under cobordism};

 \begin{scope}[shift={(6,0)}]
 \draw [ultra thick] (7,1)--(7,-1);
 \end{scope}

\begin{scope}[shift={(1,0)}]
  \path [fill=gray]  (5.6,-3) rectangle (6.6, -1);

 \draw [ultra thick] (5.6,-1)--(5.6,-3);
  \draw [ultra thick] (6.6,-1)--(6.6,-3);
 
   \filldraw [red] (5.6,-1.5) circle (2pt);
  \filldraw [blue] (5.6,-2.5) circle (2pt);
\draw [ultra thick, red] (5.6,-1.5) arc [radius=0.5, start angle=90, end angle=0];
\draw [ultra thick, blue] (5.6,-2.5) arc [radius=0.5, start angle=-90, end angle=0];
    \filldraw [green] (6.1,-2) circle (2pt);
    \end{scope}
 \end{tikzpicture}
 \captionof{figure}{First Reidemeister move in classical case}
\end{center}

2) The analogue of the second Reidemeister move:

\begin{center}
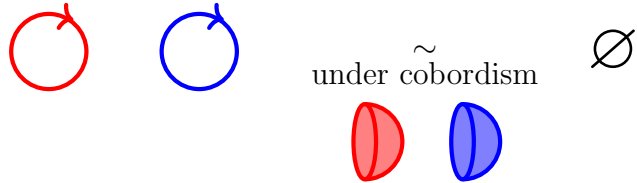

\begin{tikzpicture}
\draw[ultra thick, red, decoration={markings, mark=at position -0.8 with {\arrow{<}}},
        postaction={decorate} ]  (0,0) circle (0.5cm); 
\begin{scope}[shift={(2,0)}]
\draw[ultra thick, blue, decoration={markings, mark=at position -0.8 with {\arrow{<}}},
        postaction={decorate}]  (0,0) circle (0.5cm); 
 \end{scope}
\begin{scope}[shift={(5,0)}]
 \node (a) {$\sim$};
 \node [below] at (0,0) {under cobordism};
\end{scope}
 \node at (7.5,0) {\huge$\varnothing$}; 
 
\draw [ultra thick, red, fill=red!50] (4.2,-0.7) arc [radius=0.5, start angle=90, end angle=-90];
 \draw [ultra thick, red, fill=red!50] (4.2,-1.2) ellipse (0.15cm and 0.5cm);
\draw [ultra thick, blue, fill=blue!50] (5.5,-0.7) arc [radius=0.5, start angle=90, end angle=-90];
 \draw [ultra thick, blue, fill=blue!50] (5.5,-1.2) ellipse (0.15cm and 0.5cm);
\end{tikzpicture}
\captionof{figure}{Second Reidemeister move for $M^3\hookrightarrow \bbR^6$}
\end{center}
\begin{center}
 \begin{tikzpicture}[domain=-1:1, scale=0.6, knot gap=3]
      \draw [knot=black, ultra thick] (0.6,1) to [out=left, in=left, looseness=1.5] (0.6,-1) ;
      \draw [knot=black, ultra thick] (-0.6,1) to [out=right, in=right, looseness=1.5] (-0.6,-1);
                    \filldraw [red] (0,.75) circle (2pt);
              \filldraw [red] (0,-.75) circle (2pt);

\node at (7,0) {$\sim$};
 \node [below] at (7,0) {under cobordism};
\begin{scope}[shift={(4,0)}]
 \draw [ultra thick] (7.5,1)--(7.5,-1);
  \draw [ultra thick] (8.7,1)--(8.7,-1);
 
\end{scope}
\begin{scope}[shift={(-2.5,2)}]
 \draw [ultra thick] (4.8,-1)--(4.8,-3);
 \filldraw [red] (4.8,-1.5) circle (2pt);
  \filldraw [red] (4.8,-2.5) circle (2pt);
  \node [left] at (4.8,-1.5) {\tiny$+$};
\node [left] at (4.8,-2.5) {\tiny$-$};
   \draw [ultra thick] (6,-1)--(6,-3);
    \filldraw [blue] (6,-1.5) circle (2pt);
  \filldraw [blue] (6,-2.5) circle (2pt);
  \node [left] at (6,-1.5) {\tiny$+$};
\node [left] at (6,-2.5) {\tiny$-$};
\end{scope}
\begin{scope}[shift={(1,0)}]
 \path [fill=gray]  (4.8,-3) rectangle (5.8, -1);
 \draw [ultra thick] (4.8,-1)--(4.8,-3);
  \draw [ultra thick] (5.8,-1)--(5.8,-3);
   
 \filldraw [red] (4.8,-1.5) circle (2pt);
  \filldraw [red] (4.8,-2.5) circle (2pt);
  \node [left] at (4.8,-1.5) {\tiny$+$};
\node [left] at (4.8,-2.5) {\tiny$-$};
\draw [ultra thick, red] (4.8,-1.5) arc [radius=0.5, start angle=90, end angle=-90];
     \path [fill=gray]  (7,-3) rectangle (8, -1);
     \draw [ultra thick] (7,-1)--(7,-3);
          \draw [ultra thick] (8,-1)--(8,-3);
           
 \filldraw [blue] (7,-1.5) circle (2pt);
  \filldraw [blue] (7,-2.5) circle (2pt);
  \node [left] at (7,-1.5) {\tiny$+$};
\node [left] at (7,-2.5) {\tiny$-$};

\draw [ultra thick, blue] (7,-1.5) arc [radius=0.5, start angle=90, end angle=-90];
\end{scope}
\end{tikzpicture}
 
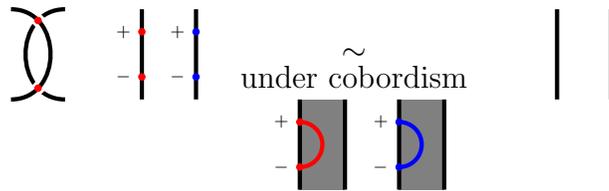
\captionof{figure}{Second Reidemeister move in classical case}
\end{center}

There are other possible embedded cobordisms which do not appear in the classical case. 

3) Saddle cobordism  containing points of Whitney umbrella:
 \begin{figure}[htbp]
\begin{center}
 \includegraphics[width=0.3\textwidth]{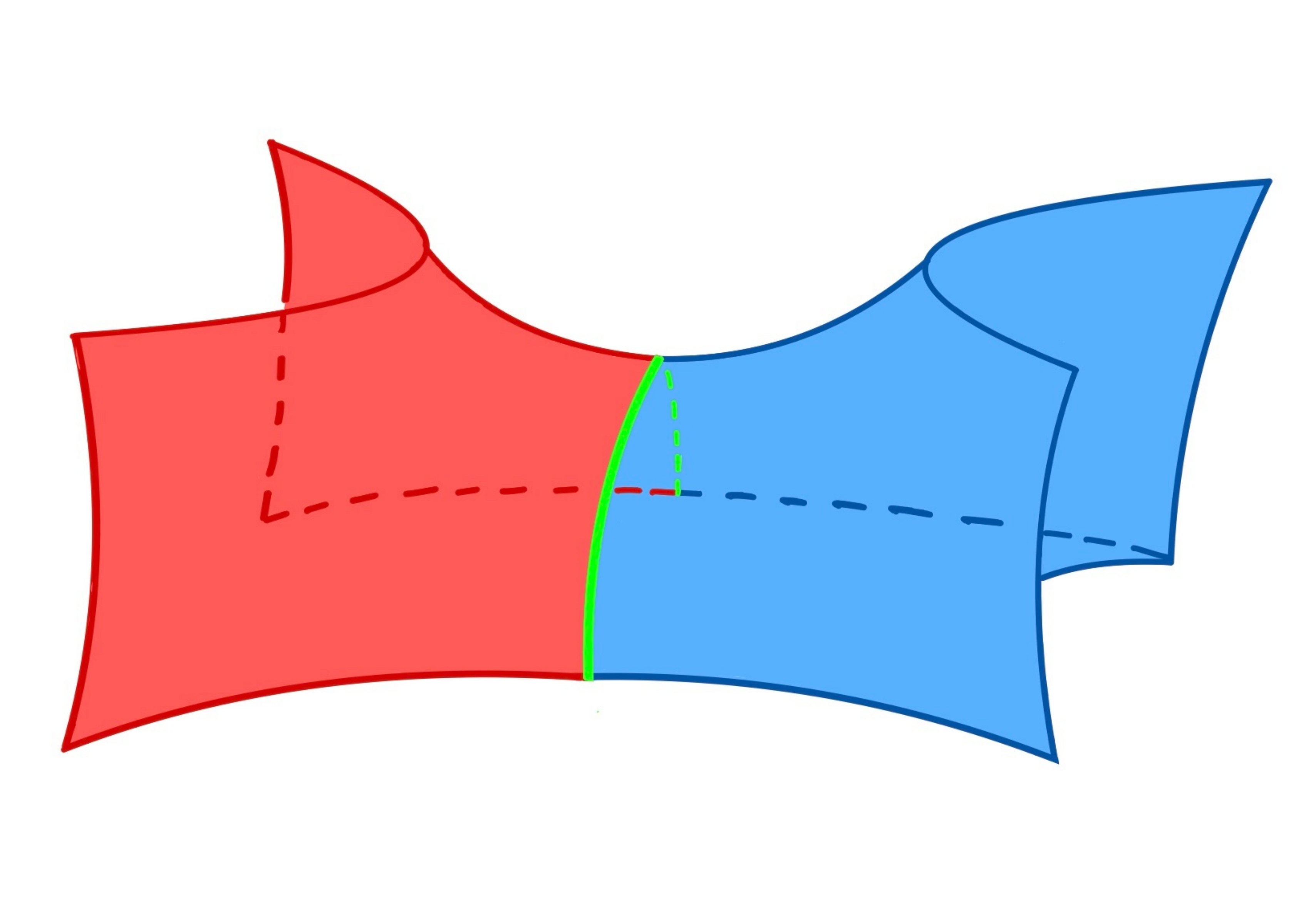}
 \end{center}
 \end{figure}

\begin{center}
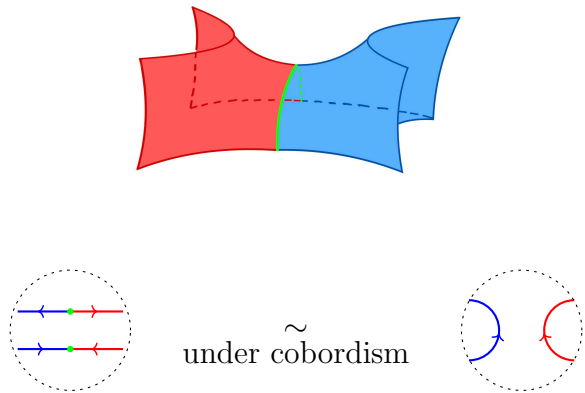

\begin{tikzpicture}
\draw[ thick, red, decoration={markings, mark=at position -0.4 with {\arrow[scale= 0.7]{<}}},
        postaction={decorate}] (0.7,0.25)--(0,0.25);
\draw[thick, blue, decoration={markings, mark=at position -0.5 with {\arrow[scale= 0.7]{<}}},
        postaction={decorate}] (-0.7,0.25)--(0,0.25);
\filldraw[green] (0,0.25) circle (1pt);
\draw[thick, red, decoration={markings, mark=at position -0.4 with {\arrow[scale= 0.7]{>}}},
        postaction={decorate}] (0.7,-0.25)--(0,-0.25);
\draw[ thick, blue, decoration={markings, mark=at position -0.5 with {\arrow[scale= 0.7]{>}}},
        postaction={decorate}] (-0.7,-0.25)--(0,-0.25);
\filldraw[green] (0,-0.25) circle (1pt);
\draw [dashed, dash pattern=on 1pt off 2pt] (0,0) circle (0.8cm);

 \node at (3,0) {$\sim$};

 \node [below] at (3,0) {under cobordism};
 \begin{scope}[shift={(6,0)}]
\draw [ thick, red, decoration={markings, mark=at position -0.4 with {\arrow[scale= 0.7]{<}}},
        postaction={decorate}] (0.7,0.4) arc [radius=0.4, start angle=90, end angle=270];
\draw [ thick, blue, decoration={markings, mark=at position -0.4 with {\arrow[scale= 0.7]{<}}},
        postaction={decorate}] (-0.7,0.4) arc [radius=0.4, start angle=90, end angle=-90];
\draw [dashed,dash pattern=on 1pt off 2pt] (0,0) circle (0.8cm);
\end{scope}
\end{tikzpicture}
\captionof{figure}{Saddle (and its local picture) with Whitney umbrella points}
\end{center}

This cobordism can be used to get rid of Whitney umbrella points.

4) Two saddle cobordisms occurring simultaneously:

\begin{figure}[htbp]
\begin{center}
 \includegraphics[width=0.5\textwidth]{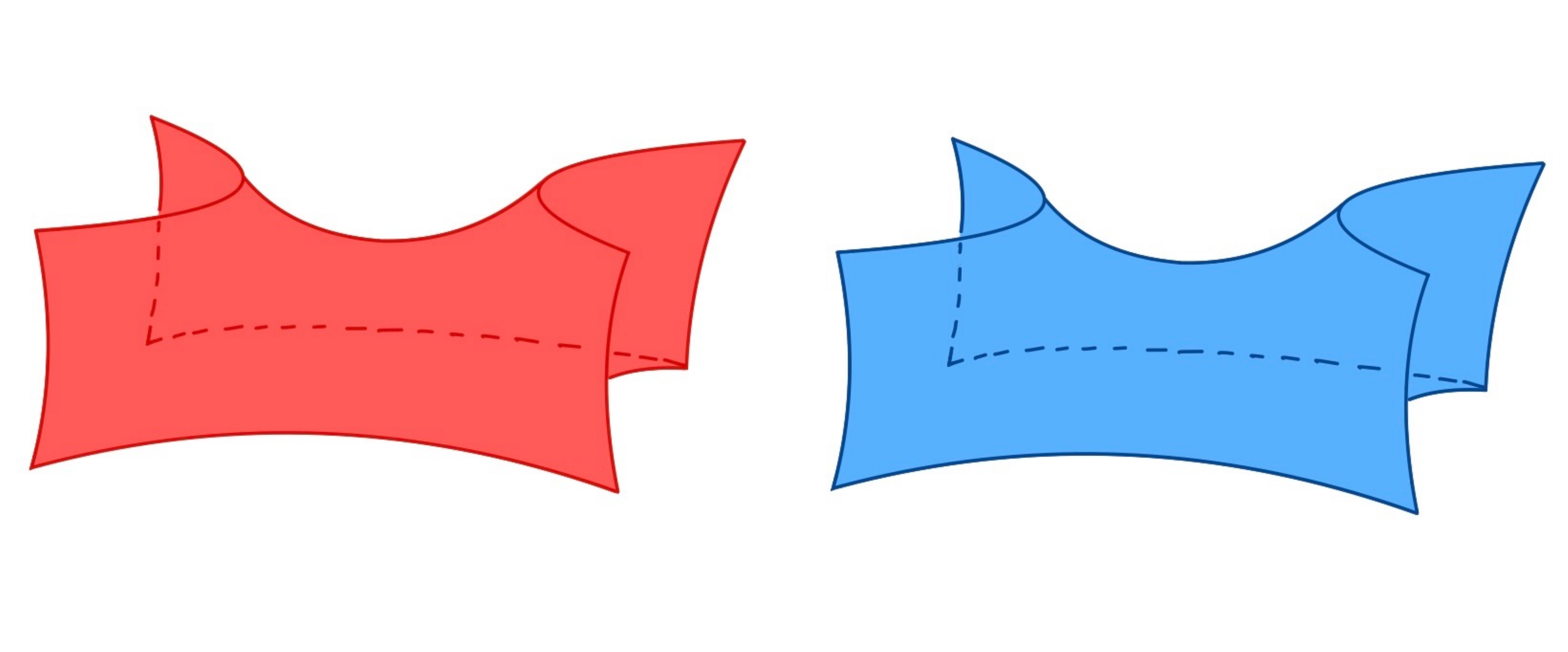}
 \end{center}
\end{figure}

\begin{center}
\begin{tikzpicture}
\draw[ thick, red, decoration={markings, mark=at position -0.4 with {\arrow[scale= 0.7]{>}}},
        postaction={decorate}] (0.7,0.25)--(-0.7,0.25);

\draw[thick, red, decoration={markings, mark=at position -0.4 with {\arrow[scale= 0.7]{<}}},
        postaction={decorate}] (0.7,-0.25)--(-0.7,-0.25);

\draw [dashed,dash pattern=on 1pt off 2pt] (0,0) circle (0.8cm);

 \node at (3,-1) {$\sim$};

 \node [below] at (3,-1) {under cobordism};
 \begin{scope}[shift={(6,0)}]
\draw [ thick, red, decoration={markings, mark=at position -0.4 with {\arrow[scale= 0.7]{>}}},
        postaction={decorate}] (0.7,0.4) arc [radius=0.4, start angle=90, end angle=270];
\draw [ thick, red, decoration={markings, mark=at position -0.4 with {\arrow[scale= 0.7]{<}}},
        postaction={decorate}] (-0.7,0.4) arc [radius=0.4, start angle=90, end angle=-90];
\draw [dashed,dash pattern=on 1pt off 2pt] (0,0) circle (0.8cm);
\end{scope}
\begin{scope}[shift={(0,-2)}]
\draw[ thick, blue, decoration={markings, mark=at position -0.4 with {\arrow[scale= 0.7]{>}}},
        postaction={decorate}] (0.7,0.25)--(-0.7,0.25);

\draw[thick, blue, decoration={markings, mark=at position -0.4 with {\arrow[scale= 0.7]{<}}},
        postaction={decorate}] (0.7,-0.25)--(-0.7,-0.25);

\draw [dashed, dash pattern=on 1pt off 2pt] (0,0) circle (0.8cm);
\begin{scope}[shift={(6,0)}]
\draw [ thick, blue,  decoration={markings, mark=at position -0.4 with {\arrow[scale= 0.7]{>}}},
        postaction={decorate}] (0.7,0.4) arc [radius=0.4, start angle=90, end angle=270];
\draw [ thick, blue,  decoration={markings, mark=at position -0.4 with {\arrow[scale= 0.7]{<}}},
        postaction={decorate}] (-0.7,0.4) arc [radius=0.4, start angle=90, end angle=-90];
\draw [dashed, dash pattern=on 1pt off 2pt] (0,0) circle (0.8cm);
\end{scope}
\end{scope}
\end{tikzpicture}

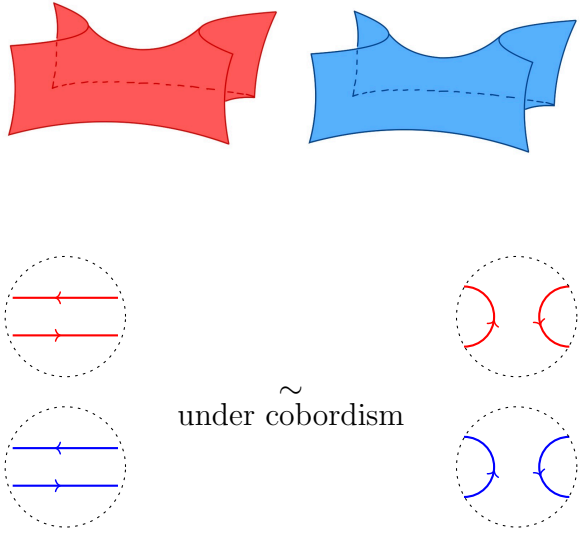
\captionof{figure}{Saddles (and their local pictures) that occur simultaneously}
\end{center}

In addition to embedded cobordisms, one has  a high-dimensional analogue of the third Reidemeister move, corresponding to the appearance of a triple point in the projection of an isotopy of $M^{2\ell-1}\hookrightarrow \bbR^{3\ell}$ to $\bbR^{3\ell-1}$. Consider $\bbR^\ell_x\times \bbR^\ell_y\times \bbR^\ell_z$ as the space-time model,  where the time coordinate $t=\sum_{i=1}^\ell x_i +\sum_{i=1}^\ell y_i+\sum_{i=1}^\ell z_i$ is the sum of all coordinates. The triple point move can be described by the following picture: 
\begin{center}
 \begin{tikzpicture}[domain=-1:1, scale=0.6, knot gap=3]
 
  \draw [knot=black, ultra thick] (-1,1)--(1.2,-1);
      \node[above] at (-1,1) {\tiny $3$};
      \draw [knot=black, ultra thick] (1.2,1)--(-1,-1);
            \node[above] at (1.2,1) {\tiny $2$};
      \draw [knot=black, ultra thick] (0.1,1)--(0.1,-1);
      \node[above] at (0.1,1) {\tiny $1$};
\node at (4,0) {$\longrightarrow$};
\begin{scope}[shift={(-7.5,0)}]
      \draw [knot=black, ultra thick] (-1,1)--(1.2,-1);
      \node[above] at (-1,1) {\tiny $3$};
      \draw [knot=black, ultra thick] (1.2,1)--(-1,-1);
            \node[above] at (1.2,1) {\tiny $2$};
      \draw [knot=black, ultra thick] (-0.5,1)--(-0.5,-1);
      \node[above] at (-0.5,1) {\tiny $1$};
      \node[left] at (-0.5, 0.5) {\tiny A};
\node[left] at (-0.5, -0.5) {\tiny B};
\node at (0.5,0.05) {\tiny C};
\end{scope}

\node at (-4,0) {$\longrightarrow$};

\begin{scope}[shift={(7.5,0)}]
\begin{knot}[
    clip width=3, consider self intersections=true,
  ignore endpoint intersections=false, flip crossing=3
    ]
    \strand [ultra thick] (0.5,1)-- (0.5,-1);
       
        \node[above] at (1,1) {\tiny $2$};
      \strand [ultra thick] (-1.2,1)--(1,-1);
         \strand [ultra thick] (1,1)--(-1.2,-1);
      \node[above] at (-1.2,1) {\tiny $3$};

     \node[above] at (0.5,1) {\tiny $1$};
        \node[right] at (0.5, 0.5) {\tiny B};
\node[right] at (0.5, -0.5) {\tiny A};
\node[above] at (-0.5,-0.3) {\tiny C};
  \end{knot}
\end{scope}
\end{tikzpicture}

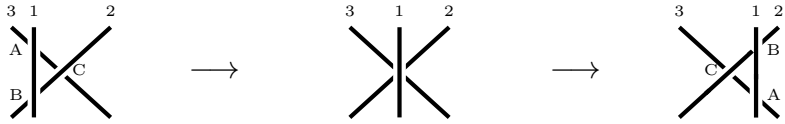
\captionof{figure}{Third Reidemeister move for $M^{2\ell-1}\hookrightarrow \bbR^{3\ell}$}
\end{center}

where sheets $1, 2$ and $3$ are given by $0 \times \bbR^\ell_y\times \bbR^\ell_z$, $\bbR^\ell_x\times 0\times \bbR^\ell_z$, and $\bbR^\ell_x\times \bbR^\ell_y\times 0$, respectively.  Let $A, B$, and $C$ denote the $(\ell-1)$-dimensional double intersections of $1\cap 3$, $1\cap 2$ and $2\cap 3$, respectively, when projected to $\bbR^{3\ell-1}$. Locally, at the moment of the triple point, the three preimages on the sheets appear as double intersections inside $(2\ell-1)$-balls, given as follows:
\begin{figure}[h]
    \centering
    \includegraphics[width=0.72\linewidth]{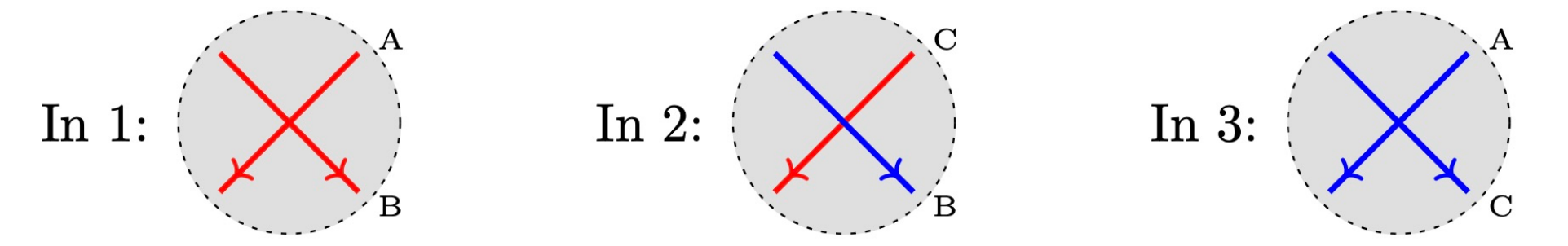}
    \caption{The preimages at the triple point}
\end{figure}

Passing through the triple point corresponds to a simultaneous local crossing change of the double intersection in each of theses balls, as shown in the figure below.
\begin{center}
\begin{tikzpicture}

\begin{knot}[ignore endpoint intersections]

\filldraw [gray!25] (0.2,0) circle (0.8cm);
\draw [dashed,dash pattern=on 1pt off 2pt] (0.2,0) circle (0.8cm);
 \draw [red, very thick, decoration={markings, mark=at position 0.9 with {\arrow{>}}},
        postaction={decorate}] (0.7, 0.5)--(-0.3, -0.5);
\node[right] at (0.7,0.6) {\tiny A};
\draw [red, very thick, decoration={markings, mark=at position 0.9 with  {\arrow{>}}},
        postaction={decorate}] (-0.3,0.5)--(0.7,-0.5);
\node[right] at (0.7,-0.6) {\tiny B};
\end{knot}
\begin{scope}[shift={(-4,0)}];
\node at (2.2,0) {$\longrightarrow$};

\begin{knot}[
    clip width=5, consider self intersections=true,
  ignore endpoint intersections=false, flip crossing=3, background color=gray!25!white
    ]
        \filldraw [gray!25] (0.2,0) circle (0.8cm);
\draw [dashed,dash pattern=on 1pt off 2pt]  (0.2,0) circle (0.8cm);

\strand [draw=red, very thick, only when rendering/.style={
    postaction=decorate,
  },
  decoration={
    markings,
    mark=at position 0.9 with {\arrow{To}}
  }] (0.7, 0.5)--(-0.3, -0.5);
\node[right] at (0.7,0.6) {\tiny A};
\strand [draw=red, very thick, only when rendering/.style={
    postaction=decorate,
  },
  decoration={
    markings,
    mark=at position 0.9 with {\arrow{To}}
  }] (-0.3,0.5)--(0.7,-0.5);
\node[right] at (0.7,-0.6) {\tiny B};

\node[thick] at (-1.2,0) {\small $1$};
\end{knot}
\end{scope}
\begin{scope}[shift={(4,0)}];
\node at (-2,0) {$\longrightarrow$};

\begin{knot}[
    clip width=5, background color=gray!25!white]
\filldraw [gray!25] (0.2,0) circle (0.8cm);
\draw [dashed,dash pattern=on 1pt off 2pt] (0.2,0) circle (0.8cm);
\strand[red, very thick, only when rendering/.style={
    postaction=decorate,
  },
  decoration={
    markings,
    mark=at position 0.9 with {\arrow{To}}
  }] (-0.3,0.5)--(0.7,-0.5);
\node[right] at (0.7,0.6) {\tiny A};
\strand [red, very thick, only when rendering/.style={
    postaction=decorate,
  },
  decoration={
    markings,
    mark=at position 0.9 with {\arrow{To}}
  }] (0.7, 0.5)--(-0.3, -0.5);
\node[right] at (0.7,-0.6) {\tiny B};

\end{knot}

\end{scope}

\begin{scope}[shift={(0,-2)}];
\begin{knot}[ignore endpoint intersections]
\filldraw [gray!25] (0.2,0) circle (0.8cm);
\draw [dashed,dash pattern=on 1pt off 2pt]  (0.2,0) circle (0.8cm);
 \draw [red,very thick,decoration={markings, mark=at position 0.9 with {\arrow{>}}},
        postaction={decorate}] (0.7, 0.5)--(-0.3, -0.5);
\node[right] at (0.7,0.6) {\tiny C};
\draw [blue, very thick,decoration={markings, mark=at position 0.9 with {\arrow{>}}},
        postaction={decorate}] (-0.3,0.5)--(0.7,-0.5);
\node[right] at (0.7,-0.6) {\tiny B};

\end{knot}
\end{scope}

\begin{scope}[shift={(4,-2)}];
\node at (-2,-2) {$\longrightarrow$};

\begin{knot}[
    clip width=5, background color=gray!25!white]
\filldraw [gray!25] (0.2,0) circle (0.8cm);
\draw [dashed,dash pattern=on 1pt off 2pt] (0.2,0) circle (0.8cm);
\strand [blue, very thick, only when rendering/.style={
    postaction=decorate,
  },
  decoration={
    markings,
    mark=at position 0.9 with {\arrow{To}}
  }] (-0.3,0.5)--(0.7,-0.5);
\node[right] at (0.7,-0.6) {\tiny B};
\strand [red, very thick, only when rendering/.style={
    postaction=decorate,
  },
  decoration={
    markings,
    mark=at position 0.9 with {\arrow{To}}
  }] (0.7, 0.5)--(-0.3, -0.5);
\node[right] at (0.7,0.6) {\tiny C};

\end{knot}
\end{scope}

\begin{scope}[shift={(-4,-2)}];
\node at (2.2,-2) {$\longrightarrow$};

\begin{knot}[
    clip width=5, background color=gray!25!white]
\filldraw [gray!25] (0.2,0) circle (0.8cm);
\draw [dashed,dash pattern=on 1pt off 2pt] (0.2,0) circle (0.8cm);
\strand [red, very thick, only when rendering/.style={
    postaction=decorate,
  },
  decoration={
    markings,
    mark=at position 0.9 with {\arrow{To}}
  }] (0.7, 0.5)--(-0.3, -0.5);
\node[right] at (0.7,-0.6) {\tiny B};
\strand [blue, very thick, only when rendering/.style={
    postaction=decorate,
  },
  decoration={
    markings,
    mark=at position 0.9 with {\arrow{To}}
  }](-0.3,0.5)--(0.7,-0.5);
\node[right] at (0.7,0.6) {\tiny C};
\node[thick] at (-1.2,0) {\small $2$};

\end{knot}
\end{scope}

\begin{scope}[shift={(0,-4)}];
\begin{knot}[ignore endpoint intersections]
\filldraw [gray!25] (0.2,0) circle (0.8cm);
\draw [dashed,dash pattern=on 1pt off 2pt]  (0.2,0) circle (0.8cm);
 \draw [blue, very thick, decoration={markings, mark=at position 0.9 with {\arrow{>}}},
        postaction={decorate}] (0.7, 0.5)--(-0.3, -0.5);
\node[right] at (0.7,0.6) {\tiny A};
\draw [blue, very thick,decoration={markings, mark=at position 0.9 with {\arrow{>}}},
        postaction={decorate}] (-0.3,0.5)--(0.7,-0.5);
\node[right] at (0.7,-0.6) {\tiny C};
\
\end{knot}
\end{scope}

\begin{scope}[shift={(4,-4)}];
\node at (-2,2) {$\longrightarrow$};

\begin{knot}[
    clip width=5, background color=gray!25!white]
\filldraw [gray!25] (0.2,0) circle (0.8cm);
\draw [dashed,dash pattern=on 1pt off 2pt]  (0.2,0) circle (0.8cm);
\strand [blue, very thick, only when rendering/.style={
    postaction=decorate,
  },
  decoration={
    markings,
    mark=at position 0.9 with {\arrow{To}}
  }] (-0.3,0.5)--(0.7,-0.5);
\node[right] at (0.7,-0.6) {\tiny C};
\strand [blue, very thick, only when rendering/.style={
    postaction=decorate,
  },
  decoration={
    markings,
    mark=at position 0.9 with {\arrow{To}}
  }] (0.7, 0.5)--(-0.3, -0.5);
\node[right] at (0.7,0.6) {\tiny A};

\end{knot}
\end{scope}

\begin{scope}[shift={(-4,-4)}];
\node at (2.2,2) {$\longrightarrow$};
\filldraw [gray!25] (0.2,0) circle (0.8cm);
\draw [dashed,dash pattern=on 1pt off 2pt]  (0.2,0) circle (0.8cm);
\begin{knot}[
    clip width=5, background color=gray!25!white]

\strand [blue, very thick, only when rendering/.style={
    postaction=decorate,
  },
  decoration={
    markings,
    mark=at position 0.9 with {\arrow{To}}
  }] (0.7, 0.5)--(-0.3, -0.5);
\node[right] at (0.7,-0.6) {\tiny C};
\strand [blue, very thick, only when rendering/.style={
    postaction=decorate,
  },
  decoration={
    markings,
    mark=at position 0.9 with {\arrow{To}}
  }](-0.3,0.5)--(0.7,-0.5);
\node[right] at (0.7,0.6) {\tiny A};
\node[thick] at (-1.2,0) {\small $3$};

\end{knot}
\end{scope}
\end{tikzpicture}

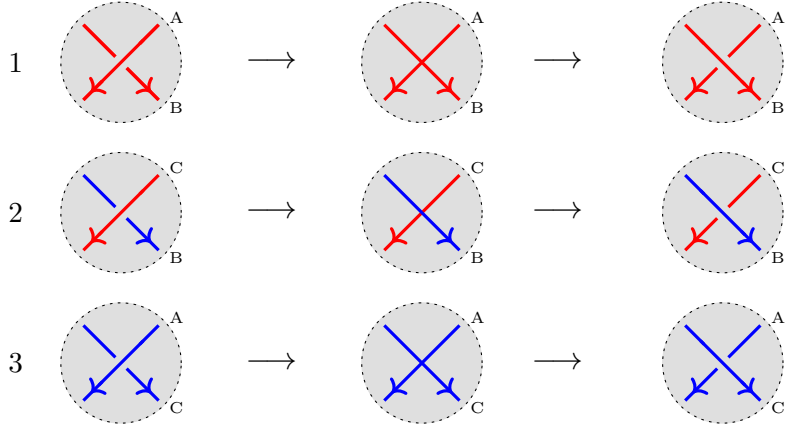
\captionof{figure}{Local picture of the triple point move}
\label{fig:3rd}
\end{center}

For more details, we refer to the proof of Theorem~\ref{inv1}. 
\end{subsection}
\end{section}

\begin{section}{Link invariants}\label{Section3}

We now focus on the link spaces \sloppy $Emb(\underset{k}{\sqcup}S^{2\ell-1},\mathbb{R}^{3\ell})$ and $Emb(\underset{k}{\sqcup}\mathbb{R}^{2\ell-1},\mathbb{R}^{3\ell})$ 
and first discuss the splitting of their sets of connected components for $\ell\geq 2$. 

It is known from \cite[Lemma~3.7]{SONG}, \cite[Lemma~4.8]{RK} that for $q\geq 3$ and any $k\in \mathbb{N}$, we have:  
\begin{equation}\label{iso}
\pi_0Emb(\underset{k}{\sqcup}S^n,\bbR^{n+q})=\pi_0Emb_{\partial}(\underset{k}{\sqcup}\bbR^n,\bbR^{n+q}).
\end{equation}
Moreover, it is an isomorphism of finitely generated abelian groups
 (see \cite{HAE, HAE4}). By~\cite[\S 2.5]{HAE4}, the product on $\pi_0Emb(\underset{k}{\sqcup}S^n,\bbR^{n+q})$ is given by component-wise connected sum along thin tubes. Since the complement is simply connected, the paths of the tubes do not matter.  The product in $\pi_0Emb_{\partial}(\underset{k}{\sqcup}\bbR^n,\bbR^{n+q})$ is induced by the concatenation of links. In addition, there are well-defined tubing homomorphisms 
 \begin{equation}\label{eq:tube}
 \mathcal{T}_{ij}:\pi_0Emb(\underset{k}{\sqcup}S^n,\bbR^{n+q})\rightarrow \pi_0Emb(\underset{k-1}{\sqcup}S^n,\bbR^{n+q}),\,\, 1\leq i< j\leq k,
 \end{equation}
 which merge $i$-th and $j$-th components of a link into one by joining them with a tube.

 \begin{subsection}{Brunnian splitting}
A \textit{Brunnian link} is a non-trivial link that becomes trivial if any of its components is removed. Let $\pi^{br}_0Emb(\underset{k}{\sqcup}S^n,\bbR^{n+q})$ denote the subgroup of $\pi_0Emb(\underset{k}{\sqcup} S^n,\bbR^{n+q})$ consisting of Brunnian links for $k\geq 2$. For $k=1$ we set $\pi_0^{br}Emb(S^n,\bbR^{n+q})=\pi_0Emb(S^n,\bbR^{n+q})$. 
 
 \begin{proposition}\thlabel{split}
 For $q\geq 3$, 
  \begin{equation*}
\pi_0Emb(\underset{k}{\sqcup}S^n,\bbR^{n+q})=\underset{i=1}{\overset{k}{\bigoplus}}\Big(\underset{\binom ki}{\oplus}\pi^{br}_0Emb(\underset{i}{\sqcup}S^n,\bbR^{n+q})\Big)
\end{equation*}

 \end{proposition}
 \begin{proof}
 Consider the homomorphism $$\pi_0Emb(\underset{k}{\sqcup} S^n,\bbR^{n+q})\rightarrow \underset{i=1}{\overset{k}{\bigoplus}}\pi_0Emb(S_i^n,\bbR^{n+q})$$ induced by the maps where we forget all but one component. The kernel of this map is $\pi_0^{U}Emb(\underset{k}{\sqcup} S^n,\bbR^{n+q})$, the subgroup consisting of $k$-component links where each component is unknotted. Moreover, there is a map 
\sloppy
$\underset{i=1}{\overset{k}{\bigoplus}}\pi_0Emb(S_i^n,\bbR^{n+q})\rightarrow \pi_0Emb(\underset{k}{\sqcup} S^n,\bbR^{n+q})$ that creates a $k$-component link where each $S_i^n$ is contained in a small $(n+q)$-ball disjoint from the other balls. Thus, $\underset{i=1}{\overset{k}{\bigoplus}}\pi_0Emb(S_i^n,\bbR^{n+q})$ is a retract of the group $\pi_0Emb(\underset{k}{\sqcup} S^n,\bbR^{n+q})$, and we get the following splitting 
 \begin{equation}\label{splitting}
 \pi_0Emb(\underset{k}{\sqcup}S^n,\bbR^{n+q})=\Big(\underset{i=1}{\overset{k}{\bigoplus}}\pi_0Emb(S_i^n,\bbR^{n+q})\Big)\oplus \pi_0^{U}Emb(\underset{k}{\sqcup} S^n,\bbR^{n+q}).
 \end{equation}

We can further continue the splitting in \eqref{splitting} as follows. For any pair $\{i,j\}\subset\{1,\ldots,k\}$, we consider a map 
\sloppy
$\pi_0^{U}Emb(\underset{k}{\sqcup} S^n,\bbR^{n+q})\rightarrow \pi_0^{U}Emb(\underset{2}{\sqcup} S^n,\bbR^{n+q})=\pi_0^{br}Emb(\underset{2}{\sqcup} S^n,\bbR^{n+q})$ induced by removing all but the $i$-th and $j$-th components. There are $\binom k2$ choices of such maps. Thus, we get a map $\pi_0^{U}Emb(\underset{k}{\sqcup} S^n,\bbR^{n+q})\rightarrow \underset{\binom k2}{\oplus}\pi^{br}_0Emb(\underset{2}{\sqcup}S^n,\bbR^{n+q})$, where the kernel is given by the subgroup of links so that upon removing all but (any) two components, the resulting link is trivial. We can continue to split the obtained group and iteratively doing so we get
$\pi_0^{U}Emb(\underset{k}{\sqcup}S^n,\bbR^{n+q})=\underset{i=2}{\overset{k}{\bigoplus}}\Big(\underset{\binom ki}{\oplus}\pi^{br}_0Emb(\underset{i}{\sqcup}S^n,\bbR^{n+q})\Big)$, which concludes the result.
\end{proof}

\begin{cor}
    For $\ell\geq 2$,
    \begin{equation}
\pi_0Emb(\underset{k}{\sqcup}S^{2\ell-1},\mathbb{R}^{3\ell})=\underset{i=1}{\overset{3}{\bigoplus}}\Big(\underset{\binom 3i}{\oplus}\pi^{br}_0Emb(\underset{i}{\sqcup}S^{2\ell-1},\mathbb{R}^{3\ell})\Big)
\end{equation}
\end{cor}
\begin{proof}
    By \cite[Theorem 9.4]{HAE4}, the subgroup $\pi^{br}_0Emb(\underset{i}{\sqcup}S^{2\ell-1},\bbR^{3\ell})$ is trivial for $i>3$. 
\end{proof}

 In particular, for $2$- and $3$-component links, we obtain the following:
\begin{equation}\label{brsplit2}
\pi_0Emb(\underset{2}{\sqcup}S^{2\ell-1},\bbR^{3\ell})=
\Big(\underset{2}{\oplus}\pi_0Emb(S^{2\ell-1},\bbR^{3\ell})\Big)\oplus \Big( \pi_{0}^{br}Emb(\underset{2}{\sqcup}S^{2\ell-1},\bbR^{3\ell})\Big),
\end{equation}
and 
\begin{multline}\label{br3}
\pi_0Emb(\underset{3}{\sqcup}S^{2\ell-1},\bbR^{3\ell})=\\
\Big(\underset{3}{\oplus}\pi_0Emb(S^{2\ell-1},\bbR^{3\ell})\Big)\oplus \Big(\underset{3}{\oplus} \pi_{0}^{br}Emb(\underset{2}{\sqcup}S^{2\ell-1},\bbR^{3\ell})\Big)\oplus \pi_{0}^{br}Emb(\underset{3}{\sqcup}S^{2\ell-1},\bbR^{3\ell}).
\end{multline}

Haefliger's work provides the following results for each of the summands in~\eqref{br3}. By \cite[Corollary~8.14]{HAE},
 \begin{align*}
\pi_{0} Emb(S^{2\ell-1}, \bbR^{3\ell}) = \left\{ \begin{array}{cc} 
              \bbZ & \hspace{3mm} \ell\geq 2, \ \ell \text{ even,} \\
                \bbZ_2 & \hspace{2.8mm} \ell\geq 3, \ \ell \text{ odd.} 
                              \end{array} \right.
                                           \end{align*}
In the even case,  Theorem~\ref{knotinv} gives a combinatorial formula for this invariant.
                                           
By \cite[Corollary 10.3]{HAE4} and using the homotopy groups of spheres, we obtain                   
  \begin{align*}
 \pi_{0}^{br}Emb(\underset{2}{\sqcup}S^{2\ell-1},\bbR^{3\ell})=  \left\{ \begin{array}{cc} 
  \bbZ^2 & \hspace{3mm} \ell=2, \\
         \bbZ^2\oplus torsion  & \hspace{3mm}   \ell\geq 4 \hspace{2mm}  even , \\
               torsion  & \hspace{2mm} \ell\geq 3 \hspace{2mm} odd. 
                              \end{array} \right.
                                           \end{align*}

The invariants presented in Theorem~\ref{inv2} detect the non-torsion part of $\pi_{0}^{br}Emb(\underset{2}{\sqcup}S^{2\ell-1},\bbR^{3\ell})$ along with generators, that is, only when $\ell$ is even. 

From \cite[Theorem 9.4]{HAE4},                                                 $$ \pi_{0}^{br}Emb(\underset{3}{\sqcup}S^{2\ell-1},\bbR^{3\ell})=\bbZ,$$
and the invariant in Theorem~\ref{inv1} establishes this isomorphism. 

\end{subsection}

\begin{subsection}{Formula for 3-component links}\label{inv3}

Consider a spherical $3$-component link $X\sqcup Y\sqcup Z\hookrightarrow \bbR^{3\ell}$, where each component is a $(2\ell-1)$-sphere, and let $f$ denote its projection to $\bbR^{3\ell-1}$. For each pair $i,j\in\{X, Y, Z\}$, the preimage of their intersection under $f$ consists of the double point sets $L_{ij}^{+}$ and $L_{ij}^{-}$ inside component~$i$, where $L_{ij}^{+}$ corresponds to the overcrossing (where $i$ lies above~$j$), and $L_{ij}^{-}$ corresponds to undercrossing (where $i$ lies below~$j$). Each preimage is of dimension $\ell-1$, and is oriented using the Ekholm orientation when $\ell$ is even, and according to our specified convention when $\ell$ is odd, see Subsection~\ref{orient}.

In the following lemma, for notational simplicity, we denote the images of the components under the projection $f$ in $\bbR^{3\ell-1}$ by the same letters: that is, we write $X$, $Y$, and $Z$ for $f(X)$, $f(Y)$, and $f(Z)$, respectively. 

\begin{lemma}
Let $F:X\sqcup Y\sqcup Z\hookrightarrow \bbR^{3\ell}$ be a spherical $3$-component link, and let $f=p\circ F$ be its generic projection to $\bbR^{3\ell-1}$. Then 
\begin{multline}
    E(f)\coloneqq lk_X(L_{XY},L_{XZ})= lk_Y(L_{YZ},L_{YX})= lk_Z(L_{ZX},L_{ZY})\\
    =lk_{\bbR^{3\ell-1}}(Z,X\cap Y)=lk_{\bbR^{3\ell-1}}(X,Y\cap Z)=lk_{\bbR^{3\ell-1}}(Y,Z\cap X).
\end{multline}
where each $L_{ij}$, $i,j\in\{X,Y,Z\}$, is oriented according to the convention specified in Subsection~\ref{orient} (we use the Ekholm orientation when $\ell$ is even: $L_{ij}=L_{ij}^+\sqcup L_{ij}^-$).
\end{lemma}
\begin{proof} Without loss of generality, we prove that $$lk_X(L_{XY},L_{XZ})=(-1)^{\ell}lk_X(L_{XZ},L_{XY})=lk_{\bbR^{3\ell-1}}(Z,X\cap Y).$$ By definition, $lk_X(L_{XZ},L_{XY})$ is the algebraic intersection number $L_{XZ}\cdot D(L_{XY})$ of $L_{XZ}$ with $D(L_{XY})$, where $D(L_{XY})$ is an $\ell$-chain bounded by $L_{XY}$. On the other hand, $lk_{\bbR^{3\ell-1}}(Z,X\cap Y)$ is $Z\cdot f(D(L_{XY}))$. Since $f(D(L_{XY}))$ is in general position, it intersects $Z$ transversely, and as sets these intersections coincide with those counted in $lk_X(L_{XZ},L_{XY})$. Thus the two numbers differ, if at all, only by a sign.

For $lk_X(L_{XZ},L_{XY})$, the orientation at the intersection is determined by $Or(X\cap Z)$ followed by the complementary orientation inside $X$. At the corresponding intersection point in $Z\cdot f(D(L_{XY}))$, the orientation is given by $Or(Z\cap X)$, followed by the complementary orientation in $Z$, and finally completed by the complementary orientation of $X$ in $\mathbb{R}^{3\ell-1}$. Rearranging these factors introduces a factor of $(-1)^{\ell}$, and with our orientation conventions, this agrees with the sign convention in the definition of $lk_X(L_{XY},L_{XZ})$. Hence, the two intersection numbers coincide, proving the desired equality. The remaining equalities follow by cyclic symmetry.  
\end{proof}

\begin{thm}\thlabel{inv1}
The following formula defines a homomorphism 
\sloppy
$$\mathcal{V}:\pi_{0}Emb(\underset{3}{\sqcup}S^{2\ell-1},\bbR^{3\ell})\rightarrow \bbZ$$ that vanishes on the first two summands of the Brunnian splitting \eqref{br3}, and restricts to an isomorphism on the last summand $ \pi_{0}^{br}Emb(\underset{3}{\sqcup}S^{2\ell-1},\bbR^{3\ell})\approx \bbZ$. Given a representative $F$ of a spherical $3$-component link $X\sqcup Y\sqcup Z\hookrightarrow \bbR^{3\ell}$, with projection $f$ to $\bbR^{3\ell-1}$, the invariant $\mathcal{V}(F)$ is defined by:
\begin{multline}\label{invxyz}
\sloppy
      \mathcal{V}(F) \coloneqq \frac{1}{2}E(f)-\frac{1}{2}\Big[lk_X(L{}_{XY}^{+},L_{XZ}^{-})+lk_X(L_{XY}^{-},L_{XZ}^{+})\\+lk_Y(L_{YZ}^{+},L_{YX}^{-})+lk_Y(L_{YZ}^{-},L_{YX}^{+})+lk_Z(L_{ZX}^{+},L_{ZY}^{-})+lk_Z(L_{ZX}^{-},L_{ZY}^{+})\Big],
         \end{multline}
  where each $L_{ij}^{\pm}$, $i,j\in\{X,Y,Z\}$, is oriented according to the convention specified in Subsection~\ref{orient} (we use the Ekholm orientation when $\ell$ is even).
    
    \end{thm}
 
\begin{remark}\thlabel{gpaction}
The expression in \eqref{invxyz} is equipped with $\bbZ_{2}\times (\bbZ_2\wr S_3)$-symmetry, acting uniformly on each summand. To be precise, the reflection exchanging overcrossings and undercrossings (accounting for the first $\bbZ_2$-factor) acts trivially. The second $\bbZ_2$ arises from orientation reversal of any component, changing the sign of all the linking numbers, and thus multiplies the whole expression by $-1$. The $S_3$-action permuting the components is trivial when $\ell$ is even; and is given by the sign representation of~$S_3$ when $\ell$ is odd.
    
\end{remark}

\begin{example}\thlabel{borromean}\textbf{High-dimensional Borromean link:} Denote the coordinates in $\bbR^{3\ell}$ by $(\mathbf{x,y,z})=(x_1,\ldots,x_{\ell},y_1,\ldots,y_{\ell},z_1,\ldots,z_{\ell})$. The high-dimensional Borromean link $\mathcal{B}$ consists of three spheres defined by the following equations, see Figure~\ref{fig:borB}.

\begin{figure}[ht]
\centering
 \begin{minipage}[c]{.55\textwidth}
        \centering
 \includegraphics[width=0.7\textwidth]{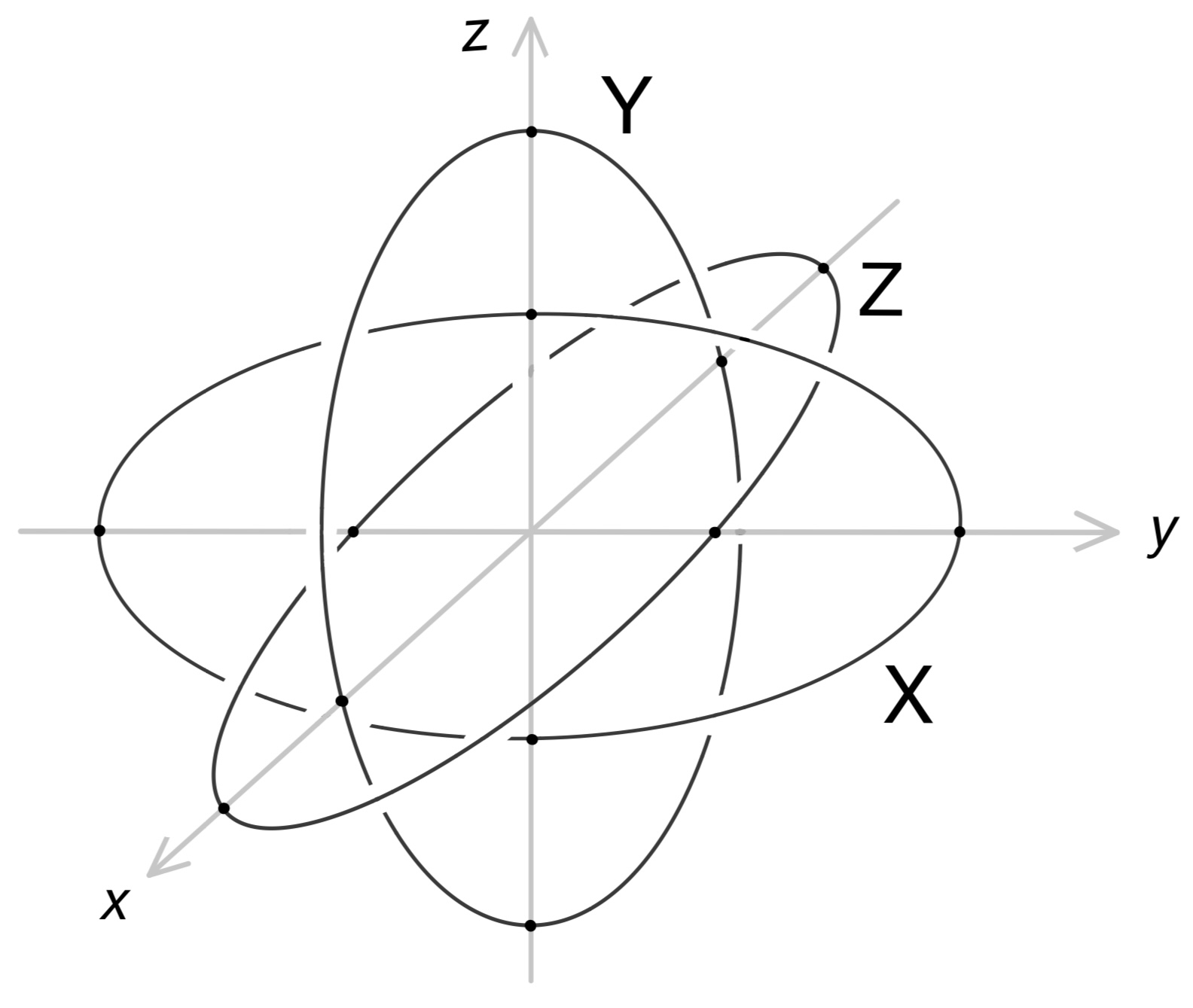}
\end{minipage} 
\begin{minipage}[c]{0.05\textwidth}
        \centering\footnotesize
\begin{align*}
X:\mathbf{x} & =0,\hspace{5mm}   \frac{|\mathbf{y}|^2}{\alpha^2}+\frac{|\mathbf{z}|^2}{\beta^2}=1; & \\
Y:\mathbf{y} & =0,\hspace{5mm}   \frac{|\mathbf{z}|^2}{\alpha^2}+\frac{|\mathbf{x}|^2}{\beta^2}=1; & \\
Z:\mathbf{z} & =0,\hspace{5mm}   \frac{|\mathbf{x}|^2}{\alpha^2}+\frac{|\mathbf{y}|^2}{\beta^2}=1  .
\end{align*}
\end{minipage}
 \caption{High dimensional Borromean link $\mathcal{B}$,  $\alpha>\beta>0$.}
   \label{fig:borB}
 \end{figure}

    Removing any component of $\mathcal{B}$ results in a trivial link, making $\mathcal{B}$ a Brunnian link. Moreover, it is known to be the generator of $\pi_{0}^{br}Emb(\underset{3}{\sqcup}S^{2\ell-1},\bbR^{3\ell})=\bbZ$ (see \cite{MAS}), it also follows from Example~\ref{brac3} and Proposition~\ref{pr:split}.
    
We claim that $\mathcal{V}(\mathcal{B})=-1$.
   For computations, we use an alternative description of $\mathcal{B}$ provided by Sakai in~\cite[\S 3]{SAK1}, which simplifies the analysis of double points in the projection to $\bbR^{3\ell-1}$. Let $\mathcal{B}=X\sqcup Y\sqcup Z\subset \bbR^{3\ell}$ where (fix $\alpha, \beta>0$ so that $2\beta<\alpha$)\begin{center}
\begin{align*}
X: & =\partial\{(\mathbf{0,y,z})\in(\bbR^{\ell})^{\times 3}\mid |\mathbf{y}| \leq \alpha, |\mathbf{z}| \leq \beta\}\simeq S^{2\ell-1}, & \\
Y: & =\partial\{(\mathbf{x,0,z})\in(\bbR^{\ell})^{\times 3}\mid |\mathbf{z}| \leq \alpha, |\mathbf{x}| \leq \beta\}\simeq S^{2\ell-1}, & \\
Z: & =\partial\{(\mathbf{x,y,0})\in(\bbR^{\ell})^{\times 3}\mid |\mathbf{x}| \leq \alpha, |\mathbf{y}| \leq \beta\}\simeq S^{2\ell-1}.
\end{align*}
\end{center}
Smoothing the corners gives smooth $(2\ell-1)$-spheres, each composed of two parts. For example: $$ X=(D^{\ell}_{\mathbf{y}}(\alpha)\times S^{\ell-1}_{\mathbf{z}}(\beta))\sqcup (S^{\ell-1}_{\mathbf{y}}(\alpha)\times D^{\ell}_{\mathbf{z}}(\beta)),$$ where $D^{\ell}_{\mathbf{y}}(\alpha)$ and $D^{\ell}_{\mathbf{z}}(\beta)$ are disks of radius $\alpha$ and $\beta$ in $(0\times\bbR_{\mathbf{y}}^{\ell}\times 0)$ and $(0\times0\times \bbR_{\mathbf{z}}^{\ell})$, respectively, and $S^{\ell-1}_{\mathbf{y}}(\alpha)=\partial D^{\ell}_{\mathbf{y}}(\alpha)$ and $S^{\ell-1}_{\mathbf{z}}(\beta)=\partial D^{\ell}_{\mathbf{z}}(\beta)$. Projecting $\mathcal{B}$ along $p:\bbR^{3\ell}\rightarrow \bbR^{3\ell-1} \times 0$ in the direction of the vector $\mathbf{n}=(1,\ldots,1)\in~\bbR^{3\ell}$ results in double intersections consisting of six disjoint $(\ell-1)$-spheres $A_1\sqcup \ldots \sqcup A_6$. 
\begin{figure}[ht]
        \centering
    \includegraphics[width=0.85\linewidth]{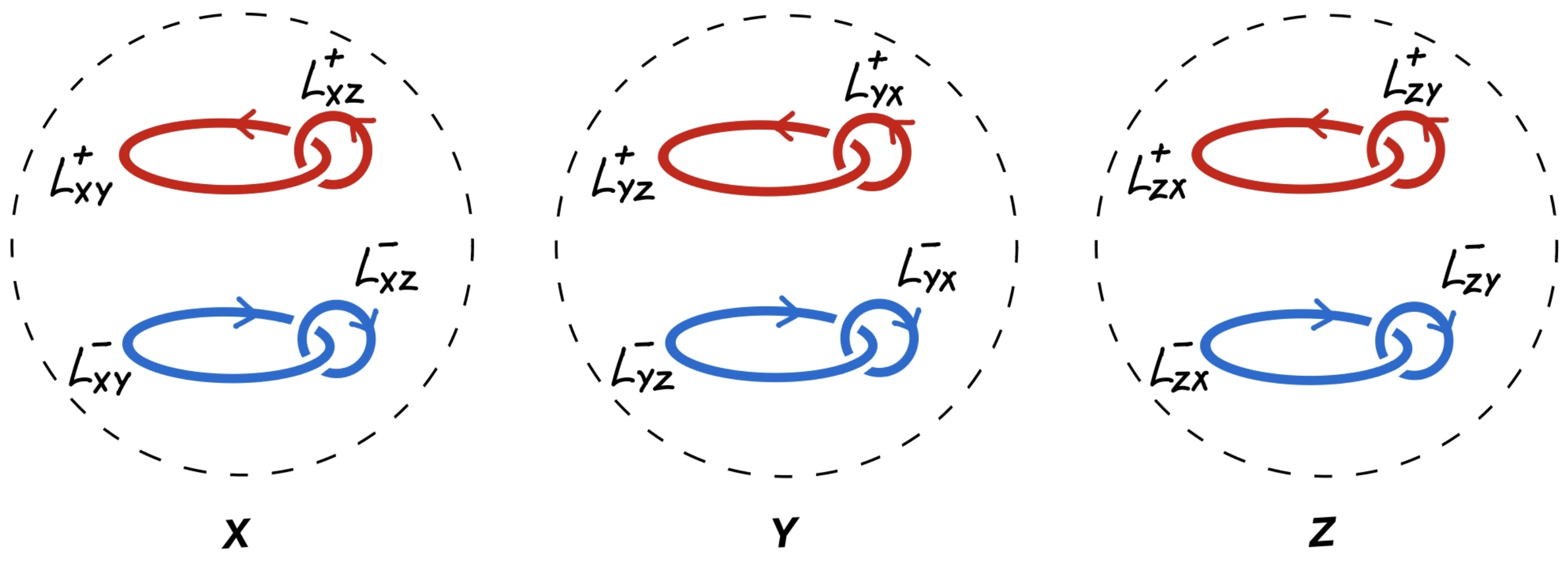}
        \caption{The set of double points of $p\circ\mathcal{B}$}
        \label{fig:6hopf}
    \end{figure}
Moreover, the set of double points contains six disjoint Hopf links, two in each component $X,Y,$ and $Z$, see Figure~\ref{fig:6hopf}. 
 For example, inside $X$\begin{align*}
L_{XY}^+ & =\{(\mathbf{0},\beta'\mathbf{n}',\mathbf{z})\in X\mid |\mathbf{z}|=\beta\},\qquad \hspace{0.9cm}L_{XY}^- =\{(\mathbf{0},-\beta'\mathbf{n}',\mathbf{z})\in X\mid |\mathbf{z}|=\beta\}, \\
L_{XZ}^+ & =\{(\mathbf{0},\mathbf{y},\beta'\mathbf{n}')\in X\mid |\mathbf{y}-\beta'\mathbf{n}'|=\beta\}, \quad L_{XZ}^- =\{(\mathbf{0},\mathbf{y},-\beta'\mathbf{n}')\in X\mid |\mathbf{y}+\beta'\mathbf{n}'|=\beta\}
 \end{align*} form two disjoint Hopf links $L_{XY}^+\sqcup \L_{XZ}^+$ and $L_{XY}^-\sqcup \L_{XZ}^-$. According to our orientation convention, each Hopf link (out of six) has  linking number~$-1$, showing that $\mathcal{V}(\mathcal{B})=\frac 12 E(p\circ\mathcal B)=-1$. For more details, see~\cite[\S 3]{SAK1}.

\end{example}

\begin{example}\thlabel{brac3}\textbf{A link corresponding to configuration spaces:} 
There is a graphing map $G_{*}:\pi_{2\ell-1}^{br}C(3,\bbR^{\ell+1})\rightarrow \pi_0^{br}Emb_{\partial}(\underset{3}{\sqcup}\bbR^{2\ell-1},\bbR^{3\ell})$, see Appendix~\ref{appen}. From \cite{HIL}, it is known that $\pi_{2\ell-1}^{br}C(3,\bbR^{\ell+1})$ is isomorphic to $\bbZ$, and the generator is given by the Whitehead bracket $[\alpha_{12},\alpha_{23}]\in \pi_{2\ell-1}C(3,\bbR^{\ell+1})$, where $\alpha_{12},\alpha_{23}\in\pi_{\ell}C(3,\bbR^{\ell+1})$ are classes corresponding to braids in which strand $2$ loops around strand $1$ or strand $3$, respectively. Proposition~\ref{pr:split} implies that $G_*$ is an isomorphism. We claim that under $G_*$, the generator $[\alpha_{12},\alpha_{23}]$ gets mapped to $-\mathcal{B}$.  
 The set of double points of $G_*([\alpha_{12},\alpha_{23}])$ when projected to $\bbR^{3\ell-1}$ along some projection $p$ is depicted in the following picture, see~\cite[\S~3.7.1]{NG1} for more details:

\begin{figure}[ht]
        \centering
    \includegraphics[width=0.85\linewidth]{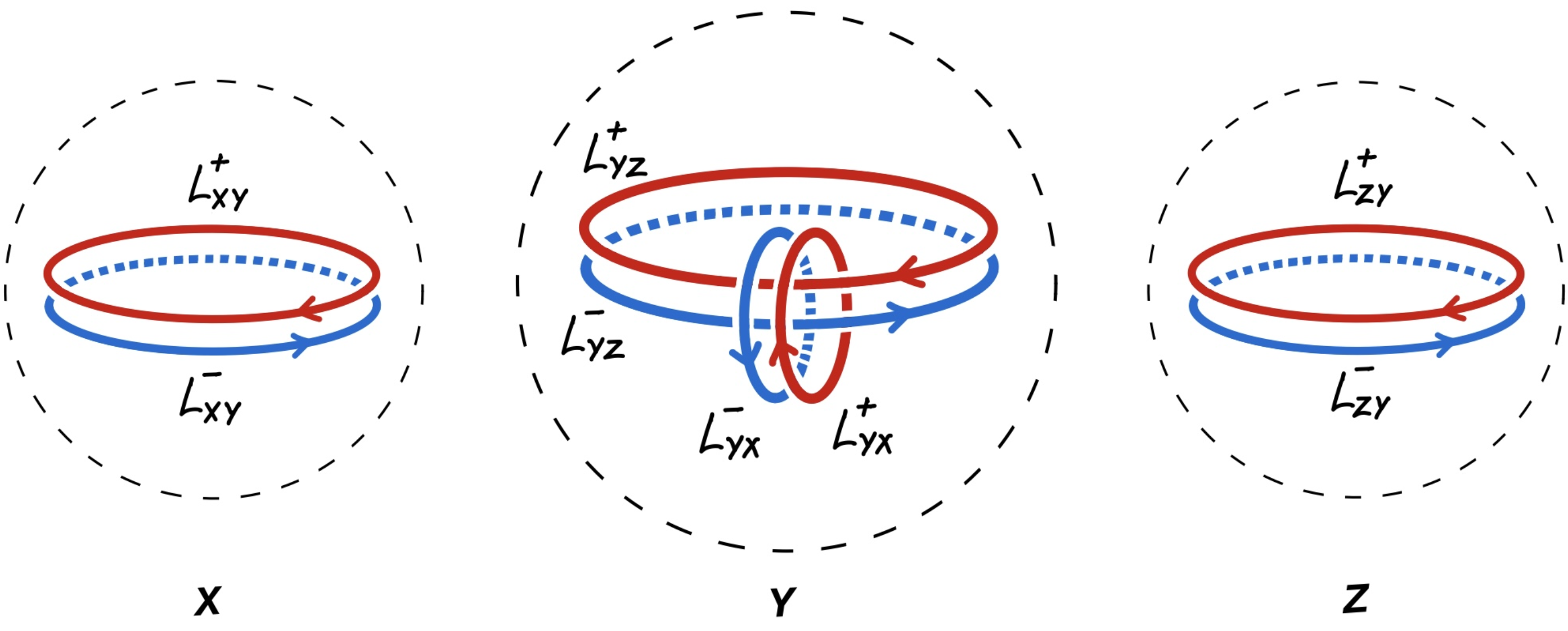}
        \caption{The set of double points of $p\circ G_*([\alpha_{12},\alpha_{23}])$}
        \label{fig:G}
    \end{figure}
 



From this figure we can see that $E(p\circ G_*[\alpha_{12},\alpha_{23}])=lk_X(L_{XY},L_{XZ})=0$, while $lk_Y(L_{YZ}^+,L_{YX}^-)=lk_Y(L_{YZ}^-,L_{YX}^+)=-1$, implying $\mathcal V(G_*[\alpha_{12},\alpha_{23}])=1=-\mathcal V(\mathcal B)$. 
 
\end{example}

\begin{proof}[Proof of Theorem \ref{inv1}]
    The well-definedness of the formula follows from its invariance under isotopy. Since embedded cobordism moves do not affect linking numbers, the expression is preserved under such moves. However, the higher-dimensional third Reidemeister move can change linking numbers, where three sheets intersect at a triple point. If two
    or three sheets belong to the same component (e.g. two from $X$ and one from $Y$), the invariant $\mathcal{V}(F)$ is unchanged. However, if the three sheets come from distinct components, then moving one of them changes the linking number between this component and the intersection of the other two by $\pm 1$, so $E(f)$ changes by $\pm 1$. At the same time, the over/undercrossing data at the triple point change by exactly one crossing (see Figure~\ref{fig:3rd}), contributing an additional $\pm 1$. We claim that these 
    two contributions cancel each other, leaving $\mathcal{V}(F)$ invariant.
    
    To verify full invariance, we must examine how $\mathcal{V}(F)$ behaves under changes in the height ordering of the components $X, Y, Z$ at a triple point, as well as under orientation reversal. Since the formula is symmetric by Remark~\ref{gpaction}, it suffices to show invariance in a single case: a triple point where component $X$ lies above $Y$, and $Y$ lies above $Z$. The following is a local model for this situation considered as a projection to $\bbR^{3\ell-1}$. 
    
Let $\{F_t\}$ be a $1$-parameter family of embeddings $X\sqcup Y\sqcup Z\hookrightarrow \bbR^{3\ell}$ for $t\in (-\epsilon, \epsilon)$, such that their projection along the last coordinate yields a generic family $\{f_t$\} of immersions $X\sqcup Y\sqcup Z\looparrowright \bbR^{3\ell-1}$. 
Suppose that $X$, $Y$, and $Z$ lie at constant heights $a$, $b$, and $c$, respectively, with $a>b>c$. We assume that at $t=0$, the image $f_0$ has a triple point at the origin in $ \bbR^{3\ell-1}$.
We choose following positively oriented local coordinates: 
\begin{align*}
\text{ On } X:& \quad (x,y_1,\ldots,y_{\ell-1},z_1,\ldots z_{\ell-1}), \\
\text{ On } Y:& \quad(y,z_1,\ldots,z_{\ell-1},x_1,\ldots x_{\ell-1}), \\
\text{ On } Z:& \quad(z,x_1,\ldots,x_{\ell-1},y_1,\ldots, y_{\ell-1}), \\
\text{ On } \bbR^{3\ell-1}:& \quad(x,y,x_1,\ldots,x_{\ell-1},y_1,\ldots,y_{\ell-1},z_1,\ldots,z_{\ell-1}).
 \end{align*}
 
Near the triple point in $ \bbR^{3\ell-1}$, we model $f_t$ as:
\begin{align*}
f_t|_{X} & =(x,0,\underbrace{0,\ldots,0}_{\ell-1},y_1,\ldots,y_{\ell-1},z_1,\ldots,z_{\ell-1}),\\
f_t|_{Y} & =(t,y,x_1,\ldots,x_{\ell-1},\underbrace{0,\ldots,0}_{\ell-1},z_1,\ldots,z_{\ell-1}),\\
f_t|_{Z} & =(z,z,x_1,\ldots,x_{\ell-1},y_1,\ldots, y_{\ell-1},\underbrace{0,\ldots,0}_{\ell-1}).
\end{align*}

Note that only $Y$ is moved by this isotopy. The $(\ell-1)$-dimensional double intersections of $f_t$ are locally given by:
\begin{align*}
f_t|_{X}\cap f_t|_{Y} & =(t,0,\underbrace{0,\ldots,0}_{\ell-1},\underbrace{0,\ldots,0}_{\ell-1},s_1,\ldots, s_{\ell-1}),\\
f_t|_{Z}\cap f_t|_{X} & =(0,0,\underbrace{0,\ldots,0}_{\ell-1},\tilde{s}_1,\ldots,\tilde{s}_{\ell-1},\underbrace{0,\ldots,0}_{\ell-1}),\\
f_t|_{Y}\cap f_t|_{Z} & =(t,t,\doubletilde{s}_1,\ldots,\doubletilde{s}_{\ell-1},\underbrace{0,\ldots,0}_{\ell-1},\underbrace{0,\ldots,0}_{\ell-1}).
\end{align*}

According to the orientation convention in Subsection~\ref{orient}, if
$(e_1,e_2,\ldots,e_{3\ell-1})$ denotes the basis for $\bbR^{3\ell-1}$ in the new coordinates, then the intersections are oriented so that
$$(\underbrace{-e_{2\ell+1}, e_{2\ell+2},\ldots,e_{3\ell-1}}_{Or(f_t|_{Y}\cap f_t|_{X})}, \underbrace{(-1)^{\ell}e_2,e_3,\ldots,e_{\ell+1}}_{\text{rest of }Or(Y)},\underbrace{-e_1,e_{\ell+2},\ldots, e_{2\ell}}_{\text{rest of } Or(X)}),$$ 
$$(\underbrace{(-1)^{\ell}e_{\ell+2},\ldots,e_{2\ell}}_{Or(f_t|_{Z}\cap f_t|_{X})}, \underbrace{(-1)^{\ell}(e_1+e_2),e_3,\ldots,e_{\ell+1}}_{\text{rest of } Or(Z)},\underbrace{-e_1,e_{2\ell+1},\ldots,e_{3\ell-1}}_{\text{rest of } Or(X)}),$$
and
$$(\underbrace{(-1)^{\ell}e_3,e_4,\ldots,e_{\ell+1}}_{Or(f_t|_{Y}\cap f_t|_{Z}))},\underbrace{(-1)^{\ell}e_2,e_{2\ell+1},\ldots,e_{3\ell-1}}_{\text{rest of } Or(Y)}, \underbrace{-(e_1+e_2),e_{\ell+2},\ldots,e_{2\ell}}_{\text{rest of } Or(Z)}),$$ 
give the positive orientation on $\bbR^{3\ell-1}$.
 
To compute the change in $E(f)$, we analyze how the linking number $lk_{\bbR^{3\ell-1}}(Y,Z\cap X)=lk(f_t|_{Y}, f_t|_{Z}\cap f_t|_{X})$ changes as $t$ crosses $0\in \bbR^{3\ell-1}$. We claim that this linking number jumps by
 $(-1)^{\ell}$ when we pass a triple point. Let $f_t|_{Z}\cap f_t|_{X}$ bound a singular $\ell$-chain $D(ZX)$ in $\bbR^{3\ell-1}$, given near the intersection by $$\{(r,0,\underbrace{0,\ldots,0}_{\ell-1},\tilde{s}_1,\ldots,\tilde{s}_{\ell-1},\underbrace{0,\ldots,0}_{\ell-1})|r\leq 0\}.$$ For $t=-\delta$, where $\delta$  is a small positive number, this chain intersects $f_t|_{Y}$ transversely at a single point 
 $(-\delta,\underbrace{0,\ldots,0}_{3\ell-2})$. The orientation at the intersection is given by: $$\big(\underbrace{e_2,e_{2\ell+1},\ldots,e_{3\ell-1},e_3,\ldots,e_{\ell+1}}_{Or(f_t|_{Y})}, \underbrace{e_1, (-1)^{\ell}e_{\ell+2},e_{\ell+3},\ldots,e_{2\ell}}_{Or(D(ZX))}\big),$$ which contributes $(-1)^{\ell-1}$ to the linking number. For $t=\delta$, there is no intersection, and hence the jump is $\Delta E(f)=\Delta lk_{\bbR^{3\ell-1}}(Y,Z\cap X)=0-(-1)^{\ell-1}=(-1)^{\ell}$.

For $t=-\delta$, the preimages of the double intersection of $f_t$ in each component are given by the following three local pictures:
 \begin{center}
\begin{tikzpicture}[scale=0.8]

\begin{knot}[
    clip width=8, consider self intersections=true,
  ignore endpoint intersections=false, flip crossing=3
    ]
 \strand [blue, very thick, only when rendering/.style={
    postaction=decorate,
  },
  decoration={
    markings,
    mark=at position 0.9 with {\arrow{To}}
  }] (1.4, 1)--(-1.4, -0.2);

         \strand [red, very thick, only when rendering/.style={
    postaction=decorate,
  },
  decoration={
    markings,
    mark=at position 0.9 with {\arrow{To}}
  }] (0, 1.4)--(0, -0.7);
        \node[left] at (0,1.7) {$L_{YZ}^{+}$};
        \node[right] at (1.5,1.1) {$L_{YX }^{-}$};
         \node at (0,-1.2) {$\mathbf{Y}$};
        
\end{knot}    

\begin{scope}[shift={(-4.8,0)}];
\begin{knot}[
    clip width=8, consider self intersections=true,
  ignore endpoint intersections=false, flip crossing=3
    ]
 \strand [red, very thick, only when rendering/.style={
    postaction=decorate,
  },
  decoration={
    markings,
    mark=at position 0.9 with {\arrow{To}}
  }] (1.4, 1)--(-1.4, -0.2);

         \strand [red, very thick, only when rendering/.style={
    postaction=decorate,
  },
  decoration={
    markings,
    mark=at position 0.9 with {\arrow{To}}
  }] (0, 1.4)--(0, -0.7);
        \node[left] at (0,1.7) {$L_{XY}^{+}$};
        \node[right] at (1.5,1.1) {$L_{XZ }^{+}$};
               \node at (0,-1.2) {$\mathbf{X}$};
        
\end{knot}  
\end{scope}      

\begin{scope}[shift={(4.8,0)}];
\begin{knot}[
    clip width=8, consider self intersections=true,
  ignore endpoint intersections=false, flip crossing=3
    ]
 \strand [blue, very thick, only when rendering/.style={
    postaction=decorate,
  },
  decoration={
    markings,
    mark=at position 0.9 with {\arrow{To}}
  }] (1.4, 1)--(-1.4, -0.2);

         \strand [blue, very thick, only when rendering/.style={
    postaction=decorate,
  },
  decoration={
    markings,
    mark=at position 0.9 with {\arrow{To}}
  }] (0, 1.4)--(0, -0.7);
        \node[left] at (0,1.7) {$L_{ZY}^{-}$};
        \node[right] at (1.5,1.1) {$L_{ZX }^{-}$};
               \node at (0,-1.2) {$\mathbf{Z}$};

\end{knot}  
\end{scope}
\end{tikzpicture}

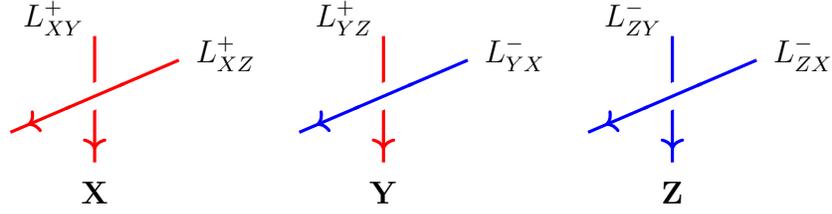
\captionof{figure}{The set of double points near the triple point, for $t<0$}
\end{center}        

Note that now we only care about the contribution of $lk(L_{YZ}^{+},L_{YX}^{-})$ in $Y$, where 
\begin{align*}
\hspace{0.3cm}  L^{-}_{YX} & =(0,-s_1,s_2,\ldots, s_{\ell-1},\underbrace{0,\ldots,0}_{\ell-1}), \\
  L^{+}_{YZ}(t) & = (t,\underbrace{0,\ldots,0}_{\ell-1},(-1)^{\ell}\doubletilde{s}_1,\doubletilde{s}_2,\ldots,\doubletilde{s}_{\ell-1}),
  \end{align*}
are the preimages in $Y$, oriented according to our convention in Subsection~\ref{orient}.
 Assume $L_{YX}^{-}$ is a boundary of an $\ell$-chain \sloppy$D(L_{YX}^{-})=\{(r,-s_1,s_2,\ldots, s_{\ell-1},0,\ldots,0)|r\leq 0\}$ that intersect $L_{YZ}^{+}$ exactly once when $t=-\delta$, see Figure~\ref{fig:Ychange}. At the intersection point, the orientation
\begin{footnotesize}
$$\Big(Or(L_{YZ}^{+}), \frac{\partial}{\partial y}, Or(L_{YX}^{-})\Big)=\Big((-1)^{\ell}\frac{\partial}{\partial x_1},\frac{\partial}{\partial x_2},\ldots,\frac{\partial}{\partial x_{\ell-1}}, \frac{\partial}{\partial y}, -\frac{\partial}{\partial z_1},\frac{\partial}{\partial z_2},\ldots,\frac{\partial}{\partial z_{\ell-1}}\Big),$$
\end{footnotesize}yields a contribution of $(-1)^{\ell-1}$ to $lk(L_{YZ}^{+},L_{YX}^{-})$. Since the intersection vanishes at $t=\delta$, we have $\Delta lk(L_{YZ}^{+},L_{YX}^{-})=(-1)^{\ell}$.

 \begin{center}
\begin{tikzpicture}

\begin{knot}[
    clip width=9, consider self intersections=true,
  ignore endpoint intersections=false, flip crossing=3
    ]
 \strand [blue, very thick, only when rendering/.style={
    postaction=decorate,
  },
  decoration={
    markings,
    mark=at position 0.9 with {\arrow{To}}
  }] (1.45, 0.9)--(-1.45, -0.3);
                         \strand [blue, dashed, dash pattern=on 1pt off 2pt, very thick] (0.4, 0.9)--(-1.5, 0.1);

         \strand [red, very thick, only when rendering/.style={
    postaction=decorate,
  },
  decoration={
    markings,
    mark=at position 0.9 with {\arrow{To}}
  }] (0, 1.4)--(0, -0.7);
                                  \strand [blue, dashed, dash pattern=on 1pt off 2pt, very thick] (0.4, 0.9)--(1.4, 0.9);
                                                   \strand [blue, dashed, dash pattern=on 1pt off 2pt, very thick] (-1.5,0.1)--(-0.7, 0.1);
        \node[left] at (0,1.7) {$L_{YZ}^{+}$};
        \node[right] at (1.5,1.1) {$L_{YX }^{-}$};
              \node[left] at (0.5,-1.5) {$t<0$};
\end{knot}        

\begin{scope}[shift={(6,0)}];
\begin{knot}[
    clip width=8, consider self intersections=true,
  ignore endpoint intersections=false, flip crossing=3
    ]
         \strand [red, very thick, only when rendering/.style={
    postaction=decorate,
  },
  decoration={
    markings,
    mark=at position 0.9 with {\arrow{To}}
  }] (0, 1.4)--(0, -0.7);
 \strand [blue, very thick, only when rendering/.style={
    postaction=decorate,
  },
  decoration={
    markings,
    mark=at position 0.9 with {\arrow{To}}
  }] (1.4, 1)--(-1.4, -0.2);
    
                 \strand [blue, dashed, dash pattern=on 1pt off 2pt, very thick] (0.4, 1)--(-1.5, 0.2);
                                  \strand [blue, dashed, dash pattern=on 1pt off 2pt, very thick] (0.4, 1)--(1.4, 1);
                                                   \strand [blue, dashed, dash pattern=on 1pt off 2pt, very thick] (-1.5,0.2)--(-0.7, 0.2);
        \node[left] at (0,1.7) {$L_{YZ}^{+}$};
        \node[right] at (1.5,1.1) {$L_{YX }^{-}$};
                      \node[left] at (0.5,-1.5) {$t>0$};

\end{knot}     
\end{scope}   
\end{tikzpicture}

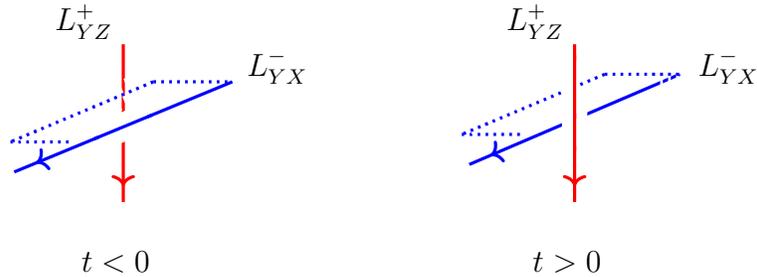
\captionof{figure}{$D(L_{YX}^-)$ intersected with $L_{YZ}^+$}
\label{fig:Ychange}
\end{center} 

Thus, the overall change in the formula is:
\begin{align*}
\Delta\mathcal{V}(F) & = \mathcal{V}(F_{\delta})-\mathcal{V}(F_{-\delta})\\
&=\Delta E(f)-\Delta lk(L_{YZ}^{+},L_{YX}^{-})\\
& =(-1)^{\ell}-(-1)^{\ell}\\
&=0.
 \end{align*}
Using Examples~\ref{borromean} and \ref{brac3}, we conclude that $\mathcal{V}:\pi_0^{br}Emb(\underset{3}{\sqcup}S^{2\ell-1},\bbR^{3\ell})\rightarrow \bbZ$ is an isomorphism. Indeed, it follows from the facts that $\mathcal{B}$ is a generator of $\pi_0^{br}Emb(\underset{3}{\sqcup}S^{2\ell-1},\bbR^{3\ell})$ and $\mathcal{V}(\mathcal{B})=-1$.
  
\end{proof}

 \end{subsection}

\begin{subsection}{The invariants for 2-component links}\label{2complinks}

We now set $\ell=2k$. Consider a $2$-component link $X\sqcup Y\hookrightarrow \bbR^{6k}$, where each component is a $(4k-1)$-sphere. Under projection to $\bbR^{6k-1}$, the set of double points is $(2k-1)$-dimensional, oriented according to the adjusted Ekholm orientation (see Subsection~\ref{orient}). 

 \begin{thm}\thlabel{inv2}
The following formulas give a homomorphism 
$$\mathcal{W}=(\mathcal{W}_1,\mathcal{W}_2):\pi_{0}Emb(\underset{2}{\sqcup}S^{4k-1},\bbR^{6k})\rightarrow \bbQ^2,$$ which vanishes on the first summand of the Brunnian splitting \eqref{brsplit2}, and is rationally an isomorphism on the second summand. Given a representative $F$ of a spherical $2$-component link $X\sqcup Y\hookrightarrow \bbR^{6k}$, we assign two invariants:
\begin{align}\label{eq:W1}
\mathcal{W}_1(F) & \coloneqq\frac{1}{2} lk_X(L_{XX}, L_{XY})-lk_Y(L_{YX}^{+},L_{YX}^{-});\\ \label{eq:W2} 
\mathcal{W}_2(F) & \coloneqq\frac{1}{2} lk_Y(L_{YY}, L_{YX})-lk_X(L_{XY}^{+},L_{XY}^{-}),
\end{align}
where each $L_{ij}$ is taken with the adjusted Ekholm orientation, that is, $L_{ij}=L_{ij}^+\sqcup -L_{ij}^-$ for $i,j=X,Y$.
Moreover, the image of homomorphism $\mathcal{W}$ is
 \begin{align*}
 \left\{ \begin{array}{cc} 
  \{(a,b)\in\bbZ^2|a+b\in2\bbZ\} & \hspace{2.8mm} k=1,2,4\\
              2\bbZ\times 2\bbZ & \hspace{3mm} \text{otherwise.} 
          \end{array} \right.
    \end{align*}

\end{thm}

\begin{example}\thlabel{generator}
There are two linearly independent elements in $\pi_{0}^{br}Emb(\underset{2}{\sqcup}S^{4k-1},\bbR^{6k})$ constructed using the Borromean link $\mathcal{B}=X\sqcup Y\sqcup Z$. Specifically, we relabel the components of $\mathcal B$ as $X_1\sqcup X_2\sqcup Y$ and $X\sqcup Y_1\sqcup Y_2$, repectively, and then form new $2$-component links
$$\mathcal{B}_1=(X_1\# X_2)\sqcup Y \text{  and  } \mathcal{B}_2=X\sqcup (Y_1\# Y_2),$$ where $\#$ denotes the connected sum along a thin tube. Since the complement of $\mathcal{B}$ is simply connected, the isotopy class of the resulting link is independent of the chosen path for tubing. Moreover, it is easy to compute that $\mathcal{W}(\mathcal{B}_1)=(-2,0)$ and $\mathcal{W}(\mathcal{B}_2)=(0,-2)$. 
Consider, for example $\mathcal B_1$.  
Its set of double points of the projection 
is given in Figure~\ref{fig:B_1}. 
Each of the six Hopf links have linking number $-1$,
 see Example~\ref{borromean}, which implies 
 $\frac{1}{2} lk_X(L_{XX}, L_{XY})=-2$, while
 $lk_Y(L_{YX}^{+},L_{YX}^{-})=\frac{1}{2} lk_Y(L_{YY}, L_{YX})=lk_X(L_{XY}^{+},L_{XY}^{-})=0$.
 
\begin{figure}[ht]
\centering
 \includegraphics[width=0.92\textwidth]{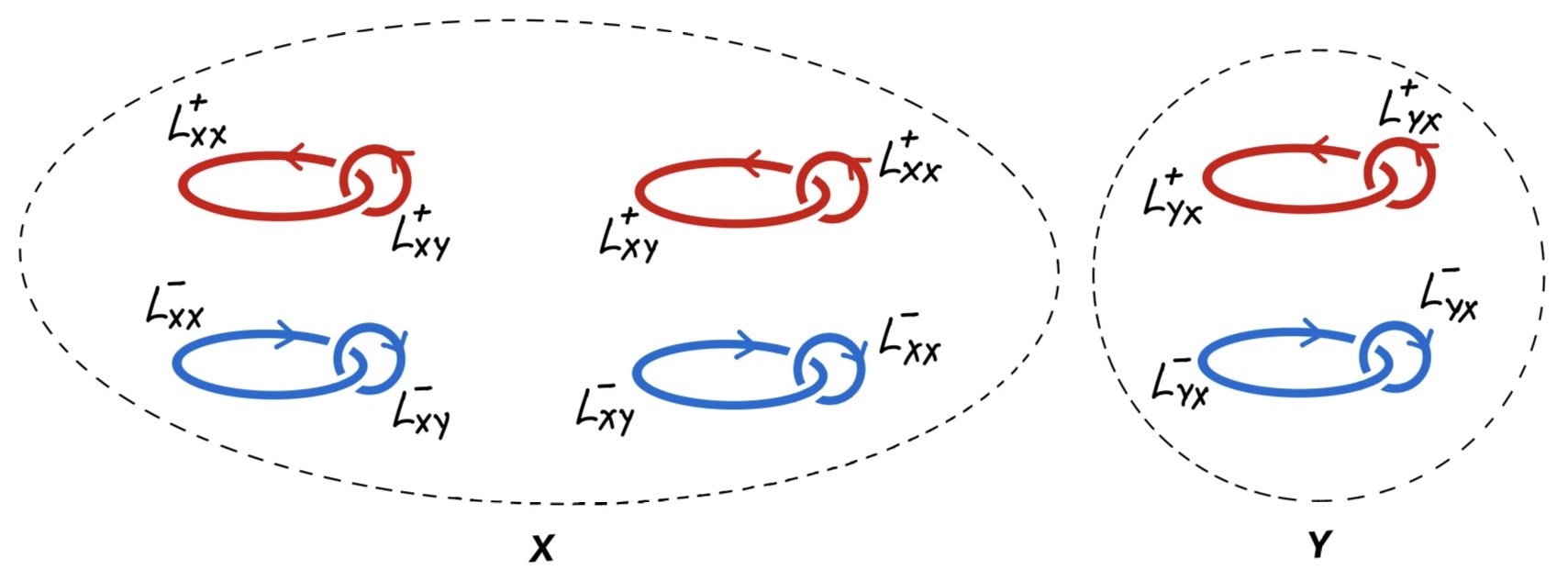}
 \caption{The set of double points of $p\circ \mathcal{B}_1$}
 \label{fig:B_1}
 \end{figure}


\end{example}

\begin{example}\thlabel{brac2}
Consider $C(2,\bbR^{2k+1})\simeq S^{2k}$, and the element $[\alpha_{12},\alpha_{12}]\in\pi_{4k-1}C(2,\bbR^{2k+1})$ corresponding to the (rational) generator $[id, id]$ of the $\bbZ$-summand of $\pi_{4k-1}S^{2k}$, where $[id,id]$ is the Whitehead bracket of the identity map of $S^{2k}$ with itself. we claim that the graphing map (see Appendix~\ref{appen}) \sloppy $$G_{*}:\pi_{4k-1}S^{2k}\rightarrow \pi_0Emb_{\partial}(\underset{2}{\sqcup}\bbR^{4k-1},\bbR^{6k})$$
sends $[id,id]$ to $-\mathcal{B}_1-\mathcal{B}_2$, up to a torsion element. (This statement  also appeared in \cite{RK}.) Under a projection $p$ to $\bbR^{6k-1}$, the set of double points of $G_*([id,id])$ is depicted in Figure~\ref{fig:brac2}, see
\cite[Section~3.7.2]{NG1} for details. It implies $\mathcal W (G_*([\alpha_{12},\alpha_{12}))=(2,2)$ since  $\frac{1}{2} lk_X(L_{XX}, L_{XY})=
\frac{1}{2} lk_Y(L_{YY}, L_{YX})=0$ and
$lk_Y(L_{YX}^{+},L_{YX}^{-})=lk_X(L_{XY}^{+},L_{XY}^{-})=2$.
\begin{figure}[h]
\centering
 \includegraphics[width=0.58\textwidth]{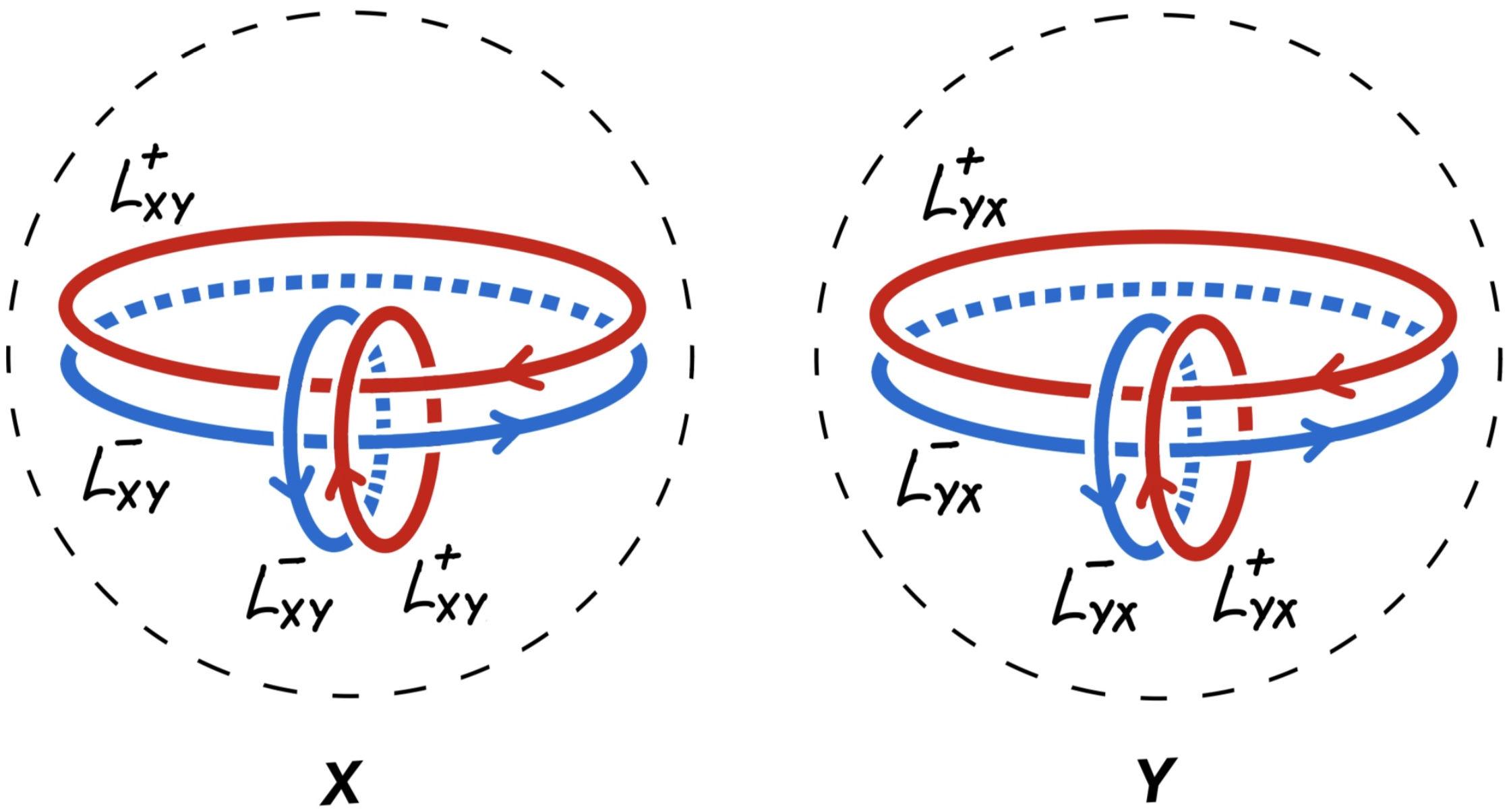}
 \caption{The set of double points of $p\circ G_*([id,id])$}
 \label{fig:brac2}
 \end{figure}

For $k=1,2,4$, the generator of the $\bbZ$ summand of $\pi_{4k-1}S^{2k}$ corresponds to the Hopf fibration $Hopf:S^{4k-1}\rightarrow S^{2k}$, which under $G_{*}$ gets mapped to $\frac{1}{2}(-\mathcal{B}_1-\mathcal{B}_2)$, up to a torsion element. The set of double points of $G_*([Hopf])$ under a projection is shown in Figure~\ref{fig:hopf}, see \cite[Remark~3.7.1]{NG1}.
\begin{figure}[h]
\centering
 \includegraphics[width=0.55\textwidth]{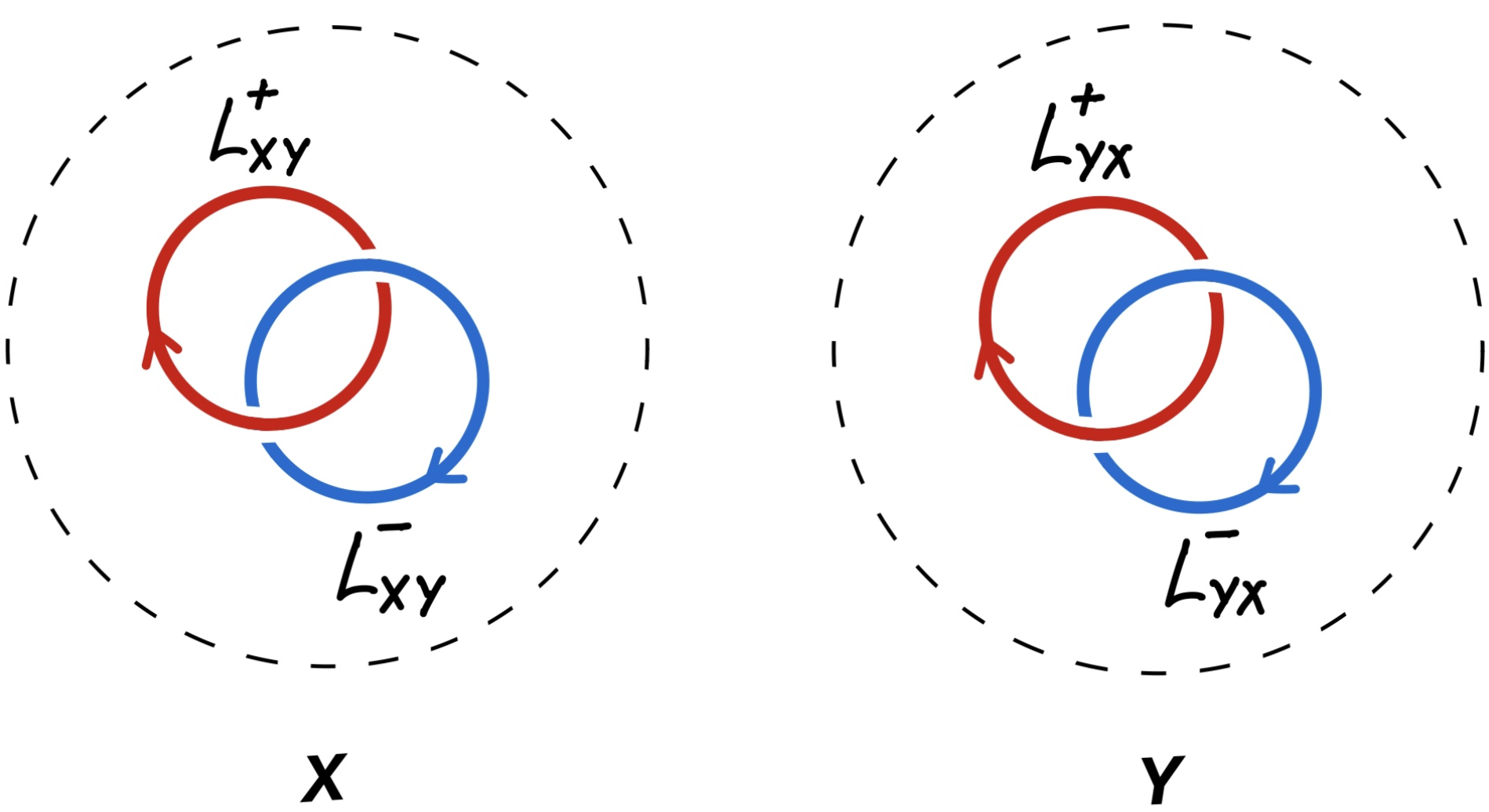}
 \caption{The set of double points of $p\circ G_*([Hopf])$}
 \label{fig:hopf}
 \end{figure}
    
\end{example}

\begin{proof}[Proof of Theorem \ref{inv2}]
The formulas for $\mathcal{W}_1$ and $\mathcal{W}_2$ are expressed in terms of linking numbers that involve self-intersections of the components $X$ and $Y$, possibly including Whitney umbrella points. These linking numbers are unaffected by cobordism moves. The only potential variation occurs under the third Reidemeister move. Since $\mathcal{W}_1$ and $\mathcal{W}_2$ are structurally similar, it suffices to verify the invariance of $\mathcal{W}_1$. The only relevant case is when two sheets of $X$, say $X_1$ and $X_2$, and one sheet of $Y$ meet in a triple point. There are three subcases depending on whether $Y$ lies between, above or below the two sheets of $X$. Note that the second summand of  $\mathcal{W}_1$ changes only 
in the first situation i.e., when $Y$ is in between.

 We first check the case when at the triple point, $X_1$ lies above $X_2$, and $X_2$ lies above $Y$ ($X_1>X_2>Y$ for notation).  Since $\mathcal{W}_1$ is unchanged under mirror symmetry, it is also invariant when $Y>X_1>X_2$.  
  The following is a local model for this situation considered as a projection to $\bbR^{6k-1}$.
Consider a generic $1$-parameter family $\{f_t$\} of immersions, for $t\in (-\epsilon, \epsilon)$, of some link $X\sqcup Y\hookrightarrow \bbR^{6k}$ projected to $\bbR^{6k-1}$ along the last coordinate. Suppose a triple point occurs at $t=0$, located at the origin in $\bbR^{6k-1}$. Since $X_1>X_2>Y$, we may assume that their last coordinates are constant values $a>b>c$, respectively. One can find the following local coordinates 
$$(x_1,x^2_1,\ldots,x^2_{2k-1},y_1,\ldots y_{2k-1})$$
$$(x_2,y_1,\ldots,y_{2k-1},x_1^1,\ldots x^1_{2k-1})$$
$$(y,x^1_1,\ldots,x^1_{2k-1},x^2_1,\ldots, x^2_{2k-1})$$
$$(x_1,x_2,x^1_1,\ldots,x^1_{2k-1},x^2_1,\ldots,x^2_{2k-1},y_1,\ldots,y_{2k-1})$$
on $X_1, X_2, Y,$ and $\bbR^{6k-1}$, respectively, where all coordinate systems are assumed to be positively oriented. Near the triple point $0\in \bbR^{6k-1}$, we have
\begin{align*}
f_t|_{X_1} & =(x_1,0,\underbrace{0,\ldots,0}_{2k-1},x^2_1,\ldots,x^2_{2k-1},y_1,\ldots,y_{2k-1}),\\
f_t|_{X_2} & =(t,x_2,x^1_1,\ldots,x^1_{2k-1},\underbrace{0,\ldots,0}_{2k-1},y_1,\ldots,y_{2k-1}),\\
f_t|_{Y} & =(y,y,x^1_1,\ldots,x^1_{2k-1},x^2_1,\ldots, x^2_{2k-1},\underbrace{0,\ldots,0}_{2k-1}).
\end{align*}
The double intersections of $f_t$ are given by $(2k-1)$-dimensional oriented manifolds as follows:

\begin{align*}
f_t|_{X_1}\cap f_t|_{X_2} & =(t,0,\underbrace{0,\ldots,0}_{2k-1},\underbrace{0,\ldots,0}_{2k-1},-s_1,\ldots, s_{2k-1}),\\
f_t|_{X_1}\cap f_t|_{Y} & =(0,0,\underbrace{0,\ldots,0}_{2k-1},\tilde{s}_1,\ldots,\tilde{s}_{2k-1},\underbrace{0,\ldots,0}_{2k-1}),\\
f_t|_{X_2}\cap f_t|_{Y} & =(t,t,\doubletilde{s}_1,\ldots,\doubletilde{s}_{2k-1},\underbrace{0,\ldots,0}_{2k-1},\underbrace{0,\ldots,0}_{2k-1}).
\end{align*}
Note that the only linking numbers that contribute to the change of $\mathcal{W}_1$ are $lk(L_{X_1X_2}^{+},L_{X_1Y}^{+})$ and $lk(L_{X_2X_1}^{-},L_{X_2Y}^{+})$ in $X_1$ and $X_2$, respectively.  
The respective preimages of the double intersection of $f_t$ are as follows: 
\begin{align*}
 \text{In } X_1:\hspace{0.3cm} L^{+}_{X_1X_2}(t) & =(t,\underbrace{0,\ldots,0}_{2k-1},-s_1,s_2,\ldots, s_{2k-1}), \\
 L^{+}_{X_1Y} & = (0,\tilde{s}_1,\ldots,\tilde{s}_{2k-1},\underbrace{0,\ldots,0}_{2k-1}),\\
  \text{In } X_2:\hspace{0.3cm}  L^{-}_{X_2X_1} & =(0,-s_1,s_2,\ldots, s_{2k-1},\underbrace{0,\ldots,0}_{2k-1}), \\
  L^{+}_{X_2Y}(t) & = (t,\underbrace{0,\ldots,0}_{2k-1},\doubletilde{s}_1,\doubletilde{s}_2,\ldots,\doubletilde{s}_{2k-1}).
  \end{align*}

For $t=-\delta$, where $\delta$ is a small positive number, we get two local pictures:
 \begin{center}
\begin{tikzpicture}

\begin{knot}[
    clip width=8, consider self intersections=true,
  ignore endpoint intersections=false, flip crossing=3
    ]
 \strand [blue, very thick, only when rendering/.style={
    postaction=decorate,
  },
  decoration={
    markings,
    mark=at position 0.9 with {\arrow{To}}
  }] (1.4, 1)--(-1.4, -0.2);

         \strand [red, very thick, only when rendering/.style={
    postaction=decorate,
  },
  decoration={
    markings,
    mark=at position 0.9 with {\arrow{To}}
  }] (0, 1.4)--(0, -0.7);
        \node[left] at (0,1.7) {$L_{X_2Y}^{+}$};
        \node[right] at (1.5,1.1) {$L_{X_2X_1 }^{-}$};
         \node at (0,-1.2) {$\mathbf{X_2}$};
        
\end{knot}    

\begin{scope}[shift={(-4.8,0)}];
\begin{knot}[
    clip width=8, consider self intersections=true,
  ignore endpoint intersections=false, flip crossing=3
    ]
 \strand [red, very thick, only when rendering/.style={
    postaction=decorate,
  },
  decoration={
    markings,
    mark=at position 0.9 with {\arrow{To}}
  }] (1.4, 1)--(-1.4, -0.2);

         \strand [red, very thick, only when rendering/.style={
    postaction=decorate,
  },
  decoration={
    markings,
    mark=at position 0.9 with {\arrow{To}}
  }] (0, 1.4)--(0, -0.7);
        \node[left] at (0,1.7) {$L_{X_1X_2}^{+}$};
        \node[right] at (1.5,1.1) {$L_{X_1Y }^{+}$};
               \node at (0,-1.2) {$\mathbf{X_1}$};
        
\end{knot}  
\end{scope}      

\end{tikzpicture}
\captionof{figure}{The inverse image of the self-intersection of $f_t$ in $X$, $t<0$}
\end{center}        
 
We calculate the change in linking numbers in $X_1$ and $X_2$ going from $f_{-\delta}$ and $f_{\delta}$. 

Assume $L_{X_1Y}^{+}$ bounds a singular $2k$-chain 
\sloppy $D(L_{X_1Y}^{+})=\{(r,\tilde{s}_1,\ldots,\tilde{s}_{2k-1},0,\ldots,0)|r\leq 0\}$. At $t=-\delta$, this chain intersects $L_{X_1X_2}^{+}$ exactly at a single point. The orientation at this intersection point is given by
\begin{footnotesize}$$\Big(\underbrace{-\frac{\partial}{\partial y_1},\frac{\partial}{\partial y_2},\ldots,\frac{\partial}{\partial y_{2k-1}}}_{Or(L_{X_1X_2}^+)}, \frac{\partial}{\partial x_1}, \underbrace{\frac{\partial}{\partial x^2_1},\ldots, \frac{\partial}{\partial x^2_{2k-1}}}_{Or(L_{X_1Y}^+)}\Big),$$\end{footnotesize}which contributes $-1$ to $lk(L_{X_1X_2}^{+},L_{X_1Y}^{+})$. As this intersection vanishes for $t=\delta$, the change is $\Delta lk(L_{X_1X_2}^{+},L_{X_1Y}^{+})=0-(-1)=+1$.

Similarly, suppose $L_{X_2Y}^{+}$ bounds a singular $2k$-chain $D(L_{X_2Y}^{+})=\{(r+t,0,\ldots,0,\doubletilde{s}_1,\doubletilde{s}_2,\ldots,\doubletilde{s}_{2k-1})|r\geq 0\}$. At $t=-\delta$, this intersects $L_{X_2X_1}^{-}$ at one point with orientation
\begin{footnotesize}
$$\Big(Or(L_{X_2X_1}^{-}),- \frac{\partial}{\partial x_2}, Or(L_{X_2Y}^{+})\Big)=\Big(-\frac{\partial}{\partial y_1},\frac{\partial}{\partial y_2},\ldots,\frac{\partial}{\partial y_{2k-1}},-\frac{\partial}{\partial x_2}, \frac{\partial}{\partial x^1_1},\ldots,\frac{\partial}{\partial x^1_{2k-1}} \Big),$$
\end{footnotesize}contributing $-1$ to $lk(L_{X_2X_1}^{-},L_{X_2Y}^{+})$. This intersection also disappears at $t=\delta$, so $\Delta lk(L_{X_2X_1}^{-},L_{X_2Y}^{+})=+1$.

Hence, the net change in $\mathcal{W}_1(F)$ while passing the triple point is:
\begin{align*}
\Delta\mathcal{W}_1(F) &= \mathcal{W}_1(F_{\delta})-\mathcal{W}_1(F_{-\delta})\\
&=\frac{1}{2}\Big(\Delta lk(L_{X_1X_2}^{+},L_{X_1Y}^{+}) -\Delta lk(L_{X_2X_1}^{-},L_{X_2Y}^{+})\Big)\\
&=\frac{1}{2}(1-1)\\
&=0
 \end{align*}

Now suppose $X_1>Y>X_2$, the immersions $f_t|_{X_1},f_t|_{X_2},f_t|_{Y}$ and all intersections (with orientations) are the same as before. However, instead of $L_{X_2Y}^{+}$ and $L_{YX_2}^{-}$, we now have $L_{X_2Y}^{-}$ and $L_{YX_2}^{+}$ in $X_2$ and $Y$, respectively:
\begin{align*}
 L^{-}_{X_2Y}(t) & = (t,\underbrace{0,\ldots,0}_{2k-1},\doubletilde{s}_1,\doubletilde{s}_2,\ldots,\doubletilde{s}_{2k-1}),\\
 L_{YX_2}^{+}(t) & = (t,\doubletilde{s}_1,\doubletilde{s}_2,\ldots,\doubletilde{s}_{2k-1},\underbrace{0,\ldots,0}_{2k-1}).
\end{align*}

The linking numbers that contribute to $\mathcal{W}_1$ are $lk(L_{X_1X_2}^{+},L_{X_1Y}^{+})$, $lk(L_{X_2X_1}^{-},L_{X_2Y}^{-})$, and $lk(L_{YX_2}^{+},L_{YX_1}^{-})$.
By replacing $L_{X_2Y}^{+}$ by $L_{X_2Y}^{-}$ in the previous argument, we get $\Delta lk(L_{X_1X_2}^{+},L_{X_1Y}^{+})=+1$ and $\Delta lk(L_{X_2X_1}^{-},L_{X_2Y}^{-})=+1$. The change in $lk(L_{YX_2}^{+},L_{YX_1}^{-})$ is computed by taking a $2k$-chain $D(L_{YX_1}^{-})=\{(r,0,\ldots,0,\tilde{s}_1,\ldots,\tilde{s}_{2k-1})|r\leq 0\}$ around $L_{YX_1}^{-}$. This chain intersects $L_{YX_2}^{+}$ exactly at one point when $t=-\delta$, and there is no intersection when $t=\delta$. At the intersection point, we have 
\begin{footnotesize}$$\Big(\underbrace{\frac{\partial}{\partial x^1_1},\ldots,\frac{\partial}{\partial x^1_{2k-1}}}_{Or(L_{YX_2}^+)}, \frac{\partial}{\partial y}, \underbrace{\frac{\partial}{\partial x^2_1},\ldots, \frac{\partial}{\partial x^2_{2k-1}}}_{Or(L_{YX_1}^-)}\Big)$$\end{footnotesize}which contributes $-1$ to $lk(L_{YX_2}^{+},L_{YX_1}^{-})$, and thus we get $\Delta lk(L_{YX_2}^{+},L_{YX_1}^{-})=+1$.

Hence, the overall change in $\mathcal{W}_1$ is:
\begin{align*}
\Delta\mathcal{W}_1(F) &= \mathcal{W}_1(F_{\delta})-\mathcal{W}_1(F_{-\delta})\\
&=\frac{1}{2}\Big(\Delta lk(L_{X_1X_2}^{+},L_{X_1Y}^{+}) +\Delta lk(L_{X_2X_1}^{-},L_{X_2Y}^{-})\Big)-\Delta lk(L_{YX_2}^{+},L_{YX_1}^{-}) \\
&=\frac{1}{2}(1+1)-1\\
&=0.
 \end{align*}

To determine the image of the map $\mathcal{W}:\pi_{0}^{br}Emb(\underset{2}{\sqcup}S^{4k-1},\bbR^{6k})\xrightarrow{(\mathcal{W}_1,\mathcal{W}_2)} \bbQ\oplus \bbQ$, we first recall that $\mathcal{W}(\mathcal{B}_1)=(-2,0)$ and $\mathcal{W}(\mathcal{B}_2)=(0,-2)$ (see Example~\ref{generator}).
Now consider the tubing map $$T:\pi_{0}^{br}Emb(\underset{2}{\sqcup}S^{4k-1},\bbR^{6k})\xrightarrow{(T_1,T_2)} \bbZ\oplus \bbZ,$$ where $T_1$ (resp. $T_2$) maps a pair $X\sqcup Y$ to $X\#Y$ (resp. $X\#(-Y)$) in $\pi_{0}Emb(S^{4k-1},\bbR^{6k})=\bbZ$.  Both maps send $\mathcal{B}_1$ and $\mathcal{B}_2$, up to sign, to the generator Haefliger trefoil knot. In particular, we have $$T(\mathcal{B}_1)=(1,-1) \text{ and }T(\mathcal{B}_2)=(1,1),$$ due to the symmetry of $\pi_{0}^{br}Emb(\underset{3}{\sqcup}S^{4k-1},\bbR^{6k})$, see Remark~\ref{gpaction}, which gives $$X\sqcup Y\sqcup Z=X\sqcup (-Y)\sqcup (-Z)=-(X\sqcup Y\sqcup (-Z)).$$  These maps fit in a commutative diagram, which in terms of the rational generators is:

\begin{center}
\begin{tikzcd}[column sep=small, row sep=small] 
    & \mathcal{B}_1,\mathcal{B}_2  \arrow{ddl}[swap]{T} \arrow{ddr}{\mathcal{W}} &\\
    & \circlearrowleft &\\[-2ex]
    (1,-1),(1,1) \arrow{rr}[swap]{A} & &  (-2,0),(0,-2) \\
\end{tikzcd}
\end{center}

Here, the linear map $A$ is given by left multiplication by the matrix \begin{align*}\begin{bmatrix} -1 & 1\\ -1 & -1 \end{bmatrix}. \end{align*} 
The elements $(1,-1)$ and $(1,1)$ generate the sublattice $L=\{(a,b)\in \bbZ^2| a+b\in 2\bbZ\}$, which has index $2$ in $\bbZ^2$. Thus, the image of $T$ is $\bbZ^2$ or $L$, depending on whether the element $(1,0)$ lies in the image. This can be seen using Example~\ref{brac2}, which provides a map $G_*:\pi_{4k-1}S^{2k} \rightarrow \pi_{0}^{br}Emb(\underset{2}{\sqcup}S^{4k-1},\bbR^{6k})$ sending the (rational) generator $[id,id]$ to $-\mathcal{B}_1-\mathcal{B}_2$, up to torsion. Moreover, for $k=1,2,4$, the $\bbZ$ summand of the group $\pi_{4k-1}S^{2k}$ is generated by $[Hopf]=\frac{1}{2}[id,id]$, which maps to $\frac{1}{2}(-\mathcal{B}_1-\mathcal{B}_2)$ under $G_*$. Theorem~\ref{retract} ensures the existence of a retraction $r:\pi_{0}^{br}Emb(\underset{2}{\sqcup}S^{4k-1},\bbR^{6k})\rightarrow \pi_{4k-1}S^{2k}$ to the inclusion~$G_*$.
 Thus, we have the following diagram:
\[
\begin{tikzcd}
\pi_{4k-1}S^{2k} \arrow[r,"G_*"] 
    &\pi_{0}^{br}Emb(\underset{2}{\sqcup}S^{4k-1},\bbR^{6k}) \arrow[r,"\mathcal{W}"] \arrow[d,"T"'] & \bbQ^2\\
   & \bbZ^2 \arrow[ur,"A"'] 
\end{tikzcd}
\]
Note that $\mathcal{W}(-\mathcal{B}_1-\mathcal{B}_2)=(2,2)$ and $T(-\mathcal{B}_1-\mathcal{B}_2)=(-2,0)$. Now suppose $k\neq 1,2,4$, and assume that $(-1,0)$ is in the image of $T$. Then there exists $J\in\pi_{0}^{br}Emb(\underset{2}{\sqcup}S^{4k-1},\bbR^{6k})$ such that $T(J)=(-1,0)$. Since $T$ is a homomorphism, we must have $J=\frac{1}{2}(-\mathcal{B}_1-\mathcal{B}_2)$, which maps under $r$ to $\frac{1}{2}[id, id]$. In case $k\neq 1,2,4$, this contradicts to the fact that $[id, id]$ is a generator of the $\bbZ$ summand of $\pi_{4k-1}S^{2k}$. Hence, for $k\neq 1,2,4$, we conclude that $T(\pi_{0}^{br}Emb(\underset{2}{\sqcup}S^{4k-1},\bbR^{6k}))=L$, which implies $\mathcal{W}(\pi_{0}^{br}Emb(\underset{2}{\sqcup}S^{4k-1},\bbR^{6k}))=A(L)=2\bbZ\times2\bbZ$. For $k=1,2,4$, it follows that $T(\pi_{0}^{br}Emb(\underset{2}{\sqcup}S^{4k-1},\bbR^{6k}))=\bbZ^2$, and thus $\mathcal{W}(\pi_{0}^{br}Emb(\underset{2}{\sqcup}S^{4k-1},\bbR^{6k}))=A(\bbZ^2)=L$.

\end{proof}

In the special case where there are no Whitney umbrella points,
the formulas \eqref{eq:W1} and~\eqref{eq:W2} for the invariant $\mathcal W$ can be made to look similar to that of $\mathcal V$ in \eqref{invxyz}, see Corollary~\ref{inv2_imm}.

\begin{lemma}\thlabel{other}
   For a $2$-component link $X\sqcup Y\hookrightarrow \bbR^{6k}$, let the projection $f$ of $X$ (respectively, $f$ of $Y$) to $\bbR^{6k-1}$ be an immersion. Then, the following equality holds$$E_X(f)\coloneq \frac 12 lk_X(L_{XX},L_{XY})=lk_{\bbR^{6k-1}}(X\cap X, Y)$$ 
    $$\Big(\text{respectively, }E_Y(f)\coloneq \frac 12 lk_Y(L_{YY},L_{YX})=lk_{\bbR^{6k-1}}(Y\cap Y, X)\Big),$$
    
    where the sets of double points $L_{ij}$ for $i,j=X,Y$ are oriented using the Ekholm orientation $L_{ij}=L_{ij}^+\sqcup L_{ij}^-$, see Subsection~\ref{orient}.\footnote{To compare, note that in Theorem~\ref{inv2} we use the adjusted Ekholm orientation 
    $L_{ij}=L_{ij}^+\sqcup -L_{ij}^-$.}
\end{lemma}
\begin{proof}

We prove the equality for $E_X(f)$; the argument for $E_Y(f)$ is identical by swapping $X$ and $Y$. Choose a singular $4k$-chain $D_Y$ with boundary $\partial D_Y=Y$. Then by definition
$$lk_{\bbR^{6k-1}}(X\cap X, Y)=(X\cap X)\cdot D_Y.$$
Set $L_{XY}=f^{-1}(X\cap Y)$ and $D_{XY}\coloneq f^{-1}(X\cap D_Y)$. Using transversality and a careful computation of orientations, we have $\partial D_{XY}=L_{XY}$, so that
$$lk_X(L_{XX},L_{XY})=L_{XX}\cdot D_{XY}.$$
Since $f$ restricted to $L_{XX}$ is a $2:1$ covering of $X\cap X$, as sets 
$$\abs{L_{XX}\cdot D_{XY}}=2\abs{(X\cap X)\cdot D_Y},$$
up to a sign. Now consider $X\cap X\subset \bbR^{6k-1}$ as the intersection of two sheets $X_+$ and $X_-$, oriented using the transverse intersection rule. By Ekholm convention, the preimages $L_{XX}^+\subset X_+$ and $L_{XX}^-\subset X_-$ inherit the orientation of $X\cap X$. For $L_{XX}\cdot D_{XY}$, the sign of each intersection point is determined by comparing $(Or(L_{XX}), Or(D_{XY}))$ with the orientation of the sheet $X_{\pm}$; equivalently, it depends on whether $Or(D_{XY})$ agrees with the complementary orientation of $X_{\pm}$. For $(X\cap X)\cdot D_Y$, note that the $4k$-dimensional complement $D_Y$ of $X\cap X$
 in $\bbR^{6k-1}$ naturally splits into two $2k$-dimensional pieces, one in each sheet. Hence, the sign of the intersection depends on whether $Or(D_Y)$ agrees with the complementary orientations of both $X_+$ and $X_-$ together. Therefore, the contributions to the sign from $L_{XX}\cdot D_{XY}$ and $(X\cap X)\cdot D_Y$ agree, which proves the claim.
 
\end{proof}

The following is an immediate corollary of Theorem~\ref{inv2} and Lemma~\ref{other}. Note that the formulas look similar to the $3$-component case in Theorem~\ref{inv1} and also to the knot case in Theorem~\ref{knotinv}: their first summands $E_X(f)$ and $E_Y(f)$ depend only on the
projection $f$, while the rest is a combination of linking numbers between the over- and under-crossings.
\begin{cor}\thlabel{inv2_imm}
    If the projection $f$ of $X$ (respectively, $Y$) to $\bbR^{6k-1}$ is an immersion, then  $$ \mathcal{W}_1(F)=E_X(f)-lk(L^+_{XX},L^-_{XY})-lk(L^-_{XX},L^+_{XY})-lk(L^+_{YX},L^-_{YX})$$
   $$ \Big(\text{respectively, } \mathcal{W}_2(F)=E_Y(f)-lk(L^+_{YY},L^-_{YX})-lk(L^-_{YY},L^+_{YX})-lk(L^+_{XY},L^-_{XY})\Big).  $$
    
 \end{cor}

\end{subsection}

\end{section}

\begin{section}{Formula for the Haefliger invariant}
\label{conjecture}
We now present a combinatorial formula for the Haefliger invariant of embeddings $S^{4k-1}\hookrightarrow \bbR^{6k}$ whose projections to $\bbR^{6k-1}$ are generic immersions. Recall that the isotopy classes of smooth embeddings of $S^{4k-1}$ into $\bbR^{6k}$ are in one-to-one correspondence with the integers via the Haefliger invariant (see \cite{HAE2} and \cite[\S 5.16 and Corollary 8.14]{HAE}): $$\mathcal{H}:\pi_0Emb(S^{4k-1},\bbR^{6k})\xrightarrow{\approx} \bbZ.$$
Moreover, it is known that a nonzero multiple (two, when $k=1$) of the Haefliger trefoil knot, the generator of $\pi_0Emb(S^{4k-1},\bbR^{6k})$, can be isotoped to an embedding whose projection to $\bbR^{6k-1}$ is still an embedding (but non-trivial  as immersion in $\bbR^{6k-1}$), see \cite[Theorem~5.17 and Corollary~6.7]{HAE}. 
Immersions $S^{4k-1}\looparrowright\bbR^{6k-1}$ have been extensively studied in \cite{EK,EK2,EK1}. 

\begin{subsection}{Invariants of immersions $S^{4k-1}\looparrowright \bbR^{6k-1}$}

According to Smale \cite{SM}, there is a bijection between the set (group) of regular homotopy classes of immersions $S^{4k-1}\looparrowright \bbR^{6k-1}$ and the elements of the group $\pi_{4k-1}V_{6k-1,4k-1}$, where $V_{6k-1,4k-1}$ denotes the Stiefel manifold of $(4k-1)$-frames in $(6k-1)$-space. Given an immersion $f:S^{4k-1}\looparrowright \bbR^{6k-1}$, let $\Omega(f)\in \pi_{4k-1}V_{6k-1,4k-1}$ denote its Smale invariant. In these dimensions, it is easier to understand $\Omega(f)$ rationally. Consider the fiber sequence $S^{2k}\rightarrow V_{6k-1,4k-1}\rightarrow V_{6k-1,4k-2}$, which induces a rational isomorphism between $\pi_{4k-1}S^{2k}$ and $\pi_{4k-1}V_{6k-1,4k-1}$. Thus, any immersion $f:S^{4k-1}\rightarrow \bbR^{6k-1}$ can be related to the rational generator $[id, id]$ of $\pi_{4k-1}S^{2k}$, that is, $\Omega(f)=\Omega_{S^{2k}}(f)\cdot [id,id]\in \pi_{4k-1}V_{6k-1,4k-1}\otimes \bbQ$ for some $\Omega_{S^{2k}}(f)\in\bbQ$.
 
In order to understand the relation between the 
Smale invariant and the set of double points
of genereric immersions $S^{4k-1}\looparrowright \bbR^{6k-1}$, Ekholm~\cite{EK,EK2} introduced a \textit{linking invariant}, that we denote by $E$ here, see also~\cite{EK1}. 
Let $f:S^{4k-1}\looparrowright\bbR^{6k-1}$ be a generic immersion with the $(2k-1)$-dimensional self-intersection $f(L)\subset \bbR^{6k-1}$, where $L\subset S^{4k-1}$ is the preimage of the double points. The map $f|_{L}:L\rightarrow f(L)$ is a $2$-fold covering. Since the codimension of $f$ is even, we fix an orientation on $f(L)$ (and thus on $L$) to be the Ekholm orientation, see Subsection~\ref{orient}. Let $N(S^{4k-1})$ denote the $2k$-dimensional normal bundle of $f(S^{4k-1})$ in $\bbR^{6k-1}$, and let $N(S^{4k-1})|_{L}$ denote its restriction to $L$. Since the inclusion $L\hookrightarrow S^{4k-1}$ is null-homotopic, the normal bundle $N(S^{4k-1})|_{L}$ must be trivial: $$N(S^{4k-1})|_{L}\cong L\times \bbR^{2k},$$ and this trivialization is uniquely determined up to homotopy (since $\pi_iS^{4k-1}=0$ for $i\leq 2k$). Note that $N(S^{4k-1})|_{L}$ is naturally identified with the $2k$-dimensional normal bundle $N(L)$ over $L$ in~$S^{4k-1}$ pulled back along the involution. Let $(v_1,\ldots,v_{2k})$ be the framing of this homotopically unique trivialization of~$N(L)$. For each path component $L_i\subset L$, the set of non-vanishing sections over $N({L_i})$ is in one-to-one correspondence with the homotopy classes of maps from $L_i$ to $S^{2k-1}$, which are classified by their degree:$$\pi_0\Gamma(N(L_i)\backslash L_i)=[L_i^{2k-1},S^{2k-1}]=\bbZ.$$ Hence, we may choose a non-zero normal vector field $v$ along $L\subset S^{4k-1}$ such that $[\widetilde{L}]=0\in H_1(S^{4k-1}-L)$, where $\widetilde{L}$ denote a copy of $L$ shifted slightly along $v$. In particular, for each component $L_i$ we get\begin{equation}\label{framing}
    lk(L_i,\widetilde{L})=lk(L_i,\widetilde{L}_i)+\underset{j\neq i}{\sum} lk(L_i,L_j)=0.
\end{equation}
Since $lk(\widetilde{L_i},L_j)=lk(L_i,\widetilde{L_j})$, the above equation can be rewritten as 
  \begin{equation}\label{ekfr1}
     lk(\widetilde{L_i},L)=lk(\widetilde{L_i},L_i)+\underset{j\neq i}{\sum} lk(L_i,L_j)=0.
 \end{equation}
Define a vector field $w$ along $f(L)$ such that at the double point $p=f(p_1)=f(p_2)$, we have $w(p)\coloneq df(v(p_1))+df(v(p_2))$. Let $\widetilde{f(L)}\subset \bbR^{6k-1}$ be the result of pushing $f(L)$ slightly along~$w$. Then $\widetilde{f(L)}\cap f(S^{4k-1})=\emptyset$, and we define $$E(f):= lk(\widetilde{f(L)}, f(S^{4k-1})),$$
where the linking number is computed in $\bbR^{6k-1}$.   
\end{subsection} 

\begin{subsection}{The formula}

For computational purposes, it is easier to work with  long immersions and embeddings, so we state our formula in that case. However, recall that $\pi_0Emb(S^{4k-1},\bbR^{6k})=\pi_0Emb_{\partial}(\bbR^{4k-1},\bbR^{6k})$ and $\pi_0Imm(S^{4k-1},\bbR^{6k-1})=\pi_0Imm_{\partial}(\bbR^{4k-1},\bbR^{6k-1})$. Moreover, formula~\eqref{conh1} also holds for spherical knots. One shall only be careful with the definition of $\Omega_{S^{2k}}(f)$. 
    \begin{thm}\thlabel{knotinv}
    Consider a long knot $F:\bbR^{4k-1}\hookrightarrow \bbR^{6k}$ such that its projection $f=p\circ F:\bbR^{4k-1}\looparrowright \bbR^{6k-1}$ is a generic immersion, with $L=L^+\sqcup L^-\subset \bbR^{4k-1}$ as the set of double points oriented using the Ekholm orientation. Then, the formula  \begin{equation}\label{conh1}
    \mathcal{H}(F):=\frac{1}{2}lk(L^+,L^-)+\frac{1}{6}E(f)+\frac{1}{3}\Omega_{S^{2k}}(f),\end{equation}
 produces the Haefliger isomorphism $\mathcal{H}:~\pi_0Emb_{\partial}(\bbR^{4k-1},\bbR^{6k})\xrightarrow{\approx} \bbZ$.
\end{thm}

\begin{example}\thlabel{haefeg}
    We claim that $\mathcal{H}(\mathcal{T})=1$, where $\mathcal{T}$ denotes the Haefliger trefoil knot, which is constructed by tubing the three components of the Borromean link $\mathcal{B}=X\sqcup Y\sqcup Z\subset \bbR^{6k}$, as shown in Figure~\ref{fig:HT}. 
    \begin{figure}[h]
\centering
 \includegraphics[width=0.45\textwidth]{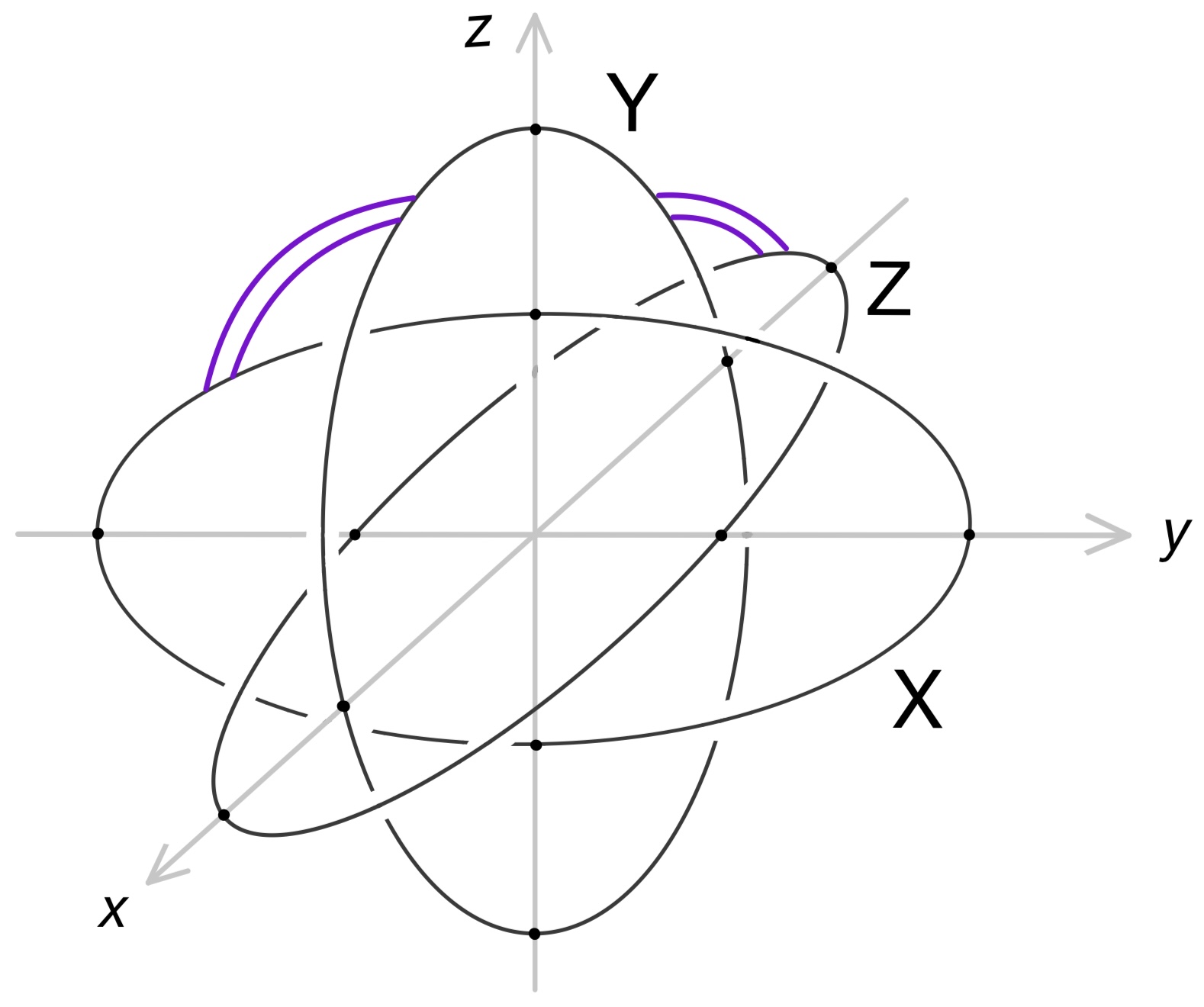}
 \caption{Haefliger trefoil knot $\mathcal{T}$}
 \label{fig:HT}
 \end{figure}
    
We use the description of $\mathcal{B}$ as given in Example~\ref{borromean} (now $\ell=2k$). The Haefliger trefoil is defined as the connected sum $\mathcal{T}\coloneq X\#Y\#Z\#F_0$, where $F_0:\bbR^{4k-1}\subset \bbR^{6k}$ is the standard inclusion. Let $f\coloneq p\circ \mathcal{T}$ be the generic immersion to $\bbR^{6k-1}$, where $p$ is projection along the vector $(1,\dots,1)\in\bbR^{6k}$.
Choosing the tubes in general position, the self-intersection of $f$ agrees with that of $p\circ \mathcal{B}$, consisting of six disjoint $(2k-1)$-spheres in~$\bbR^{6k-1}$. The double point set $L=L^+\sqcup L^-\subset \bbR^{4k-1}$ of $f$ is also the same: six disjoint Hopf links, each with linking number $-1$ (see Figure~\ref{fig:6hopf}). Since none of these Hopf links involve the linking between an overcrossing and undercrossing components, we have $lk(L^+,L^-)=0$. Moreover, $\Omega(f)=0$ because $f$ is trivial as an immersion, as it can be 
regularly homotoped so that $X,Y,Z$
become  standard copies of $S^{4k-1}\subset \bbR^{4k}\subset\bbR^{6k-1}$. Thus, the only non-trivial contribution comes from the term $E(f)$. 

We compute $E(f)$ using the three pairwise intersections $f(X)\cap f(Y), f(X)\cap f(Z),$ and $f(Y)\cap f(Z)$. Consider $f(X)\cap f(Y)=A_1\sqcup A_2$, with preimages $L_{XY}=L_{XY}^+\sqcup L_{XY}^-$ and $L_{YX}=L_{YX}^+\sqcup L_{YX}^-$ such that $A_1=f(L_
{XY}^+)=f(L_{YX}^-)$ and $A_2=f(L_{XY}^-)=f(L_{YX}^+)$. 

 For each $A_i$, let $\widetilde{A_i}$ be the shift of $A_i$ along the sum of the shifts corresponding to its preimages, see the definition of $E$. Moreover, let $(\widetilde{A_{i}})_X$ (respectively, $(\widetilde{A_{i}})_Y$) be the shift of $A_i$ only along the framing inside $X$ (respectively, $Y$). It is easy to check that $$lk_{X}(\widetilde{L_{XY}^+},L_{XY})=lk_{\bbR^{6k-1}}((\widetilde{A_1})_X,f(Y)),$$ using the general position argument and comparing the orientations on both sides. 
 Similarly, $$lk_{Y}(\widetilde{L_{YX}^-},L_{YX})=lk_{\bbR^{6k-1}}((\widetilde{A_1})_Y,f(X)),$$ and $$lk_{X}(L_{XY}^+,L_{XZ})=lk_{\bbR^{6k-1}}(A_1,f(Z))=lk_{Y}(L_{YX}^-,L_{YZ}).$$ 


Then, it follows that
\begin{align*}
lk_{\bbR^{6k-1}}(\widetilde{A_1},f(\bbR^{4k-1}))  =& \quad lk((\widetilde{A_1})_X,f(Y))+lk((\widetilde{A_1})_Y,f(X))+lk(\widetilde{A_1},f(Z))\\
=& \quad lk_{X}(\widetilde{L_{XY}^+},L_{XY})+ lk_{Y}(\widetilde{L_{YX}^-},L_{YX})+lk_X(L_{XY}^+,L_{XZ})\\
& \Big(\text{ note that the last term can also be written as  } lk_{Y}(L_{YX}^-,L_{YZ})\Big)\\
=& \quad \underbrace{lk_{X}(\widetilde{L_{XY}^+},L_{XY})+lk_X(L_{XY}^+,L_{XZ})}_{=0 \text{ (using }\eqref{ekfr1}) }+lk_{Y}(\widetilde{L_{YX}^-},L_{YX})\\
=& \quad -lk_Y(L_{YX}^-,L_{YZ}) \qquad \Big(\text{  or  } -lk_{X}(L_{XY}^+,L_{XZ}) \Big)
\end{align*}
(the last equality follows by using \eqref{ekfr1} one more time). 

A parallel computation yields $$lk_{\bbR^{6k-1}}(\widetilde{A_2},f(\bbR^{4k-1}))= -lk_Y(L_{YX}^+,L_{YZ}) \quad \Big(\text{  or  } -lk_{X}(L_{XY}^-,L_{XZ}) \Big).$$

Thus, \begin{align*}
    lk_{\bbR^{6k-1}}(\widetilde{f(X)\cap f(Y)}, f(\bbR^{4k-1}))&= lk(\widetilde{A_1},f(\bbR^{4k-1}))+lk(\widetilde{A_2},f(\bbR^{4k-1}))\\
    &=-lk_Y(L_{YX}^-,L_{YZ})-lk_Y(L_{YX}^+,L_{YZ})\\
    &=-lk_Y(L_{YX},L_{YZ})\qquad \Big(\text{  or  } -lk_{X}(L_{XY},L_{XZ}) \Big)
\end{align*}

 We get similar results by repeating the argument for the other two intersections $f(Y)\cap f(Z)$ and $f(Z)\cap f(X)$. Thus, $E(f)=-3lk_X(L_{XY},L_{XZ})=-3lk_Y(L_{YX},L_{YZ})=-3lk_Z(L_{ZX},L_{ZY}).$
In any component $X,Y,Z$, there are two Hopf links, each with linking number $-1$. Hence, \begin{align*}
    \mathcal{H}(\mathcal{T})&=\frac{1}{6}\cdot \Big(-3\cdot lk_X(L_{XY}^+,L_{XZ}^+)-3\cdot lk_X(L_{XY}^-,L_{XZ}^-)\Big)\\
    &=\frac{1}{6}\cdot 6=1
\end{align*}

\end{example}

\begin{lemma}\thlabel{reghom}
    As defined in \eqref{conh1}, $\mathcal{H}$ remains invariant under isotopy whose projection is a regular homotopy, that is, under the third Reidemeister move and cobordism moves without Whitney umbrellas.
\end{lemma}
\begin{proof}
 Since the Smale invariant $\Omega$ is already preserved under regular homotopy, it suffices to analyze the first two terms. As shown in \cite{EK,EK2,EK1}, $E$ is unchanged under cobordism moves that avoid Whitney umbrellas, but it jumps by $\pm 3$ under the triple point move; so with the factor $\frac{1}{6}$ this gives $\pm\frac{1}{2}$. The linking number $lk(L^+,L^-)$ also remains constant under such cobordism moves, since no new over/undercrossing linkings are introduced. However, under the third Reidemeister move (see Figure~\ref{fig:3rd}), it changes by $\pm 1$,  thus also contributing $\pm\frac{1}{2}$. We now show that these two contributions cancel, ensuring invariance of the total expression.

 Let $\{F_t\}$ be a $1$-parameter family of long embeddings $\bbR^{4k-1}\hookrightarrow\bbR^{6k}$, and let $\{f_t\}$ denote their projections to $\bbR^{6k-1}$, forming a generic $1$-parameter family of immersions.
Suppose $f_0$ has a triple point $0\in\bbR^{6k-1}$, with three sheets $X$, $Y$ and $Z$ ordered so that $X$ lies over $Y$, and $Y$ over $Z$. With the same local coordinates around the triple point as chosen in the proof of Theorem~\ref{inv1} (now $\ell=2k$), only  $lk(L_{XY}^+,L_{XZ}^+),lk(L_{YZ}^+,L_{YX}^-),$ and $lk(L_{ZX}^-,L_{ZY}^-)$ are affected. It is straightforward to verify that each of these linking numbers increases by one when moving through the triple point. Furthermore, only $lk(L_{YZ}^+,L_{YX}^-)$ contributes to the change in $lk(L^+,L^-)$, and therefore $\Delta lk(L^+,L^-)=1$. 

To see the change in $E$, let us choose one (out of three) of the self-intersections, say $$A_{XZ}=f_t(X)\cap f_t(Z)=(0,0,\underbrace{0,\ldots,0}_{2k-1},\tilde{s}_1,\ldots,\tilde{s}_{2k-1},\underbrace{0,\ldots,0}_{2k-1}),$$ with preimages $L_{XZ}^+$ and $L_{ZX}^-$. By the definition of $E$, the shifted copies of these preimages must satisfy \eqref{framing}, that is,
$$\Delta lk(L_{XZ}^+,\widetilde{L_{XZ}^+})=-\Delta lk(L_{XZ}^+,L_{XY}^+)=-1,$$ $$\Delta lk(L_{ZX}^-,\widetilde{L_{ZX}^-})=-\Delta lk(L_{ZX}^-,L_{ZY}^-)=-1,$$ (we assume that the shifting distance is very small compared to $\delta$).

Let $\widetilde{(A_{XZ})}_X$ (respectively, $\widetilde{(A_{XZ})}_Z$) be the shift of $A_{XZ}$ along the framing in $X$ (respectively, $Z$) corresponding to $\widetilde{L_{XZ}^+}$ (respectively, $\widetilde{L_{ZX}^-}$). As we move through the triple point, $\Delta lk\big(\widetilde{A_{XZ}}, f_t(\bbR^{4k-1})\big)$ will have three contributions: $$\Delta lk\big(\widetilde{(A_{XZ})}_X, f_t(Z)\big),\quad \Delta lk\big(\widetilde{(A_{XZ})}_Z, f_t(X)\big), \text{  and  }\quad\Delta lk\big(A_{XZ}, f_t(Y)\big),$$ where the last one does not depend on the shift. Since $\Delta lk(L_{XZ}^+,\widetilde{L_{XZ}^+})=-1=\Delta lk(L_{ZX}^-,\widetilde{L_{ZX}^-})$, thus $\Delta lk\big(\widetilde{(A_{XZ})}_X, f_t(Z)\big)=-1=\Delta lk\big(\widetilde{(A_{XZ})}_Z, f_t(X)\big)$.

Let $A_{XZ}$ be the boundary of a singular $2k$-chain $$D(A_{XZ})=\{(0,0,\underbrace{0,\ldots,0}_{2k-1},\tilde{s}_1,\ldots,\tilde{s}_{2k-1},\underbrace{0,\ldots,0}_{2k-1})|r\leq 0\}.$$ For $t=-\delta$, $D(A_{XZ})$ intersects $f_t(Y)$ exactly at one point at which the orientation is given by $$(\underbrace{e_1,e_{2k+2},\ldots,e_{4k}}_{Or(D(A_{XZ}))},\underbrace{e_2,e_{4k+1},\ldots,e_{6k-1}, e_3,\ldots, e_{2k+1}}_{Or(f_t(Y))}),$$ which is negative of the standard orientation on $\bbR^{6k-1}$. Thus, the algebraic number of intersection $D(A_{XZ})\cdot f_{-\delta}(Y)$ differs from $D(A_{XZ})\cdot f_{\delta}(Y)$ by $-1$ and we get $\Delta lk\big(A_{XZ}, f_t(Y)\big)=1$. The net change is then: $$\Delta lk\big(\widetilde{A_{XZ}}, f_t(\bbR^{4k-1})\big)=-1-1+1=-1.$$ Similarly, one can check that $\Delta lk\big(\widetilde{A_{XY}}, f_t(\bbR^{4k-1})\big)$ and $\Delta lk\big(\widetilde{A_{YZ}}, f_t(\bbR^{4k-1})\big)$ are $-1$ when $t$ goes from $-\delta$ to $\delta$. Thus, $E(f_{\delta})-E(f_{-\delta})=-3$ and the overall change in our formula is given by
\begin{align*}
\Delta \mathcal{H}(F) &= \mathcal{H}(F_{\delta})- \mathcal{H}(F_{-\delta})\\
    &= \frac{1}{2}\Delta lk(L^+,L^-)+\frac{1}{6}\Delta E(f)-\frac{1}{3}\Delta\Omega_{S^{2k}}(f)\\
    &= \frac{1}{2}\cdot 1+\frac{1}{6}\cdot (-3)-\frac{1}{3}\cdot 0\\
    &=0.
\end{align*}
\end{proof}

For the proof of Theorem~\ref{knotinv}, we require an additional space $Emb_{\partial}^+(\bbR^{4k-1},\bbR^{6k})$ and a cabling map on it. Let $Emb_{\partial}^+(\bbR^{4k-1},\bbR^{6k})$ denote the space of long embeddings $F:\bbR^{4k-1}\hookrightarrow\bbR^{6k}$ equipped with a nowhere zero normal vector field $n$, such that both $F$ and $n$ agree with the standard behavior outside a compact set. More precisely, outside a compact set, the normal field is chosen to be $\mathbf{0}:=e_{6k-1}+e_{6k}$ (chosen to avoid tangency with any component of the Borromean link). Each element of this space is given by a pair $(F,n)$. We define the cabling map $$C:Emb_{\partial}^+(\bbR^{4k-1},\bbR^{6k})\rightarrow Emb_{\partial}(\underset{2}{\sqcup}\bbR^{4k-1},\bbR^{6k})$$ by
$$C(F,n)= F\sqcup \widetilde{F}_n,$$
where $\widetilde{F}_n$ denotes a shifted copy of $F$ along the normal field~$n$.

\begin{lemma}\thlabel{embplus}
    \begin{align*}
        \pi_0Emb_{\partial}^+(\bbR^{4k-1},\bbR^{6k})&=\pi_0Emb_{\partial}(\bbR^{4k-1},\bbR^{6k})\oplus \pi_{4k-1}S^{2k}\\
        &= \bbZ\oplus \pi_{4k-1}S^{2k}.
    \end{align*}
\end{lemma}
\begin{proof}
Consider the fiber sequence $$\Omega^{4k-1}S^{2k}\rightarrow Emb_{\partial}^+(\bbR^{4k-1},\bbR^{6k})\rightarrow Emb_{\partial}(\bbR^{4k-1},\bbR^{6k}),$$ that produces a long exact sequences of homotopy groups:
\begin{equation}\label{seq}
    \ldots\rightarrow \pi_{4k-1}S^{2k}\xrightarrow{i_*} \pi_0Emb_{\partial}^+(\bbR^{4k-1},\bbR^{6k})\xrightarrow{j_*} \pi_0Emb_{\partial}(\bbR^{4k-1},\bbR^{6k}).
\end{equation}
Any long embedding $\bbR^{4k-1}\hookrightarrow \bbR^{6k}$ has a trivial normal bundle (the Haefliger trefoil is regular homotopic to the trivial immersion). This implies that there exists a section $s:\pi_0Emb_{\partial}(\bbR^{4k-1},\bbR^{6k})\rightarrow \pi_0Emb_{\partial}^+(\bbR^{4k-1},\bbR^{6k})$ that assigns to any embedding a standard choice of normal vector field. 

Furthermore, consider the following commutative diagram:
\[
\begin{tikzcd}
\pi_{4k-1}S^{2k} \arrow[r,"i_*"] \arrow[dr, "G_*"]
& \pi_0Emb_{\partial}^+(\bbR^{4k-1},\bbR^{6k}) \arrow[d,"C"]\\
    &\pi_{0}Emb_{\partial}(\underset{2}{\sqcup}\bbR^{4k-1},\bbR^{6k})  
\end{tikzcd}
\]
The map $G_*$ is split-injective (by Theorem~\ref{retract}), which implies that $i_*$ is injective. Since the map $j_*$ in~\eqref{seq} has a section, the result follows.

\end{proof}

\begin{proof}[Proof of Theorem~\ref{knotinv}]
We define two group homomorphisms $$\theta_1, \theta_2:\pi_0Emb_{\partial}^+(\bbR^{4k-1},\bbR^{6k})\rightarrow\bbQ$$ as follows. 
By the compression theorem (see \cite{COM}), any element  $F\in Emb_{\partial}^+(\bbR^{4k-1},\bbR^{6k})$ can be isotoped so that the normal vector field becomes vertical. After this isotopy, the vertical projection of the   embedding $\tilde F$ so obtained becomes a generic immersion of $\bbR^{4k-1}$ into $\bbR^{6k-1}$. We then define $\theta_1$ and $\theta_2$ by the formulas
$$\theta_1=\frac{1}{2}lk(L^+,L^-)+\frac{1}{6}E, \quad \theta_2=\Omega_{S^{2k}}$$
applied to $\tilde F$ and its projection $p\circ \tilde F$, respectively. 
The parametrized compression theorem (see \cite{COM}) guarantees that any path between such elements $\tilde F, \tilde G\in Emb_{\partial}^+(\bbR^{4k-1},\bbR^{6k})$ is homotopic to an isotopy 
with vertical normal vector field for every point in $\bbR^{4k-1}$ and at every moment $t\in [0,1]$ implying 
that the projection is a  regular homotopy. By Lemma~\ref{reghom}, it follows that $\theta_1$ is 
well-defined since it is invariant under isotopies whose projections are regular homotopies. 
The well-definedness of $\theta_2$ also follows directly from the parametrized compression theorem, though less surprisingly as $$\pi_0 Imm_{\partial}^+(\bbR^{4k-1},\bbR^{6k})=\pi_{4k-1}V_{6k,4k}=\pi_{4k-1}V_{6k-1,4k-1}=\pi_0Imm_{\partial}(\bbR^{4k-1},\bbR^{6k-1}).$$

Together, these invariants detect the rational isomorphism $$\pi_0Emb_{\partial}^+(\bbR^{4k-1},\bbR^{6k})\simeq_\bbQ \bbQ\oplus \bbQ$$ (see Lemma~\ref{embplus}). The only information lost is the torsion part of $\pi_{4k-1}S^{2k}$. We now claim that the combination $$\theta= \theta_1+\frac{1}{3}\theta_2:\pi_0Emb_{\partial}^+(\bbR^{4k-1},\bbR^{6k}) \rightarrow \bbQ,$$ vanishes on the $\pi_{4k-1}S^{2k}$ summand. This shows that the framing component contributes trivially to $\theta$, and therefore $\theta$ descends to a well-defined invariant on $\pi_0Emb_{\partial}(\bbR^{4k-1},\bbR^{6k})$.

   To prove the claim, it suffices to check that $\theta$ vanishes on any pair $(\mathcal{U},v)$, where $\mathcal{U}$ denotes the unknot and $v$ is a nonzero normal vector field. This can be verified using the cabling map on a specific embedding. Firstly, it is easy to see that if the normal vector field is $\mathbf{0}=e_{6k-1}+e_{6k}$, then we have $\theta(\mathcal{U},\mathbf{0})=0$ and $\theta(\mathcal{T},\mathbf{0})=1$, for the latter see Example~\ref{haefeg}. Haefliger \cite[Corollary~6.7]{HAE} has shown that a multiple $c\mathcal{T}$ of the Haefliger trefoil knot is isotopic to an embedding in $\bbR^{6k-1}$. Let $\mathcal{X}\in Emb_{\partial}(\bbR^{4k-1},\bbR^{6k-1})$ denote the isotoped version of $c\mathcal{T}$. Consider the inclusion $$j:Emb_{\partial}(\bbR^{4k-1},\bbR^{6k-1})\hookrightarrow Emb^+_{\partial}(\bbR^{4k-1},\bbR^{6k})$$ mapping an embedding $F$ to $(F,\mathbf{0})$.
Then $j(\mathcal{X})$ is isotopic to $(c\mathcal{T},v)$, for some nonzero normal vector field $v$. Let $Imm_{\partial}^+(\bbR^{4k-1},\bbR^{6k})$ be the space of immersions $\bbR^{4k-1}\looparrowright\bbR^{6k}$ equipped with a normal vector field. Because $\mathcal{T}$ is trivial as an immersion, so is $c\mathcal{T}$, and hence $(c\mathcal{T},v)$ in $Imm_{\partial}^+(\bbR^{4k-1},\bbR^{6k})$ is determined only by this nontrivial framing~$v$. Note that $(c\mathcal{T},v)$ can be chosen so that $$(c\mathcal{T},v)\simeq (c\mathcal{T},\mathbf{0})\# (\mathcal{U},v),$$ where $(\mathcal{U},v)$ is the unknot with the nonzero normal vector field $v$. Since the set of possible framings up to homotopy on $\mathcal{U}$ is measured by the group $\pi_{4k-1}S^{2k}$, which is rationally equivalent to $\bbZ$ with the generator $[id,id]$, so $v$ corresponds to some multiple of $[id,id]$.

Furthermore, we have $\mathcal{W}(C(c\mathcal{T},v))=(0,0)$, where $\mathcal{W}$ is the invariant defined for $2$-component links (see Section~\ref{2complinks}). Indeed, the cabling $C(j(\mathcal{X}),\mathbf{0})$ yields a two-component link $\mathcal{X}\sqcup\mathcal{X}'$, where the components $\mathcal{X}$ and $\mathcal{X}'$ lie in $\bbR^{6k-1}\times 0$ and $(\bbR^{6k-1}\times 0)+\epsilon\cdot\mathbf{0}$, respectively.
 Under projection to $\bbR^{6k-1}$ along the vector $(1,\ldots,1)\in \bbR^{6k}$, the projected images of $\mathcal{X}$ and $\mathcal{X}'$ may, a priori, intersect nontrivially. However, we can isotope $\mathcal{X}'$ within $(\bbR^{6k-1}\times 0)+\epsilon\cdot\mathbf{0}$ so that its support lies disjoint from that of $\mathcal{X}$, ensuring their projections have empty intersection, see Figure~\ref{fig:cablingg}. It follows that $\mathcal{W}(C(j(\mathcal{X}),\mathbf{0}))=(0,0)$, and hence $\mathcal{W}(C(c\mathcal{T},v))=(0,0)$. 

   \begin{figure}[h]
\centering
 \includegraphics[width=0.7\textwidth]{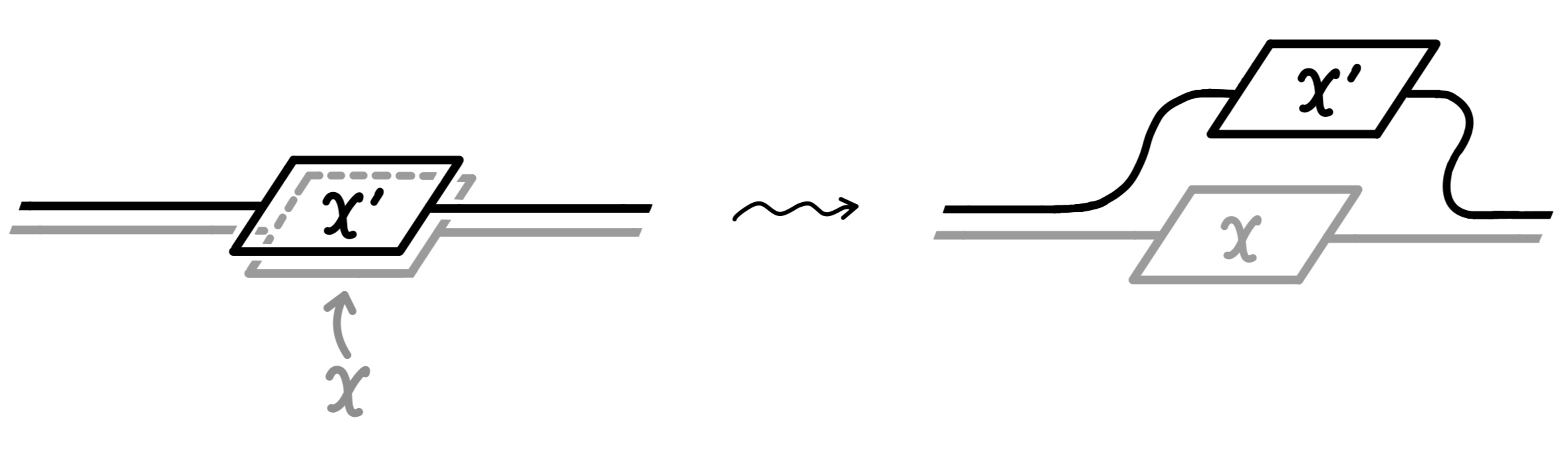}
 \caption{Isotoped $\mathcal{X}'$ in $(\mathbb{R}^{6k-1}\times 0)+\epsilon\cdot\mathbf{0}$}
 \label{fig:cablingg}
 \end{figure}

 From Proposition~\ref{brac2}, we have $\mathcal{W}(G_*([id,id]))=(2,2)$. Therefore, for $v=m\cdot [id,id]$, we get $\mathcal{W}(C(\mathcal{U},v))=(2m,2m)$. This implies that our invariant $\mathcal{W}=(\mathcal{W}_1,\mathcal{W}_2)$ when evaluated on any two-component link arising via the cabling, must satisfy $\mathcal{W}_1=\mathcal{W}_2$. Furthermore, it is easy to compute that $\mathcal{W}(C(\mathcal{T},\mathbf{0}))=\mathcal{W}(\mathcal{T}\sqcup\mathcal{T}')=(-6,-6)$, where $\mathcal{T}'$ is a slightly shifted copy of $\mathcal{T}$ along~$\mathbf{0}$. Indeed, each self-intersection of $p\circ \mathcal{T}$ in Example~\ref{haefeg} gives four self-intersections of $p\circ (\mathcal{T}\sqcup \mathcal{T}')$, see Figure~\ref{fig:cablinginter}.

\begin{center}
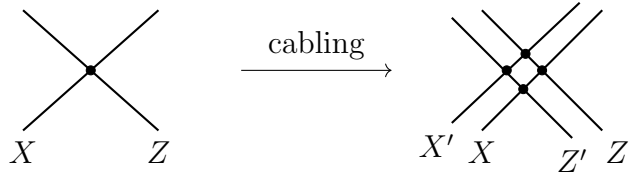

\begin{tikzpicture}
\draw[thick] (-0.8,-0.8)--(0.8,0.8);
\draw[thick] (-1.2,-0.7)--(0.5,0.9);
\draw[thick] (-0.8,0.8)--(0.8,-0.8);
\draw[thick] (-1.2,0.7)--(0.4,-0.9);
             \filldraw (0,0) circle (1.7pt);
             \filldraw (-0.22,0.22) circle (1.7pt);
             \filldraw (-0.25,-0.25) circle (1.7pt);
             \filldraw (-0.47,0) circle (1.7pt);
  \node [below] at (-0.8,-0.8) {$X$};
  \node [below] at (-1.4,-0.7) { $X'$};
\node [below] at (1,-0.8) { $Z$};
\node [below] at (0.4,-0.9) { $Z'$};
\begin{scope}[shift={(-6,0)}]
\draw[thick] (-0.9,-0.8)--(0.9,0.8);
\draw[thick] (-0.9,0.8)--(0.9,-0.8);
             \filldraw (0,0) circle (1.7pt);
  \node [below] at (-0.9,-0.8) {$X$};
    \node [below] at (0.9,-0.8) {$Z$};
\draw[->] (2,0) -- (4,0) node[above] at (3,0) {cabling};

\end{scope}
\end{tikzpicture} 
\captionof{figure}{Double intersection under cabling}
\label{fig:cablinginter}
\end{center}

Moreover, it is easy to check that the set of double points consists of two double Hopf links in each $X, Y,Z$ and $X',Y',Z'$. For example, the link inside $Z$ is shown in Figure~\ref{fig:cablingset}.

\begin{figure}[h]
\centering
 \includegraphics[width=0.6\textwidth]{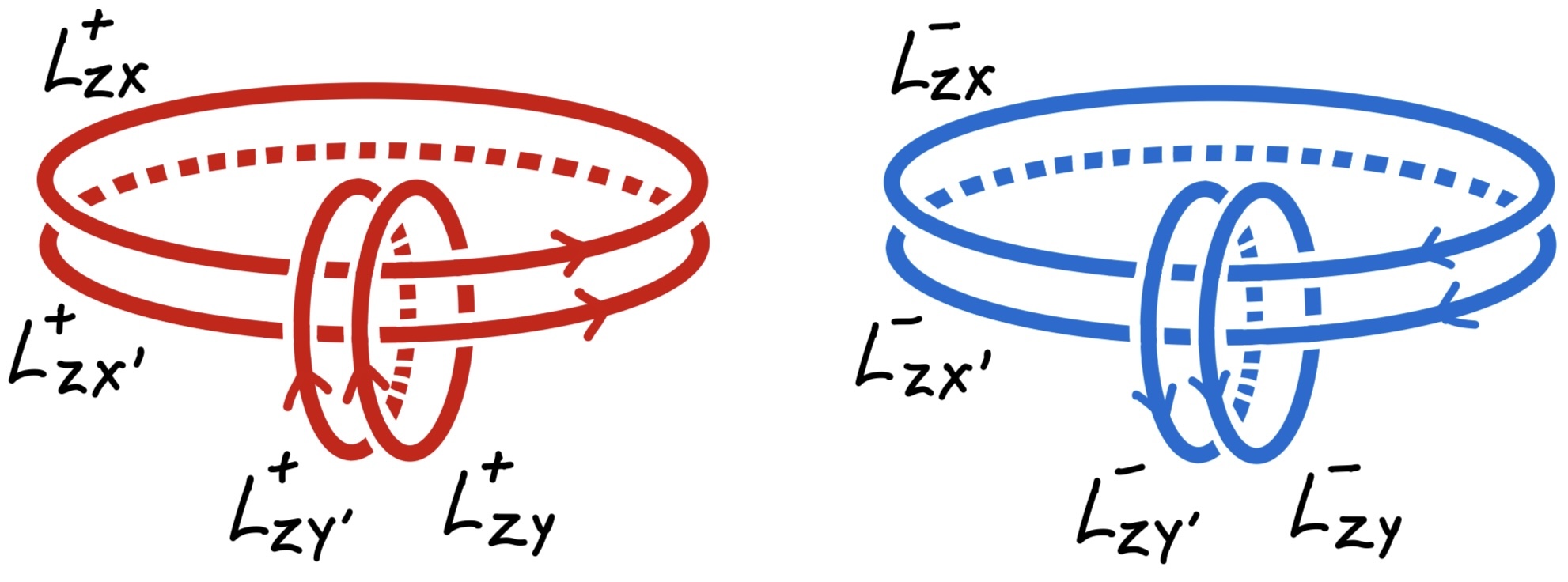}
 \caption{The preimages of the double intersection inside $Z$}
 \label{fig:cablingset}
 \end{figure}

 Thus, by Corollary~\ref{inv2_imm}, $$\mathcal{W}_1(\mathcal{T}\sqcup\mathcal{T}')=E_{X\# Y\# Z}\left(p\circ
(\mathcal{T}\sqcup\mathcal{T}')\right)=\frac 32 
lk_{X\# Y\# Z}(L_{ZX}\sqcup L_{ZY},L_{ZX'}\sqcup L_{ZY'})=-6.$$
 
 Using the additivity of $\mathcal{W}_1$ (=$\mathcal{W}_2)$, we have $$\mathcal{W}_1(C(c\mathcal{T},v))=\mathcal{W}_1(C(c\mathcal{T},\mathbf{0}))+\mathcal{W}_1(C(\mathcal{U},v))=-6c+2m. $$ Since $\mathcal{W}_1(C(c\mathcal{T},v))=0$, it follows that $m=3c$. Therefore, $(c\mathcal{T},v)$ is regularly homotopic to the trivial immersion $\bbR^{4k-1}\looparrowright\bbR^{6k}$ with a nonzero normal vector field $v=3c\cdot[id,id]\in\pi_{4k-1}S^{2k}$. Moreover, $\pi_0Imm_{\partial}^+(\bbR^{4k-1},\bbR^{6k})\overset{\Omega}{=}\pi_{4k-1}V_{6k,4k}=\pi_{4k-1}V_{6k-1,4k-1}$ implies that (rationally) $\Omega_{S^{2k}}$ evaluated on the projection of $(c\mathcal{T},v)$ must be $3c$. Hence, $$\theta(c\mathcal{T},v)=\frac{1}{2}\cdot 0+\frac{1}{6}\cdot 0 +\frac{1}{3}\cdot 3c=c.$$ By the additivity of $\theta$, we get 
 \begin{align*}
     \theta(c\mathcal{T},v) &= \theta(c\mathcal{T},\mathbf{0})+\theta(\mathcal{U},v) \\
     c&=c\cdot 1+\theta(\mathcal{U},v)
 \end{align*}
 which implies that $\theta(\mathcal{U},v)=0$. This completes the proof.
      \end{proof}
\end{subsection}

\begin{subsection}{Compatibility with Sakai's crossing change formula}\label{ss:Sakai}
    
In \cite{SAK1}, Sakai examined the Haefliger invariant for long embeddings $\bbR^{4k-1}\hookrightarrow \bbR^{6k}$ whose projections to $\bbR^{6k-1}$ are generic immersions. He derived a Lin-Wang type formula that captures the difference in the Haefliger invariant under crossing changes. 
 The result builds upon his earlier work in \cite{SAK2}, where he expressed the Haefliger invariant in terms of configuration space integrals. Our Theorem~\ref{knotinv} gives a different proof to Sakai's  \cite[Theorem 2.5]{SAK2} that we recall below.
\begin{thm}[Sakai \cite{SAK1}]
   Let $F:\bbR^{4k-1}\hookrightarrow \bbR^{6k}$ be a long embedding whose projection $p\circ F:\bbR^{4k-1}\looparrowright\bbR^{6k-1}$ 
   is a generic immersion. Let 
    the self-intersection of $p\circ F$ be $A=A_1\sqcup A_2$ (where $A_1$ and $A_2$ can be disconnected or empty), with preimages $(p\circ F)^{-1}(A_i)=L_i=L_i^+\sqcup -L_i^-$, $i=1,2$, oriented using the adjusted Ekholm orientation. Let $F_{A_1}:\bbR^{4k-1}\hookrightarrow \bbR^{6k}$ be obtained from $F$ by the crossing change at $A_1$. Then 
    \begin{equation}\label{eq:cross_change}
     \mathcal{H}(F_{A_1})-\mathcal{H}(F)= \frac{1}{2}lk(L_1,L_2).
    \end{equation}
\end{thm}

\begin{proof}
     Since the expressions~$E$  and~$\Omega$ depend only on the projection, they are unchanged in $\mathcal{H}(F_{A_1})$. Hence the only difference is captured by    interchanging the roles of $L_1^+$ and $L_1^-$ in the linking part of the formula:  
     $\mathcal{H}(F_{A_1})-\mathcal{H}(F) =  \frac{1}{2}lk(L_1^-\sqcup L_2^+,L_1^+\sqcup L_2^-) 
     -\frac{1}{2}lk(L_1^+\sqcup L_2^+,L_1^-\sqcup L_2^-)
 =\frac{1}{2}lk(L_1,L_2).$
\end{proof}

Note that  crossing change can also be
defined for  self-intersection components with Whitney umbrella points. Moreover, the right-hand side of the formula \eqref{eq:cross_change} still makes sense since one uses in it the adjusted Ekholm orientation, see Subsection~\ref{orient}. It is thus natural to conjecture that \eqref{eq:cross_change} holds for
any embedding $F\colon \bbR^{4k-1}\hookrightarrow 
\bbR^{6k}$ with a generic projection on $\bbR^{6k-1}$.

Similar formulas of crossing change can easily be derived for the invariants $\mathcal V$, $\mathcal W_1$ and~$\mathcal W_2$.

  
\end{subsection}

\end{section}

\setcounter{section}{0}
\renewcommand{\thesection}{\Alph{section}}

\begin{section}{Connection to configuration spaces}\label{appen}
The configuration space of $k$ distinct ordered points in $X$ is given by $$C(k,X)\coloneqq\{(x_1,\ldots,x_k)\in X^k| x_i\neq x_j \text{  for  } i\neq j\}.$$
For $\ell\geq3$, consider the space of based smooth maps$$\Omega^nC(k,\bbR^q)=\{D^n\rightarrow C(k,\bbR^q) \text{  constant at a neighborhood of } \partial D^n\}.$$ This space can be interpreted as the space of $n$-dimensional braids in $\bbR^{n+q}$, via the graphing map $$G:\Omega^nC(k,\bbR^q)\rightarrow Emb_{\partial}(\underset{k}{\sqcup}\bbR^n,\bbR^{n+q}),$$ $$f\mapsto G(f),$$ where for each $i=1,\ldots,k$, $$(G(f)_i)(x_1,\ldots,x_n)\coloneqq(x_1,\ldots,x_n,f_i(x_1,\ldots,x_n))\in \bbR^{n+q}, $$ see \cite{BUD2,RAF} for more details.  The induced homomorphism $$G_{*}:\pi_0\Omega^nC(k,\bbR^q)=\pi_nC(k,\bbR^q)\rightarrow \pi_0Emb_{\partial}(\underset{k}{\sqcup}\bbR^n,\bbR^{n+q}),$$ 
 was conjectured in \cite[Conjecture 5.7]{RAF} to be an injection rationally.
 We prove a stronger statement that it
is split-injective, which follows immediately from the following:  

 \begin{thm}\thlabel{retract}
      For $q\geq 3$, $\Omega^nC(k,\bbR^q)$ is a homotopy retract of $Emb_{\partial}(\underset{k}{\sqcup}\bbR^n,\bbR^{n+q})$. 
      \end{thm}

   \begin{proof}
By \cite[Theorem 2.1]{Fred}, there is a homotopy equivalence \begin{equation}\label{confprod}
\Omega C(k,\bbR^q)\rightarrow \underset{i=1}{\overset{k-1}{\prod}}\Omega(\underset{i}{\bigvee} S^{q-1})
\end{equation}
 obtained as follows.
Consider the natural map $C(k,\bbR^q)\rightarrow C(k-1,\bbR^q)$ which forgets the last point in the configuration. This gives rise to a fibration 
$$\underset{k-1}{\bigvee} S^{q-1}\rightarrow C(k,\bbR^q)\rightarrow C(k-1,\bbR^q),$$
with a section obtained by inserting a new point in $\bbR^q$ far away from the others. Because of the section, the induced fibration in loop spaces splits as a 
product implying~\eqref{confprod}:
$$\Omega C(k,\bbR^q)\simeq \Omega C(k-1,\bbR^q)\times \Omega(\underset{k-1}{\bigvee} S^{q-1}).$$
To be precise, one uses the H-space structure on $\Omega C(k,\bbR^q)$ to define the map $$\Omega C(k-1,\bbR^q)\times \Omega(\underset{k-1}{\bigvee} S^{q-1})\xrightarrow{\simeq}\Omega C(k,\bbR^q).$$

From \eqref{confprod}, we get
\begin{equation}\label{c} \Omega^n C(k,\bbR^q)\simeq   \underset{i=1}{\overset{k-1}{\prod}} \Omega^n(\underset{i}{\bigvee} S^{q-1}). \end{equation}
Now consider the  case of long embeddings. There is a fibration of H-spaces
$$Emb_{\partial}(\bbR^n,\bbR^{n+q}\backslash (\underset{k-1}{\sqcup}\bbR^n))\rightarrow Emb_{\partial}(\underset{k}{\sqcup}\bbR^n,\bbR^{n+q})\rightarrow Emb_{\partial}(\underset{k-1}{\sqcup}\bbR^n,\bbR^{n+q}),$$
with a section given by inserting a new, disjoint trivial strand in $\bbR^{n+q}$. By the same argument as above, 
\begin{equation}\label{e} Emb_{\partial}(\underset{k}{\sqcup}\bbR^n,\bbR^{n+q})\simeq  \underset{i=0}{\overset{k-1}{\prod}} Emb_{\partial}(\bbR^n,\bbR^{n+q}\backslash (\underset{i}{\sqcup}\bbR^n)).\end{equation} (We use here that $\pi_0 Emb_{\partial}(\underset{k}{\sqcup}\bbR^n,\bbR^{n+q})$ is an abelian group.)

Each factor in \eqref{c} is a retract of the corresponding factor in \eqref{e}, where for $i=0$, we take a point as a factor in \eqref{c}. More precisely, there is a natural graphing map $$\Omega^n(\underset{i}{\bigvee} S^{q-1})\simeq \Omega^n(\bbR^q\backslash \{i \text{ pts}\})\xrightarrow{G} Emb_{\partial}(\bbR^n,\bbR^{n+q}\backslash (\underset{i}{\sqcup}\bbR^n)).$$ Furthermore, the projection of $\bbR^{n+q}\backslash (\underset{i}{\sqcup}\bbR^n)$ to $\bbR^q\backslash \{i \text{ pts}\}$ gives a map $$Emb_{\partial}(\bbR^n,\bbR^{n+q}\backslash (\underset{i}{\sqcup}\bbR^n))\xrightarrow{F} Map_{\partial}(\bbR^n,\bbR^q\backslash \{i \text{ pts}\})\simeq \Omega^n(\underset{i}{\bigvee} S^{q-1}).$$ It is easy to see that 
$F\circ G= id$ on $\Omega^n(\bbR^q\backslash \{i \text{ pts}\})$, which establishes the claim. 
\end{proof}

The homotopy group $\pi_*C(k,\bbR^q)$ admits a Brunnian splitting, analogous to the one for spaces of embeddings (see Proposition~\ref{split}). In particular, for $q\geq 3$, we have
  \begin{equation}
\pi_{*}C(k,\bbR^q)=\underset{i=1}{\overset{k}{\bigoplus}}\Big(\underset{\binom ki}{\oplus}\pi^{br}_{*}C(i,\bbR^q)\Big),
\end{equation} 
where the brunnian summand is defined by $$\pi^{br}_{*}C(i,\bbR^q)\coloneq\underset{j=1}{\overset{i}{\cap}} ker\big(\pi_*C(i,\bbR^q)\xrightarrow{s_j} \pi_*C(i-1,\bbR^q)\big),$$ with each map $s_j$ forgetting the $j^{th}$ configuration point. Moreover, the graphing map $G_{*}:\pi_nC(k,\bbR^q)\rightarrow \pi_0Emb_{\partial}(\underset{k}{\sqcup}\bbR^n,\bbR^{n+q})$ respects this splitting. The same holds for the higher homotopy groups as well. We can hence conclude:

\begin{proposition}\label{pr:split}
For any $n\geq 1$, $q\geq 3$,  $k\geq 2$, and $j\geq 0$,
$\pi_{n+j}^{br}C(k,\bbR^q)$ is a direct summand 
in $\pi_j^{br} Emb_{\partial}(\underset{k}{\sqcup}\bbR^n,\bbR^{n+q})$.
\end{proposition}

\end{section}

\printbibliography

\bigskip

\hspace{0.5cm}
\footnotesize
\begin{tabular}{@{}l@{}}
  \textsc{Max Planck Institute for Mathematics, Vivatsgasse 7, 53111 Bonn, Germany}\\
  \textit{Email address}: \texttt{neetigauniyal19@gmail.com}\\[0.3cm]
  \textsc{Department of Mathematics, Kansas State University, Manhattan, KS 66506, USA}\\
  \textit{Email address}: \texttt{turchin@ksu.edu}
\end{tabular}
\end{document}